\documentclass[12pt,a4paper]{article}
\usepackage{amsmath,amssymb,amsthm,graphicx,color,bbm}
\usepackage[left=1in,right=1in,top=1in,bottom=1in]{geometry}
\usepackage{cite}
\usepackage[british]{babel}
\usepackage{enumerate}
\usepackage{hyperref} 
\usepackage{stmaryrd}
\usepackage{comment}
\usepackage{tikz}
\usepackage{pgfplots}
\usetikzlibrary{calc}

\pgfplotsset{ticks=none}
\pgfplotsset{compat=1.16}

\hypersetup{
    colorlinks=false,
    pdfborder={0 0 0},
}

\newtheorem{theorem}{Theorem}[section]
\newtheorem{corollary}[theorem]{Corollary}
\newtheorem{proposition}[theorem]{Proposition}
\newtheorem{lemma}[theorem]{Lemma}

\numberwithin{equation}{section}

\theoremstyle{definition}

\theoremstyle{remark}
\newtheorem{remark}[theorem]{Remark}

\newtheorem*{remark*}{Remark}

 \allowdisplaybreaks

\newcommand{\bbind}{{\mathbbm 1}\!}

\newcommand{\1}[1]{{\mathbbm 1}{\{#1\}}}
\newcommand{\2}[1]{{\mathbbm 1}_{#1}}

\definecolor{lightgrey}{rgb}{0.83, 0.83, 0.83}

\newcommand{\R}{{\mathbb R}}
\newcommand{\Z}{{\mathbb Z}}

\newcommand{\N}{{\mathbb N}}

\newcommand{\ZP}{{\mathbb Z}_+}
\newcommand{\RP}{{\mathbb R}_+}

\newcommand{\bbA}{{\mathbb A}}
\newcommand{\bbI}{{\mathbb I}}
\newcommand{\bbX}{{\mathbb X}}

\DeclareMathOperator{\Exp}{\mathbb{E}}
\renewcommand{\Pr}{{\mathbb P}}
\DeclareMathOperator{\Prl}{\Pr_{\lambda}}
\DeclareMathOperator{\Prll}{\Pr_{\ell,\lambda}}
\DeclareMathOperator{\Prnz}{\Pr^{\text{$r$},\Phi}_{\text{$n_0$,$z$}}}
\DeclareMathOperator{\Expl}{\Exp_{\lambda}}
\DeclareMathOperator{\Expll}{\Exp_{\ell,\lambda}}

\DeclareMathOperator{\sech}{sech}

\newcommand{\tra}{{\scalebox{0.6}{$\top$}}}

\newcommand{\eps}{\varepsilon}
\newcommand{\tl}{\tilde\ell}

\newcommand{\re}{{\mathrm{e}}}
\newcommand{\rc}{{\mathrm{c}}}
\newcommand{\ud}{{\mathrm d}}

\newcommand{\cA}{{\mathcal A}}
\newcommand{\cB}{{\mathcal B}}
\newcommand{\cC}{{\mathcal C}}
\newcommand{\cD}{{\mathcal D}}
\newcommand{\cE}{{\mathcal E}}
\newcommand{\cF}{{\mathcal F}}
\newcommand{\cG}{{\mathcal G}}

\newcommand{\cI}{{\mathcal I}}
\newcommand{\cJ}{{\mathcal J}}

\newcommand{\cP}{{\mathcal P}}

\newcommand{\cS}{{\mathcal S}}

\newcommand{\cX}{{\mathcal X}}
\newcommand{\cY}{{\mathcal Y}}
\newcommand{\cZ}{{\mathcal Z}}

\newcommand{\tcA}{{\tilde {\mathcal A}}}
\newcommand{\tcF}{{\tilde {\mathcal F}}}
\newcommand{\tcI}{{\tilde {\mathcal I}}}
\newcommand{\tx}{{\tilde x}}
\newcommand{\tA}{{\tilde A}}
\newcommand{\tZ}{{\tilde Z}}

\newcommand{\as}{\ \text{a.s.}}

\newcommand{\bigmid}{\; \bigl| \;}
\newcommand{\Bigmid}{\; \Bigl| \;}
\newcommand{\biggmid}{\; \biggl| \;}

\newcommand{\eqd}{\overset{d}{=}}

\title{Deposition, diffusion, and nucleation on an interval}
\author{Nicholas Georgiou\footnote{Department of Mathematical Sciences, Durham University, Upper Mountjoy Campus,
Stockton Road, Durham DH1 3LE.} \and Andrew R.\ Wade\footnotemark[1]} 
\date{\today}

\begin{document}

\maketitle

\begin{abstract}
Motivated by nanoscale growth of ultra-thin films,
we study a model of deposition, on an interval substrate, of particles
that perform Brownian motions until any two meet, when they nucleate to form a static island, which acts
as an absorbing barrier to subsequent particles. This is a continuum version of a lattice model 
studied in the applied literature.
We show that the associated interval-splitting process  converges
in the sparse deposition limit to a Markovian process (in the vein of Brennan and Durrett) governed by a 
splitting density with a compact Fourier series expansion but, apparently, no simple closed form. We show that the same splitting density governs the fixed deposition rate, large time asymptotics of the normalized gap distribution, so these asymptotics are independent of deposition rate. The splitting density is derived by solving an exit problem for planar Brownian motion from a right-angled triangle, extending work of Smith and Watson.
\end{abstract}

\medskip


\medskip


\section{Introduction}
\label{sec:intro}

Surface phenomena are important in chemistry, physics, and materials science.
Our probabilistic model originates with the growth of ultra-thin films.
The non-equilibrium dynamics of these self-organized growth processes are 
of central importance in understanding the construction of nanomaterials
by deposition of monomers onto a solid substrate.
The materials involved may be crystals, metals, or semiconductors, for example,
deposition may be via vapour, chemical methods, or cathodic sputtering, and 
surface binding may be chemical (chemisorption) or physical (physisorption). 
In certain contexts, thin film growth is known as `epitaxy'.
Nanoscale growth is important in the development of many technological 
devices reliant on the remarkable electrical, optical and thermal properties of thin films,
and developments in construction of nanomaterials and in atomic-scale experimental observation
have fuelled interest over the last couple of decades. We refer to~\cite{aw,bs,bv,edm,pv,vsh,venables}
for scientific background and technological applications.

Under certain energetic conditions, the early stages of submonolayer growth 
are described by so-called Volmer--Weber dynamics.
Particles are deposited onto a substrate and undergo diffusion until
sufficiently many particles come into close proximity, when they
`nucleate' to form static islands, 
which form absorbing barriers with respect to the diffusion of other
particles.
The   nucleation threshold (i.e., the number of particles that must come together to nucleate)
increases with temperature.
As time goes on, more islands form by nucleation, and these islands grow
by the accumulation of additional diffusing particles. 
Eventually,
as coverage increases, 
monomers will tend to aggregate on growing islands rather than
initiate new islands, and these growing islands will coalesce into larger structures. 
Many interesting  aspects of these dynamics 
 are discussed e.g.~in~\cite{bv,be,bgm,bm95b,bm96,bw,brune,eb,gpe,glom,mul2009,mogl,oglm,pe} and references therein. 

In the present paper we study a one-dimensional model on an interval substrate which is both space- and time-continuous,
in which two particles suffice for nucleation (`binary nucleation').
Ours is a continuum relative of a type of lattice model that has been widely
used in the applied literature, e.g.~by Bartelt \& Evans~\cite{be}
and by Blackman \& Mulheran~\cite{bm96}, for both simulation and theoretical investigations;
see also~\cite{mpa} for a related early Monte Carlo study.
The model neglects both the spatial extent of islands (this `point island' assumption
is reasonable at low coverage) and also any potential evaporation of particles.

Informally, the model is as follows. At time $t=0$,
there are no active particles and the initial island locations are $\{ 0,1\}$,
the endpoints of the interval.
\begin{itemize}
\item \emph{Deposition.} Particles are deposited on $[0,1]$ according to a space-time Poisson process
on $[0,1] \times \RP$ with intensity $\lambda > 0$.
\item \emph{Diffusion and nucleation.} Each deposited particle performs an independent Brownian motion
until it either (i) 
hits an existing island, or (ii) meets another diffusing particle. In case (i), the particle is absorbed by the island. In case (ii), we say that
nucleation has occurred, and a new island is formed at the collision site. In either case, the particle's position becomes fixed for all subsequent time.
\end{itemize}

In \S\ref{sec:results} we present our main results for the nucleation process,
which can be understood with the informal definition of the process given above. The first (Theorem~\ref{thm:small-lambda}) is a description of the $\lambda \to 0$ limit as a particular Markovian interval-splitting process, characterized in part by a splitting density $\phi_0$ on $[0,1]$.
In contrast to previous applied work, which proposed various Beta distributions in this role, our $\phi_0$ does not seem to have a simple expression in terms of elementary functions. Our second main result (Theorem~\ref{thm:gap-statistics}) treats long-time statistics of the fixed-$\lambda$ process, in particular, the normalized gap distribution. It turns out that the large-time statistics of the 
 fixed-$\lambda$ process can be described via the $\lambda \to 0$ density $\phi_0$, and so, in particular, the limits are independent of $\lambda$. In \S\ref{sec:discussion} we make some comparisons with previous work (which mostly lies 
outside the probability literature) and comment on
 possible extensions. A formal construction of our process is presented in~\S\ref{sec:construction}, along with some fundamental initial observations. The key ingredient in our limit theorems is a quantitative approximation of the evolution of our process via an interval-splitting kernel; this is derived in~\S\ref{sec:splitting}. This approximation is then
 used to derive our $\lambda \to 0$ results (in~\S\ref{sec:sparse}) and our fixed-$\lambda$, long-time results (in~\S\ref{sec:fixed-rate-regime}). The splitting kernel requires evaluation of the density $\phi_0$, which we reduce to a problem of the exit position of planar Brownian motion from a right-angled triangle, started from an arbitrary interior point: the solution to this problem, which extends old work of Smith \& Watson~\cite{sw}, is presented in~\S\ref{sec:triangle}. In~\S\ref{sec:numerics} we collect necessary analytic properties of the splitting density~$\phi_0$, as well as some numerical approximations. Finally, in~\S\ref{sec:gap-tails} we apply results of Brennan \& Durrett~\cite{bd1,bd2} to derive normalized gap-distribution statistics for interval-splitting processes; this forms an ingredient to our results but is presented in some generality so as to facilitate comparison with the various other interval-splitting parameters that have been proposed in the literature for related nucleation problems.
 
 We mention briefly that there has been much recent interest in the probability literature in systems of interacting diffusing particles: see e.g.~\cite{bgs,crtz,ss}. Several of these models include deposition or particle birth, and coalescence of diffusing particles, but coalescing particles continue to diffuse, rather than nucleate.

\section{Main results}
\label{sec:results}

We are interested in the interval fragmentation process induced by our model. 
We defer a formal construction of the model (based on a marked Poisson point process)
to~\S\ref{sec:construction} below. 
Let $I_t$ denote the number of interior islands at time $t \in \RP := [0,\infty)$, so $I_0 = 0$.
Set $\nu_0 := 0$, and for $n \in \N := \{1,2,\ldots\}$ denote the time of the $n$th nucleation by
\begin{equation}
\label{eq:tau-def}
 \nu_n := \inf \{ t \in \RP : I_t = n \} ;\end{equation}
throughout the paper, we adopt the usual convention that~$\inf \emptyset := +\infty$.
The proof of the following fact will be given in~\S\ref{sec:construction}.

\begin{lemma}
\label{lem:tau-nice}
For all $\lambda >0$, $\nu_n < \infty$ a.s.~for all $n \in \N$, and $\lim_{n \to \infty} \nu_n = \infty$, a.s.
\end{lemma}

Let $\cZ_n$ denote the vector of island locations in $[0,1]$, listed left to right, at time $\nu_n$,
so $\cZ_n \in \Delta_n$ 
where
\[ \Delta_n := \left\{ ( z_0, z_1 , \ldots, z_{n+1} ) \in [0,1]^{n+2} : 0 = z_{0} < z_{1} < \cdots < z_{n} < z_{n+1} = 1 \right\}. \]

Consider the process $\cZ := (\cZ_0, \cZ_1, \cZ_2, \ldots)$. At time $\nu_n$,
the law of $\cZ_{n+1}$ is not determined by $\cZ_n$ alone,
since there
may still be active particles in the system.
However,   our first main result (Theorem~\ref{thm:small-lambda})
shows that as 
$\lambda \to 0$, the process $\cZ$ converges
to a Markovian interval-splitting process. 
We next describe the limiting process.

Let $\cB$ denote the Borel subsets of $[0,1]$, and for $n \in \N$ set $[n] := \{1,2,\ldots,n\}$. 
Take a function $r : [0,1] \to \RP$
and a probability measure $\Phi$ on $([0,1], \cB)$.
Assume that $r(\ell) >0$ for all $\ell >0$, and $\Phi (\{0\}) = \Phi(\{1\}) =0$. Then define for each $n \in \ZP := \{ 0 \} \cup \N$ a splitting map $\Gamma_n  : \Delta_n \times [n+1] \times (0,1) \to \Delta_{n+1}$
by
\begin{equation}
    \label{eq:splitting-map}
\bigl( \Gamma_n (z ; j,v) \bigr)_i := \begin{cases} z_i & \text{if } i < j, \\
z_{j-1} + v (z_j - z_{j-1} ) & \text{if } i =j,\\
z_{i-1} &\text{if } i > j,\end{cases} \end{equation}
for $z = (z_0,z_1,\ldots,z_{n+1}) \in \Delta_n$, $j \in [n+1]$, and $v \in (0,1)$. 
We say that the process $\cS := ( \cS_0, \cS_1, \cS_2, \ldots )$, with $\cS_n= (S_{n,0},S_{n,1},\ldots,S_{n,n+1}) \in \Delta_n$ for all $n$,
is an \emph{interval-splitting process with parameters $r$ and $\Phi$},
if, for all $n \in \ZP$, $j \in [n+1]$, and  $B \in \cB$, 
\begin{equation} 
\label{eq:interval-splitting}
\Pr ( \cS_{n+1} \in \Gamma_n (\cS_n ; j , B) \mid \cS_0, \cS_1, \ldots, \cS_n )
= \frac{r( S_{n,j} - S_{n,j-1} )}{\sum_{i \in [n+1]} r( S_{n,i} - S_{n,i-1} )} \Phi (B) , \as \end{equation}
The sequence of kernels~\eqref{eq:interval-splitting} and the initial value $\cS_0 = (0,1) \in \Delta_0$ determine
the finite-dimensional distributions of $\cS$, and hence the law
of $\cS$ as a random element of the product space $\Delta_0 \times \Delta_1 \times \cdots$ with the usual (Borel) product topology.
In words, the transition from $\cS_n$ to $\cS_{n+1}$ is achieved by 
choosing the interval to be split randomly
with probabilities proportional to the function $r$ of each interval length,
and the chosen interval is split into two by choosing
a point in the interval according to the distribution~$\Phi$.
Interval-splitting processes in this generality were
studied by Brennan \& Durrett~\cite{bd1,bd2}. 

Our $\lambda \to 0$ limit
of $\cZ$ turns out to be an interval-splitting process with a particular $r$ 
and $\Phi$. To describe the $\Phi$ that arises in our limit, we need some more notation.
 Define 
\begin{equation}
\label{eq:psi-def}
 \psi (z) := \frac{24}{\pi^4} \sum_{n \text{ odd}} a_n \sin n \pi z , \text{ where }
a_n := \frac{4}{n^4} \tanh \left( \frac{n \pi}{2} \right) - \frac{\pi}{n^3} ;
\end{equation}
where `$n$ odd' means $n \in \{1,3,5,\ldots\}$. Note that $a_1 > 0$, but $a_n < 0$ for $n \geq 3$. 
In~\S\ref{sec:numerics} we will use a representation of $\psi$ involving
a special function related to the \emph{Clausen function} to show that $\psi$ is twice continuously differentiable on $[0,1]$, to give a more rapidly converging series approximation,
and to show that $\psi (z) \sim 3 z^2$ as $z \to 0$, a property that has important consequences
for some of our results, but which is well-hidden in the series representation of~\eqref{eq:psi-def}.
The probabilistic meaning of $\psi$ is as a (defective) density arising from an exit problem
for Brownian motion in a right-angled triangle: see~\S\ref{sec:triangle}. In particular, although not obvious from~\eqref{eq:psi-def}, $\psi (z) >0$ for all $z \in (0,1)$.
Also set 
\begin{equation}
\label{eq:mu-def}
 \mu := \int_0^1 \psi (z) \ud z = \frac{48}{\pi^5} \sum_{n \text{ odd}}  \frac{a_n}{n} = \frac{48}{\pi^4} \sum_{n \text{ odd}}  \frac{\sech^2 \left( \frac{n \pi}{2} \right) }{n^4};
\end{equation}
the first series  follows directly from~\eqref{eq:psi-def}, while the second is established in~\S\ref{sec:numerics}.
The second series representation in~\eqref{eq:mu-def} is useful for numerical evaluation of $\mu$, because
$\sech^2 \left( \frac{n \pi}{2} \right)$ decays exponentially in $n$.
Indeed, taking only the terms $n=1, 3$ in the final sum in~\eqref{eq:mu-def}
suffices to evaluate the first 8 decimal digits of $\mu \approx
0.07826895$ (see~\S\ref{sec:numerics} for a justification). 

Let $\phi_0$ be $\psi$ normalized to be a probability density,
and let $\Phi_0$ be the corresponding probability measure, i.e.,
\begin{equation}
\label{eq:Phi0-def}
\Phi_0 ( B ) := \int_B \phi_0 (z) \ud z := \frac{1}{\mu} \int_B \psi (z) \ud z, \text{ for } B \in \cB .\end{equation}
See Figure~\ref{fig:phi0-density} for an illustration of a numerical approximation to $\phi_0$,
and see \S\ref{sec:numerics} for a discussion of the numerics.
We can now state our first main result.
 
\begin{theorem}
\label{thm:small-lambda}
As $\lambda \to 0$, the process $\cZ$ converges, in the sense
of total-variation convergence of finite-dimensional distributions, to an interval-splitting process
with parameters $r_0$ and $\Phi_0$, where $r_0 (\ell) = \ell^4$ and~$\Phi_0$ is given by~\eqref{eq:Phi0-def}.
\end{theorem}

\begin{figure}
\centering
 \begin{tikzpicture}[domain=-0.5:1.5, scale = 5.5]
\node at (0,-0.1)       {$0$};
\node at (1,-0.1)       {$1$};
\draw[black] (0,0) -- (0,-0.04);
\draw[black,->] (0,0) -- (0,1.2);
\draw[black] (1,0) -- (1,-0.04);
\draw[black] (0,0.6*1.82692513684280) -- (-0.04,0.6*1.82692513684280);
\draw[black] (0,0) -- (-0.04,0);
\node at (-0.14,0.6*1.82692513684280) {$1.827$};
\node at (-0.08,0) {$0$};
\filldraw[draw=black,fill=lightgrey] (0.00,0.00) rectangle ++(0.04,15*0.000988);
\filldraw[draw=black,fill=lightgrey] (0.04,0.00) rectangle ++(0.04,15*0.005063);
\filldraw[draw=black,fill=lightgrey] (0.08,0.00) rectangle ++(0.04,15*0.011386);
\filldraw[draw=black,fill=lightgrey] (0.12,0.00) rectangle ++(0.04,15*0.019214);
\filldraw[draw=black,fill=lightgrey] (0.16,0.00) rectangle ++(0.04,15*0.028017);
\filldraw[draw=black,fill=lightgrey] (0.20,0.00) rectangle ++(0.04,15*0.036535);
\filldraw[draw=black,fill=lightgrey] (0.24,0.00) rectangle ++(0.04,15*0.044926);
\filldraw[draw=black,fill=lightgrey] (0.28,0.00) rectangle ++(0.04,15*0.052998);
\filldraw[draw=black,fill=lightgrey] (0.32,0.00) rectangle ++(0.04,15*0.059303);
\filldraw[draw=black,fill=lightgrey] (0.36,0.00) rectangle ++(0.04,15*0.064680);
\filldraw[draw=black,fill=lightgrey] (0.40,0.00) rectangle ++(0.04,15*0.068925);
\filldraw[draw=black,fill=lightgrey] (0.44,0.00) rectangle ++(0.04,15*0.071568);
\filldraw[draw=black,fill=lightgrey] (0.48,0.00) rectangle ++(0.04,15*0.072405);
\filldraw[draw=black,fill=lightgrey] (0.52,0.00) rectangle ++(0.04,15*0.071454);
\filldraw[draw=black,fill=lightgrey] (0.56,0.00) rectangle ++(0.04,15*0.069106);
\filldraw[draw=black,fill=lightgrey] (0.60,0.00) rectangle ++(0.04,15*0.064989);
\filldraw[draw=black,fill=lightgrey] (0.64,0.00) rectangle ++(0.04,15*0.059573);
\filldraw[draw=black,fill=lightgrey] (0.68,0.00) rectangle ++(0.04,15*0.052954);
\filldraw[draw=black,fill=lightgrey] (0.72,0.00) rectangle ++(0.04,15*0.045077);
\filldraw[draw=black,fill=lightgrey] (0.76,0.00) rectangle ++(0.04,15*0.036793);
\filldraw[draw=black,fill=lightgrey] (0.80,0.00) rectangle ++(0.04,15*0.027521);
\filldraw[draw=black,fill=lightgrey] (0.84,0.00) rectangle ++(0.04,15*0.019060);
\filldraw[draw=black,fill=lightgrey] (0.88,0.00) rectangle ++(0.04,15*0.011410);
\filldraw[draw=black,fill=lightgrey] (0.92,0.00) rectangle ++(0.04,15*0.005045);
\filldraw[draw=black,fill=lightgrey] (0.96,0.00) rectangle ++(0.04,15*0.001010);
\draw[black, line width = 0.30mm]   plot[smooth,domain=0.00001:0.5,samples=50] ({\x},  {(0.6/0.03913447756726942)*(1.62826354332251*\x+1.27323954473516*\x^3*ln(pi*\x)-3.21681489929174*\x^3-(3/2)*\x*(1-\x)+0.104719755119660*\x^5+0.0172257092668332*\x^7+0.00497953421744756*\x^9+0.00192188371777121*\x^11+0.000889522103375420*\x^13+0.000466275098856422*\x^15+0.000267471028154234*\x^17+0.000164244084872429*\x^19+0.000106368812947367*\x^21+0.0000718937510357596*\x^23-0.0408246072903266*sin(pi*\x r)-9.81799068946915*10^(-7)*sin(3*pi*\x r)-2.37634718097513*10^(-10)*sin(5*pi*\x r))});
\draw[black, line width = 0.30mm]   plot[smooth,domain=0.00001:0.5,samples=50] ({1-\x},  {(0.6/0.03913447756726942)*(1.62826354332251*\x+1.27323954473516*\x^3*ln(pi*\x)-3.21681489929174*\x^3-(3/2)*\x*(1-\x)+0.104719755119660*\x^5+0.0172257092668332*\x^7+0.00497953421744756*\x^9+0.00192188371777121*\x^11+0.000889522103375420*\x^13+0.000466275098856422*\x^15+0.000267471028154234*\x^17+0.000164244084872429*\x^19+0.000106368812947367*\x^21+0.0000718937510357596*\x^23-0.0408246072903266*sin(pi*\x r)-9.81799068946915*10^(-7)*sin(3*pi*\x r)-2.37634718097513*10^(-10)*sin(5*pi*\x r))});
\end{tikzpicture}
\caption{The smooth curve is a numerical estimate of the density $\phi_0$
using the approximant $\phi_0^{k,m}$ for $k=9$, $m=5$
(see~\S\ref{sec:numerics} for a definition) which is accurate to within $10^{-10}$ for all $x \in [0,1]$.
The histogram is a simulation estimate for the
distribution of the location of the first nucleation
at $\lambda =0.1$,
based on $10^6$ samples of a discrete version of the model
on the lattice $\{0,\frac{1}{100}, \frac{2}{100}, \ldots, 1\}$,
in which any active particle performs continuous-time simple random walk at rate $100^2$,
and the Poisson deposition rate at each site is $\lambda/100$.
}
\label{fig:phi0-density}
\end{figure}
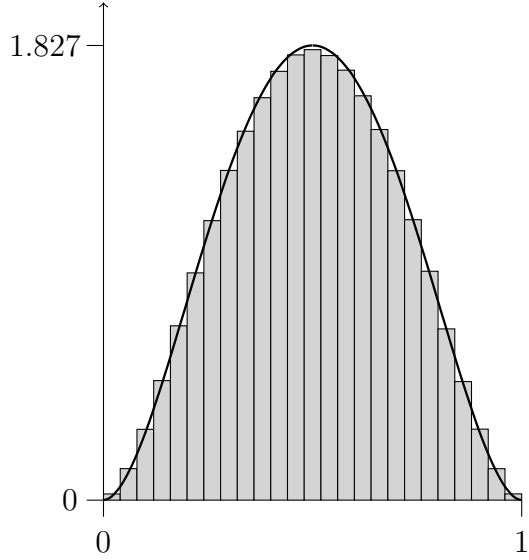

When $\lambda \in (0,\infty)$ is fixed, there is not such a neat description of the interval-splitting process. However,
after a long time, when all intervals become very small, scaling arguments show that diffusion again dominates deposition (we give details below). Roughly speaking, this means that certain large-time statistics of the
fixed-$\lambda$ process can be asymptotically described in terms of the $\lambda \to 0$ limit given in Theorem~\ref{thm:small-lambda}. To state the result, we need to introduce some notation for the statistics that we wish to consider.

Let $(L_{n,1}, L_{n,2}, \ldots, L_{n,n+1} )$ denote the gap 
lengths associated with $\cZ_n$, so if $\cZ_n = (Z_{n,0}, Z_{n,1}, \ldots, Z_{n,n+1} ) \in \Delta_n$, then
\[ L_{n,i} :=  Z_{n,i} - Z_{n,i-1}, \text{ for } 1 \leq i \leq n+1 .\]
For $x \in [0,1]$, denote the number of interior islands in $[0,x]$ after $n$ nucleations by
\[ N_n ( x) := \max \{ i \in \{0,1,\ldots,n\} : Z_{n,i} \leq x \} ;\]
the total number of interior islands is $N_n(1) = I_{\nu_n} = n$. For $U_n$ uniform on $[n+1]$, set
\[ \tilde L_n := \frac{L_{n,U_n}}{\Exp ( L_{n,U_n} )} = (n+1) L_{n,U_n} ,\]
the length of a randomly chosen gap, normalized to have unit mean. Denote the
empirical \emph{gap size distribution}, also normalized, by
\[ \cE_n (x) :=  \frac{1}{n+1} \sum_{i=1}^{n+1} \bbind \left\{  L_{n,i}  \leq \frac{x}{n+1} \right\} , \text{ for } x \in \RP.\]
Here is our main result in the case of fixed $\lambda$. Note that,
for the reasons previously indicated, the limit distributions do not depend on $\lambda$.
For a positive function $g$, we write $f(x) \sim g(x)$ to mean the ratio $f(x)/g(x)$ tends to 1.

\begin{theorem}
\label{thm:gap-statistics}
Let $\lambda \in (0,\infty)$. 
\begin{itemize}
\item[(i)] We have that $\lim_{n \to \infty} \sup_{x \in [0,1]} | n^{-1} N_n (x) - x | = 0$, a.s.
\item[(ii)] There exists a continuous probability density function $g_0$ on $\RP$,
 which can be described in terms of $r_0$ and $\phi_0$
appearing in Theorem~\ref{thm:small-lambda}, such that, for all $x \in \RP$,
\[ \lim_{n \to \infty} \Pr ( \tilde L_n \leq x ) = \int_0^x g_0 (y) \ud y, \text{ and, a.s., }
\lim_{n \to \infty}  \cE_n(x) = \int_0^x g_0 (y) \ud y  . \]
Moreover, there exist constants $c_{g,0}, c_{g,\infty}, \theta \in (0,\infty)$ such that
\begin{align*} g_0 (x)  \sim c_{g,0}\, x^2, \text{ as $x \to 0$, and }
g_0 (x) \sim  \frac{c_{g,\infty}}{x^2} \exp ( - \theta x^4), \text{ as } x \to \infty .\end{align*}
\end{itemize}
\end{theorem}

We do not have an explicit expression for $g_0$, but  $g_0$ can be characterized in terms of $r_0$ and $\phi_0$ via a distributional fixed-point equation
 derived in~\cite{bd2}: see~\S\ref{sec:gap-tails}. In~\S\ref{sec:gap-tails} we demonstrate, in a 
context of
more general  interval-splitting processes, 
the key properties of $r_0$ and $\phi_0$ that lead to the tail asymptotics for $g_0$ stated here.

\section{Discussion}
\label{sec:discussion}

A Web of Science topic search in May 2020 for ``epitaxy'' produces over 90,000 titles, covering articles in chemistry, physics, materials science,
and so on (for comparison, ``percolation'' produces about half that number). 
While, as far as we are aware, our continuum model does not seem to have been considered before,
 closely related discrete models
have   generated significant interest, and have been
studied both via simulations and various interesting, but not fully rigorous, analytical approaches 
 (see e.g.~\cite{be,bm96,glom,mogl,pe}). 
Our model corresponds to a specific case of the models of submonolayer deposition considered in~\cite{bm96} and elsewhere:
here we focus on one dimension, on binary nucleation, and on regimes where active particles are sparse.

It is natural to seek to extend our model in the following four important ways.
\begin{itemize}
\item[(a)] Take the nucleation threshold to be an integer $\alpha \geq 2$ (our case is $\alpha=2$).
\item[(b)] Allow the deposition rate $\lambda$ to depend on time or on the current number of islands, with $\lambda \to \infty$.
\item[(c)] Consider substrates in higher dimensions, so that, for example, monomers live in $[0,1]^d$, $d \in \N$ (the case $d=2$ being the most physically relevant).
\item[(d)] Permit islands to have spatial extent as an increasing function of the number of monomers that they have captured.
\end{itemize}

In discrete models, as $\alpha$ increases nucleations become 
much rarer, and quantitative differences are 
predicted by existing theory (e.g.~\cite{glom,oglm}).
In the continuum context, due to the 
impossibility of multiple simultaneous Brownian collisions, 
a meaningful model with $\alpha \geq 3$ in one dimension (or  $\alpha \geq 2$ in dimension $d \geq 2$) 
seems to require an addition of an interaction radius $\delta >0$ for particles. Thus addressing~(a) and/or~(c) may simultaneously require dealing with~(d). 

We raise point~(b) because a key feature of the analysis in the present paper is that the density of active particles is low,
 and tends to zero as time goes on. 
On the other hand, much existing work is concerned with regimes in which, at a typical time, there are many active particles in the system,
and the statistics of the system are driven by a `quasiequilibrium' between particle deposition and capture by islands~\cite{bm95b,bm96,gpe}. 
Both regimes are potentially relevant for physical applications~\cite[\S 11.2]{pv}.
While it seems likely that the results of the present paper could be extended to allow $\lambda$ to grow slowly with time,
the methods used here will not fully extend to the case where the average density of active particles remains bounded above zero.
Suitable models with any/all of the features (a)--(d) provide
much scope for probabilistic investigation.

We discuss some specific points of comparison between our results and earlier work. 
 For their model,
Blackman \& Mulheran~\cite[\S V]{bm96} consider
analogues of the parameters~$r_0$ and~$\Phi_0$ in our Theorem~\ref{thm:small-lambda}, and
argue that
\begin{itemize}
\item their analogue of $r_0(\ell)$ scales as $\ell^5$, rather than our $\ell^4$;
\item their analogue of $\Phi_0$ is the Beta$(3,3)$ distribution, which has density  proportional to $z^2 (1-z)^2$ over $z \in [0,1]$; this approaches zero as $z^2$, like our $\phi_0$.
\end{itemize}
O'Neill \emph{et al.}~describe arguments for both $\ell^3$ and $\ell^5$ scaling for the splitting exponent, and report
simulation estimates that fall between the two~\cite{oglm}. As mentioned above, the  arguments in~\cite{bm96,oglm} have many active particles in the system when a nucleation happens, so their results
are not necessarily comparable with ours. 

Statistics of the (normalized) gap distribution, such as studied in our Theorem~\ref{thm:gap-statistics}, have received a lot of attention, along with the closely-related
\emph{capture-zone distributions}, i.e., the sizes of the Voronoi intervals associated with the islands~\cite{pe,bm96,glom,mogl}. 
Stretched Gamma distributions of the form $g(x) \approx x^{\theta_1} \exp( -c x^{\theta_2} )$ have been considered
(sometimes called the \emph{generalized Wigner surmise}~\cite{pe}), but it has since been accepted that such distributions
do not capture simultaneously the $x \to 0$ and $x \to \infty$ asymptotics. For example, 
Blackman \& Mulheran~\cite{bm96} argue that, in the regime they are considering,
 the asymptotic density should look like
\begin{equation}
\label{eq:bm-gaps} g(x) \approx x^2, \text{ as } x \to 0, \text{ and } g( x) \approx \frac{1}{x^2} \exp ( - \theta x^5 ) , \text{ as } x \to \infty .\end{equation}
The predictions of~\eqref{eq:bm-gaps} are reproduced by a fragmentation approximation~\cite{glom}, while an alternative approach based
on distributional fixed-point equations apparently reproduces the asymptotics in~\eqref{eq:bm-gaps} at $0$ but not at $\infty$~\cite[\S III]{mogl}.
The exponent~$5$ in~\eqref{eq:bm-gaps} comes from Blackman \& Mulheran's predicted splitting exponent.
In Theorem~\ref{thm:general-splitting-distribution} we give a general result deriving tail asymptotics for the normalized gap distribution
in general interval-splitting processes, providing a range of asymptotics like~\eqref{eq:bm-gaps}.

\section{Construction, regeneration, and scaling}
\label{sec:construction}

It is convenient to generalize our model so that the substrate is $[0,\ell]$ for $\ell \in (0,\infty)$.
Let $\cC := \cC (\RP , \R)$,
the collection of all continuous functions from~$\RP$ to~$\R$, and let 
$\cC_0 := \{ f \in \cC : f (0) = 0\}$.
Let
$W$
denote the standard Wiener (probability) measure on~$\cC_0$,
so that $W$ is the law of standard Brownian motion on $\R$ started at the origin.

We  build our process from $\cP_{\ell,\lambda}$, a homogeneous Poisson point process of intensity $\lambda > 0$ on $[0,\ell] \times \RP$,
where each Poisson point carries an independent $\cC_0$-valued random mark distributed according to~$W$.
With probability one, all the $\RP$-coordinates of the process are distinct, and then 
we may (and do) list the points of $\cP_{\ell,\lambda}$ in order of increasing $\RP$-coordinate
as~$\Xi_1, \Xi_2, \ldots$ with $\Xi_i = ( \xi_i, s_i, b_i )$, where
$\xi_i \in [0,\ell]$,  
$b_i = (b_i(r), r \in \RP) \in \cC_0$, and
 $0 < s_1 < s_2 < \ldots$. We interpret $s_i$ as the time of deposition of the $i$th particle,
which arrives at location $\xi_i \in [0,\ell]$.
Set
\begin{equation}
\label{eq:x-def}
x_i(r) := \begin{cases} \partial &\text{if } 0 \leq r < s_i , \\
\xi_i + b_i(r-s_i) &\text{if } r \geq s_i , 
\end{cases}
\end{equation}
where $x_i (r) = \partial$ is to be interpreted
as particle~$i$ having not yet arrived by time $r$,
and $x_i(r) \in \R$ is the position of the $i$th particle at time $r \geq s_i$, ignoring interactions.
Let $\bar \R := \R \cup \{ \partial \}$.

Let $\bbI$ be the set of all finite subsets of $[0,\ell]$
(the set of possible island locations), let
$\bbA$ denote the set of all finite (or empty) subsets of $\N$ (possible
labels of active particles),
and let $\bbX := \bar \R^\N$
(locations of the particles, neglecting interactions). 
From the marked Poisson process~$\cP_{\ell,\lambda}$, we will construct
the process $\cY := (\cY_t, t \in \RP)$
where $\cY_t = ( \cI_t, \cA_t, \cX_t)$
with $\cI_t \in \bbI$, $\cA_t \in \bbA$, and $\cX_t := ( x_1(t), x_2(t), \ldots ) \in \bbX$.
The system described informally in~\S\ref{sec:intro}
is captured by $\cI_t$, the locations of the interior islands,
and $( x_j (t), j \in \cA_t)$, the locations of the active
particles.

Here is the algorithm to construct~$\cY$, starting from  $\cI_0 = \cA_0 = \emptyset$, and using $(\cX_t, t \in \RP)$ as defined by~\eqref{eq:x-def}. 
\begin{itemize}
\item[1.]
Suppose we have constructed $\cY_s$, $s \in [0,t]$. Let $i \geq 0$ be such that $s_i \leq t < s_{i+1}$ (where $s_0 := 0$).
At time $t$, let $\cI_t$ be the set of interior
islands, and let $\cA_t$ be the set of indices
of the active particles. 
For $j, k \in \cA_t$, $j < k$, let
\[ T_{j,k} := \inf \{ r \geq t : x_j (r) = x_k (r)  \} ,\] 
and, for $j \in \cA_t$, set
\[ T_j := \inf \{ r \geq t :  x_j (r) \in \cI_t \cup \{ 0, \ell \} \} .\]
Let $a_1 < a_2$ be the (a.s.~unique) indices such that $T_{a_1,a_2} = \min_{j,k:j < k} T_{j,k}$,
and let $a_0$ be the (a.s.~unique) index such that $T_{a_0} = \min_{j} T_j$.
Let $T = \min \{ T_{a_0}, T_{a_1,a_2} \}$.
\item[2.]
If $T > s_{i+1}$ then the next arrival occurs before any nucleation or absorption,
and we set $\cI_s = \cI_t$ for all $s \in (t,s_{i+1}]$,
$\cA_s = \cA_t$ for all $s \in (t,s_{i+1})$,
 and $\cA_{s_{i+1}} = \cA_t \cup \{ i + 1 \}$.
Update $t \mapsto s_{i+1}$ and return to Step~1.
\item[3.]
On the other hand, if $T < s_{i+1}$, we set
$\cI_s = \cI_t$ and $\cA_s = \cA_t$ for all $s \in (t,T)$, and proceed as follows at time $T$.
\begin{itemize}
\item
If $T_{a_1,a_2} < T_{a_0}$, nucleation of particles $a_1, a_2$ occurs at time $T$, and we set
$\cI_T = \cI_{t} \cup \{ x_{a_1} (T) \}$,
and
$\cA_T = \cA_t \setminus \{ a_1, a_2 \}$.
\item
If $T_{a_0} < T_{a_1,a_2}$, 
particle $a_0$ is captured by an existing island at time~$T$, and we set
$\cI_T = \cI_{t}$ and $\cA_T = \cA_t \setminus \{ a_0 \}$.
\end{itemize}
Then update $t \mapsto T$ and return to Step~1.
\end{itemize}

\begin{lemma}
The above construction defines a Markov process $\cY$ for all time.
\end{lemma}
\begin{proof}
The number
of Poisson arrivals in time interval $[0,t]$ is a.s.~finite,
so the number of active particles at any time 
is a.s.~finite, as is the number of islands.
Given a finite number of active particles and islands at distinct locations,
the independence property of the Poisson process and the Markov property of the Brownian motions
imply that the evolution until the next event (either nucleation, deposition, or adsorption)
is Markovian. The point-transience of planar Brownian motion implies the
following facts about multiple independent one-dimensional Brownian motions with generic starting points:
two Brownian motions
never visit a given point at the same time, 
three Brownian motions never meet simultaneously, and
 two pairs of Brownian motions have two different first meeting times.
Together with the fact that deposition locations a.s.~never coincide with the locations of any currently active particles or islands,
this means that 
 there is a.s.~a well-defined next event, and, up to and including the time of that next event,
 active particles and islands are always at distinct locations.
The above algorithm thus gives a well-defined construction from each event to the next.
Thus the process is well-defined for all time, and inherits the Markov property from the properties of the Poisson process and the Brownian motions.
\end{proof}
 
We denote by $\Prll$ the probability measure associated with the process $\cY$ constructed
above. In the special case $\ell =1$, we write simply $\Prl$. For the corresponding expectations we
use $\Expll$ and $\Expl$.
Let $A_t := | \cA_t |$ denote the number of active particles at time $t$,
and let $I_t := | \cI_t |$ denote the number of interior islands at time $t$.
Initially, $A_0 = I_0 = 0$.
Let $\cF_t := \sigma (\cY_s, 0 \leq s \leq t)$ denote the $\sigma$-algebra generated by the process up to time $t \in \RP$.

Define $\eta_0 := 0$. Also for $k \in \N$ define stopping times
\begin{align}
\label{eq:regeneration-times}
\sigma_k & := \inf \{ t > \eta_{k-1} : A_t = 1 \}, \text{ and } \eta_k := \inf \{ t > \sigma_{k} : A_t = 0 \} .\end{align}
Lemma~\ref{lem:regeneration-times-finite} below shows that 
all these stopping times are finite, a.s.; for $k \in \N$, we call the
time interval $[\sigma_{k},\eta_{k}]$ the $k$th \emph{cycle}. See Figure~\ref{fig:cycle} for an illustration.

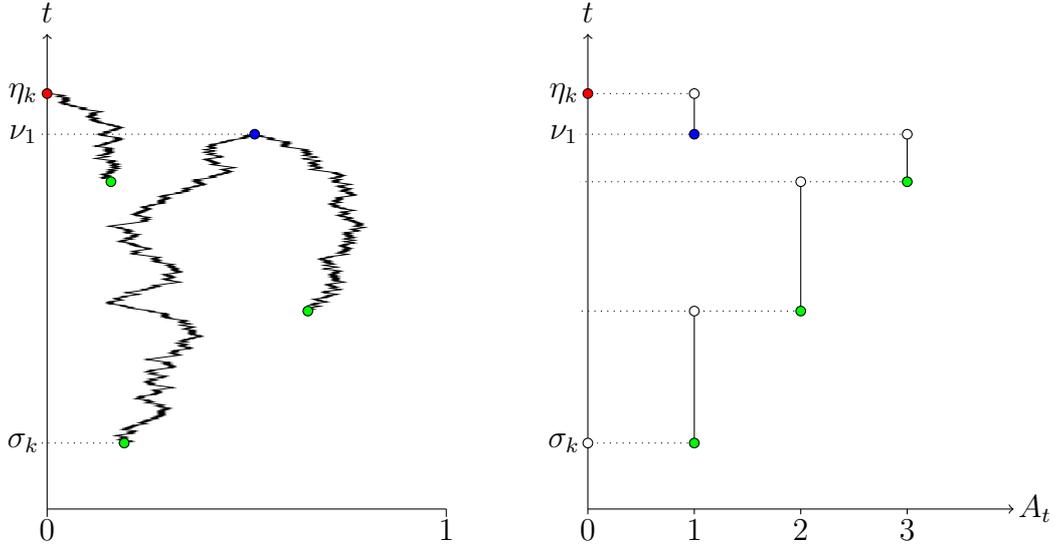
\begin{figure}
\centering
 \begin{tikzpicture}[domain=0:1, scale=0.7]
\draw[black] (0,0) -- (7.5,0);
\draw[black,->] (0,0) -- (0,9);
\node at (0,9.4)       {$t$};
\node at (0,-0.4)       {$0$};
\node at (7.5,-0.4)       {$1$};
\draw[black] (0,0) -- (0,-0.15);
\draw[black] (7.5,0) -- (7.5,-0.15);
\begin{axis}[width=8cm, height=12cm,axis lines=none,
  xmin=0,
  xmax=1,
  ymin=150,
  ymax=1400]
\addplot [mark=none] table [x index=1,y index=0] {
601 0.76
602 0.77
603 0.76
604 0.77
605 0.78
606 0.77
607 0.76
608 0.77
609 0.78
610 0.79
611 0.8
612 0.79
613 0.78
614 0.77
615 0.78
616 0.79
617 0.78
618 0.79
619 0.8
620 0.81
621 0.82
622 0.81
623 0.8
624 0.79
625 0.8
626 0.79
627 0.78
628 0.79
629 0.8
630 0.81
631 0.82
632 0.81
633 0.82
634 0.83
635 0.82
636 0.83
637 0.84
638 0.85
639 0.84
640 0.83
641 0.82
642 0.81
643 0.82
644 0.81
645 0.82
646 0.83
647 0.84
648 0.85
649 0.86
650 0.85
651 0.84
652 0.83
653 0.84
654 0.85
655 0.84
656 0.83
657 0.84
658 0.83
659 0.82
660 0.83
661 0.82
662 0.83
663 0.84
664 0.85
665 0.86
666 0.87
667 0.86
668 0.85
669 0.84
670 0.85
671 0.84
672 0.85
673 0.84
674 0.85
675 0.84
676 0.85
677 0.84
678 0.83
679 0.84
680 0.85
681 0.84
682 0.83
683 0.84
684 0.83
685 0.84
686 0.85
687 0.84
688 0.85
689 0.84
690 0.85
691 0.86
692 0.85
693 0.84
694 0.83
695 0.82
696 0.81
697 0.82
698 0.83
699 0.84
700 0.85
701 0.84
702 0.83
703 0.82
704 0.81
705 0.82
706 0.83
707 0.84
708 0.83
709 0.82
710 0.83
711 0.82
712 0.81
713 0.82
714 0.83
715 0.82
716 0.81
717 0.8
718 0.81
719 0.8
720 0.81
721 0.82
722 0.83
723 0.82
724 0.83
725 0.84
726 0.85
727 0.86
728 0.87
729 0.88
730 0.87
731 0.86
732 0.85
733 0.86
734 0.87
735 0.86
736 0.87
737 0.88
738 0.87
739 0.88
740 0.87
741 0.88
742 0.89
743 0.88
744 0.87
745 0.88
746 0.87
747 0.86
748 0.87
749 0.86
750 0.87
751 0.88
752 0.87
753 0.88
754 0.89
755 0.9
756 0.89
757 0.88
758 0.87
759 0.88
760 0.89
761 0.9
762 0.89
763 0.88
764 0.87
765 0.88
766 0.89
767 0.9
768 0.89
769 0.88
770 0.89
771 0.88
772 0.89
773 0.88
774 0.87
775 0.88
776 0.89
777 0.88
778 0.87
779 0.88
780 0.89
781 0.9
782 0.91
783 0.9
784 0.91
785 0.9
786 0.89
787 0.9
788 0.91
789 0.92
790 0.93
791 0.92
792 0.91
793 0.9
794 0.91
795 0.92
796 0.91
797 0.92
798 0.91
799 0.92
800 0.91
801 0.9
802 0.91
803 0.9
804 0.89
805 0.9
806 0.89
807 0.88
808 0.87
809 0.88
810 0.89
811 0.88
812 0.89
813 0.88
814 0.87
815 0.86
816 0.87
817 0.86
818 0.85
819 0.86
820 0.87
821 0.88
822 0.87
823 0.88
824 0.87
825 0.86
826 0.87
827 0.88
828 0.89
829 0.9
830 0.89
831 0.88
832 0.87
833 0.86
834 0.87
835 0.86
836 0.85
837 0.84
838 0.85
839 0.84
840 0.85
841 0.84
842 0.83
843 0.84
844 0.83
845 0.84
846 0.85
847 0.86
848 0.85
849 0.84
850 0.85
851 0.84
852 0.85
853 0.86
854 0.85
855 0.86
856 0.85
857 0.86
858 0.85
859 0.86
860 0.87
861 0.88
862 0.89
863 0.9
864 0.89
865 0.9
866 0.89
867 0.88
868 0.89
869 0.88
870 0.89
871 0.88
872 0.87
873 0.86
874 0.87
875 0.88
876 0.87
877 0.88
878 0.89
879 0.88
880 0.87
881 0.86
882 0.87
883 0.86
884 0.85
885 0.86
886 0.87
887 0.88
888 0.89
889 0.88
890 0.87
891 0.86
892 0.85
893 0.84
894 0.85
895 0.84
896 0.85
897 0.84
898 0.85
899 0.86
900 0.85
901 0.84
902 0.83
903 0.84
904 0.85
905 0.86
906 0.85
907 0.86
908 0.87
909 0.86
910 0.85
911 0.86
912 0.85
913 0.84
914 0.85
915 0.84
916 0.83
917 0.82
918 0.81
919 0.8
920 0.79
921 0.78
922 0.77
923 0.78
924 0.79
925 0.8
926 0.79
927 0.8
928 0.79
929 0.78
930 0.77
931 0.78
932 0.79
933 0.78
934 0.79
935 0.78
936 0.77
937 0.78
938 0.77
939 0.78
940 0.79
941 0.8
942 0.79
943 0.78
944 0.77
945 0.76
946 0.75
947 0.76
948 0.77
949 0.78
950 0.79
951 0.78
952 0.77
953 0.76
954 0.75
955 0.76
956 0.77
957 0.76
958 0.77
959 0.76
960 0.75
961 0.76
962 0.77
963 0.76
964 0.75
965 0.76
966 0.75
967 0.76
968 0.75
969 0.74
970 0.73
971 0.72
972 0.71
973 0.7
974 0.69
975 0.7
976 0.69
977 0.7
978 0.69
979 0.7
980 0.71
981 0.7
982 0.69
983 0.7
984 0.71
985 0.7
986 0.69
987 0.68
988 0.67
989 0.66
990 0.65
991 0.66
992 0.65
993 0.66
994 0.65
995 0.64
996 0.63
997 0.62
998 0.63
999 0.62
1000 0.61
1001 0.6
};
\addplot [mark=none] table [x index=1,y index=0] {
300 0.23
301 0.22
302 0.23
303 0.24
304 0.25
305 0.24
306 0.23
307 0.22
308 0.21
309 0.22
310 0.23
311 0.22
312 0.23
313 0.22
314 0.23
315 0.22
316 0.21
317 0.2
318 0.21
319 0.22
320 0.23
321 0.24
322 0.25
323 0.24
324 0.25
325 0.24
326 0.25
327 0.26
328 0.25
329 0.24
330 0.23
331 0.22
332 0.23
333 0.24
334 0.23
335 0.24
336 0.25
337 0.24
338 0.25
339 0.26
340 0.25
341 0.26
342 0.27
343 0.28
344 0.29
345 0.28
346 0.27
347 0.28
348 0.29
349 0.28
350 0.29
351 0.3
352 0.29
353 0.3
354 0.31
355 0.32
356 0.31
357 0.32
358 0.31
359 0.32
360 0.33
361 0.32
362 0.33
363 0.34
364 0.33
365 0.34
366 0.33
367 0.34
368 0.33
369 0.34
370 0.33
371 0.34
372 0.33
373 0.34
374 0.33
375 0.34
376 0.35
377 0.34
378 0.35
379 0.36
380 0.35
381 0.34
382 0.33
383 0.34
384 0.33
385 0.32
386 0.33
387 0.34
388 0.33
389 0.32
390 0.31
391 0.3
392 0.29
393 0.3
394 0.31
395 0.3
396 0.29
397 0.3
398 0.29
399 0.28
400 0.27
401 0.26
402 0.27
403 0.28
404 0.27
405 0.28
406 0.27
407 0.28
408 0.29
409 0.3
410 0.31
411 0.32
412 0.31
413 0.32
414 0.33
415 0.32
416 0.33
417 0.32
418 0.33
419 0.32
420 0.31
421 0.32
422 0.33
423 0.34
424 0.35
425 0.36
426 0.37
427 0.36
428 0.35
429 0.34
430 0.33
431 0.32
432 0.31
433 0.3
434 0.31
435 0.3
436 0.31
437 0.3
438 0.29
439 0.3
440 0.29
441 0.3
442 0.31
443 0.32
444 0.31
445 0.3
446 0.31
447 0.32
448 0.33
449 0.34
450 0.33
451 0.34
452 0.35
453 0.34
454 0.33
455 0.32
456 0.31
457 0.3
458 0.31
459 0.32
460 0.31
461 0.3
462 0.29
463 0.3
464 0.31
465 0.3
466 0.31
467 0.32
468 0.33
469 0.34
470 0.35
471 0.36
472 0.37
473 0.36
474 0.37
475 0.36
476 0.37
477 0.36
478 0.35
479 0.36
480 0.35
481 0.34
482 0.35
483 0.36
484 0.35
485 0.34
486 0.33
487 0.32
488 0.31
489 0.3
490 0.29
491 0.3
492 0.31
493 0.32
494 0.33
495 0.34
496 0.35
497 0.36
498 0.37
499 0.38
500 0.39
501 0.38
502 0.37
503 0.38
504 0.37
505 0.36
506 0.37
507 0.38
508 0.37
509 0.38
510 0.37
511 0.38
512 0.39
513 0.4
514 0.41
515 0.42
516 0.41
517 0.42
518 0.41
519 0.4
520 0.39
521 0.4
522 0.39
523 0.4
524 0.41
525 0.4
526 0.39
527 0.38
528 0.39
529 0.4
530 0.41
531 0.4
532 0.39
533 0.4
534 0.41
535 0.4
536 0.41
537 0.42
538 0.43
539 0.42
540 0.43
541 0.44
542 0.45
543 0.44
544 0.45
545 0.44
546 0.43
547 0.42
548 0.43
549 0.42
550 0.43
551 0.42
552 0.43
553 0.42
554 0.43
555 0.42
556 0.43
557 0.42
558 0.41
559 0.4
560 0.41
561 0.42
562 0.41
563 0.42
564 0.43
565 0.42
566 0.41
567 0.4
568 0.39
569 0.4
570 0.41
571 0.42
572 0.41
573 0.4
574 0.39
575 0.4
576 0.39
577 0.4
578 0.41
579 0.4
580 0.39
581 0.38
582 0.37
583 0.36
584 0.35
585 0.36
586 0.35
587 0.34
588 0.35
589 0.34
590 0.35
591 0.34
592 0.35
593 0.34
594 0.35
595 0.34
596 0.33
597 0.32
598 0.31
599 0.3
600 0.31
601 0.3
602 0.29
603 0.28
604 0.27
605 0.26
606 0.25
607 0.24
608 0.25
609 0.24
610 0.23
611 0.22
612 0.23
613 0.22
614 0.21
615 0.2
616 0.19
617 0.18
618 0.19
619 0.18
620 0.19
621 0.18
622 0.19
623 0.2
624 0.21
625 0.22
626 0.21
627 0.2
628 0.21
629 0.2
630 0.21
631 0.22
632 0.23
633 0.22
634 0.23
635 0.24
636 0.25
637 0.26
638 0.27
639 0.26
640 0.25
641 0.24
642 0.23
643 0.24
644 0.25
645 0.26
646 0.27
647 0.28
648 0.29
649 0.28
650 0.29
651 0.3
652 0.31
653 0.3
654 0.31
655 0.32
656 0.33
657 0.32
658 0.31
659 0.32
660 0.33
661 0.34
662 0.35
663 0.36
664 0.37
665 0.38
666 0.39
667 0.38
668 0.37
669 0.36
670 0.35
671 0.36
672 0.35
673 0.34
674 0.35
675 0.36
676 0.37
677 0.36
678 0.37
679 0.36
680 0.35
681 0.36
682 0.35
683 0.34
684 0.35
685 0.36
686 0.37
687 0.38
688 0.39
689 0.38
690 0.39
691 0.38
692 0.37
693 0.38
694 0.37
695 0.38
696 0.37
697 0.38
698 0.37
699 0.36
700 0.37
701 0.36
702 0.37
703 0.36
704 0.35
705 0.36
706 0.37
707 0.36
708 0.35
709 0.34
710 0.33
711 0.32
712 0.33
713 0.32
714 0.31
715 0.32
716 0.33
717 0.32
718 0.31
719 0.32
720 0.31
721 0.32
722 0.33
723 0.32
724 0.31
725 0.3
726 0.29
727 0.28
728 0.27
729 0.26
730 0.25
731 0.26
732 0.25
733 0.26
734 0.27
735 0.26
736 0.25
737 0.26
738 0.27
739 0.28
740 0.29
741 0.3
742 0.29
743 0.3
744 0.29
745 0.28
746 0.27
747 0.26
748 0.27
749 0.26
750 0.25
751 0.26
752 0.25
753 0.24
754 0.25
755 0.24
756 0.23
757 0.22
758 0.21
759 0.22
760 0.23
761 0.22
762 0.23
763 0.24
764 0.25
765 0.24
766 0.23
767 0.24
768 0.25
769 0.26
770 0.25
771 0.26
772 0.25
773 0.26
774 0.27
775 0.26
776 0.25
777 0.24
778 0.23
779 0.24
780 0.23
781 0.24
782 0.25
783 0.24
784 0.23
785 0.24
786 0.23
787 0.22
788 0.21
789 0.22
790 0.21
791 0.2
792 0.19
793 0.18
794 0.19
795 0.2
796 0.21
797 0.22
798 0.23
799 0.24
800 0.25
801 0.26
802 0.27
803 0.28
804 0.29
805 0.28
806 0.29
807 0.28
808 0.29
809 0.3
810 0.29
811 0.28
812 0.29
813 0.3
814 0.29
815 0.28
816 0.27
817 0.26
818 0.27
819 0.26
820 0.25
821 0.26
822 0.27
823 0.26
824 0.27
825 0.28
826 0.27
827 0.28
828 0.29
829 0.28
830 0.29
831 0.3
832 0.29
833 0.3
834 0.29
835 0.28
836 0.29
837 0.28
838 0.29
839 0.3
840 0.31
841 0.3
842 0.31
843 0.32
844 0.33
845 0.34
846 0.33
847 0.34
848 0.33
849 0.34
850 0.33
851 0.34
852 0.33
853 0.34
854 0.35
855 0.34
856 0.35
857 0.36
858 0.35
859 0.34
860 0.33
861 0.34
862 0.35
863 0.34
864 0.35
865 0.36
866 0.37
867 0.38
868 0.39
869 0.38
870 0.39
871 0.4
872 0.41
873 0.4
874 0.41
875 0.4
876 0.41
877 0.42
878 0.41
879 0.4
880 0.41
881 0.42
882 0.43
883 0.44
884 0.43
885 0.44
886 0.45
887 0.46
888 0.45
889 0.46
890 0.45
891 0.44
892 0.45
893 0.44
894 0.43
895 0.44
896 0.45
897 0.44
898 0.43
899 0.44
900 0.45
901 0.46
902 0.47
903 0.48
904 0.49
905 0.48
906 0.49
907 0.48
908 0.49
909 0.5
910 0.51
911 0.5
912 0.51
913 0.52
914 0.53
915 0.52
916 0.53
917 0.54
918 0.53
919 0.52
920 0.53
921 0.52
922 0.51
923 0.52
924 0.53
925 0.54
926 0.55
927 0.54
928 0.53
929 0.52
930 0.51
931 0.5
932 0.49
933 0.48
934 0.49
935 0.5
936 0.49
937 0.5
938 0.49
939 0.48
940 0.49
941 0.48
942 0.47
943 0.48
944 0.47
945 0.46
946 0.47
947 0.48
948 0.49
949 0.5
950 0.51
951 0.5
952 0.51
953 0.5
954 0.49
955 0.48
956 0.47
957 0.48
958 0.49
959 0.48
960 0.47
961 0.46
962 0.47
963 0.48
964 0.49
965 0.48
966 0.47
967 0.46
968 0.47
969 0.48
970 0.49
971 0.48
972 0.49
973 0.48
974 0.47
975 0.48
976 0.49
977 0.48
978 0.49
979 0.48
980 0.49
981 0.5
982 0.51
983 0.52
984 0.51
985 0.52
986 0.53
987 0.54
988 0.53
989 0.54
990 0.55
991 0.56
992 0.57
993 0.58
994 0.57
995 0.58
996 0.59
997 0.58
998 0.59
999 0.58
1000 0.59
1001 0.6
};
\addplot [mark=none] table [x index=1,y index=0] {
901 0.16
902 0.17
903 0.16
904 0.17
905 0.18
906 0.17
907 0.18
908 0.17
909 0.18
910 0.17
911 0.16
912 0.17
913 0.18
914 0.19
915 0.18
916 0.19
917 0.18
918 0.19
919 0.2
920 0.19
921 0.18
922 0.19
923 0.2
924 0.21
925 0.2
926 0.19
927 0.2
928 0.21
929 0.2
930 0.19
931 0.2
932 0.19
933 0.2
934 0.21
935 0.2
936 0.19
937 0.18
938 0.19
939 0.18
940 0.19
941 0.2
942 0.19
943 0.18
944 0.17
945 0.16
946 0.15
947 0.14
948 0.15
949 0.16
950 0.17
951 0.18
952 0.17
953 0.16
954 0.15
955 0.16
956 0.17
957 0.18
958 0.17
959 0.18
960 0.19
961 0.18
962 0.19
963 0.18
964 0.17
965 0.18
966 0.19
967 0.18
968 0.17
969 0.18
970 0.19
971 0.18
972 0.17
973 0.16
974 0.15
975 0.14
976 0.15
977 0.16
978 0.17
979 0.16
980 0.17
981 0.18
982 0.17
983 0.18
984 0.19
985 0.2
986 0.19
987 0.2
988 0.21
989 0.2
990 0.21
991 0.22
992 0.21
993 0.2
994 0.21
995 0.2
996 0.19
997 0.18
998 0.17
999 0.16
1000 0.17
1001 0.16
1002 0.15
1003 0.16
1004 0.17
1005 0.18
1006 0.17
1007 0.18
1008 0.17
1009 0.18
1010 0.19
1011 0.2
1012 0.19
1013 0.2
1014 0.21
1015 0.2
1016 0.21
1017 0.22
1018 0.21
1019 0.2
1020 0.21
1021 0.2
1022 0.19
1023 0.18
1024 0.19
1025 0.18
1026 0.17
1027 0.16
1028 0.15
1029 0.16
1030 0.15
1031 0.14
1032 0.15
1033 0.14
1034 0.13
1035 0.14
1036 0.13
1037 0.14
1038 0.15
1039 0.14
1040 0.13
1041 0.12
1042 0.13
1043 0.12
1044 0.13
1045 0.14
1046 0.13
1047 0.14
1048 0.13
1049 0.14
1050 0.13
1051 0.14
1052 0.15
1053 0.16
1054 0.15
1055 0.14
1056 0.13
1057 0.12
1058 0.11
1059 0.1
1060 0.11
1061 0.1
1062 0.11
1063 0.1
1064 0.09
1065 0.08
1066 0.07
1067 0.08
1068 0.07
1069 0.06
1070 0.05
1071 0.06
1072 0.05
1073 0.04
1074 0.05
1075 0.04
1076 0.05
1077 0.06
1078 0.05
1079 0.06
1080 0.05
1081 0.06
1082 0.05
1083 0.06
1084 0.05
1085 0.06
1086 0.05
1087 0.04
1088 0.03
1089 0.04
1090 0.05
1091 0.04
1092 0.03
1093 0.02
1094 0.01
1095 0.00
};
\end{axis}
\draw[dotted] (-0.1,7.1) -- (3.9,7.1);
\draw[dotted] (-0.1,1.25) -- (1.45,1.25);
\draw[black,fill=green] (1.45,1.25) circle (.5ex);
\draw[black,fill=green] (4.9,3.75) circle (.5ex);
\draw[black,fill=green] (1.2,6.2) circle (.5ex);
\draw[black,fill=blue] (3.9,7.1) circle (.5ex);
\draw[black,fill=red] (0,7.87) circle (.5ex);
\node at (-0.45,1.25) {$\sigma_k$};
\node at (-0.45,7.1) {$\nu_1$};
\node at (-0.45,7.87) {$\eta_k$};
\end{tikzpicture}
\qquad
 \begin{tikzpicture}[domain=0:1, scale=0.7]
\draw[black,->] (0,-0.1,0) -- (0,9);
\draw[black,->] (0,0) -- (8,0);
\node at (-0.45,1.25)       {$\sigma_k$};
\node at (-0.45,7.1)       {$\nu_1$};
\node at (-0.45,7.87)       {$\eta_k$};
\node at (0,9.4)       {$t$};
\node at (8.4,0)       {$A_t$};
\draw (6,-0.1) -- (6,0);
\draw (4,-0.1) -- (4,0);
\draw (2,-0.1) -- (2,0);
\draw (2,1.25) -- (2,3.75);
\draw (4,3.75) -- (4,6.2);
\draw (6,6.2) -- (6,7.1);
\draw (2,7.1) -- (2,7.87);
\node at (0,-0.4)       {$0$};
\node at (2,-0.4)       {$1$};
\node at (4,-0.4)       {$2$};
\node at (6,-0.4)       {$3$};
\draw[dotted] (2,1.25) -- (-0.15,1.25);
\draw[dotted] (4,3.75) -- (-0.15,3.75);
\draw[dotted] (6,6.2) -- (-0.15,6.2);
\draw[dotted] (6,7.1) -- (-0.15,7.1);
\draw[dotted] (2,7.87) -- (-0.15,7.87);
\draw[black,fill=white] (0,1.25) circle (.5ex);
\draw[black,fill=green] (2,1.25) circle (.5ex);
\draw[black,fill=white] (2,3.75) circle (.5ex);
\draw[black,fill=green] (4,3.75) circle (.5ex);
\draw[black,fill=white] (4,6.2) circle (.5ex);
\draw[black,fill=green] (6,6.2) circle (.5ex);
\draw[black,fill=white] (6,7.1) circle (.5ex);
\draw[black,fill=blue] (2,7.1) circle (.5ex);
\draw[black,fill=white] (2,7.87) circle (.5ex);
\draw[black,fill=red] (0,7.87) circle (.5ex);
\end{tikzpicture}
\caption{A cycle $[\sigma_k, \eta_k]$ which starts at time $\sigma_k$
with deposition of a particle into $[0,1]$, having no interior islands. The cycle contains two subsequent depositions, the first nucleation (at time $\nu_1$), and a capture of an active particle by an existing island.}
\label{fig:cycle}
\end{figure}

By definition,
$A_t = 0$ for $t \in [\eta_{k-1}, \sigma_{k})$,
so nucleation can only occur during the cycles $[\sigma_k,\eta_k]$.
Up until the first nucleation, the cycles $[\sigma_{k},\eta_{k}]$
encode a \emph{regeneration} structure that we will exploit.
Let $M_t := \max_{1 \leq i \leq I_t+1} L_{I_t,i}$, the length of the
longest gap at time~$t$. The next lemma is somewhat technical, but important: it gives a tail bound for the duration
of a cycle. The intuition is that
active particles are captured rather rapidly by existing islands, and faster still
if the gaps between islands are small.

\begin{lemma}
\phantomsection
\label{lem:regeneration-times-finite}
\begin{itemize}
\item[(i)] For any $\lambda >0$, we have that $\eta_k, \sigma_k < \infty$ for all $k \in \N$, $\Prl$-a.s.
\item[(ii)] For all $\lambda_0 \in (0,\infty)$ there exist constants $\delta = \delta(\lambda_0) >0$ and $C_1 = C_1 (\lambda_0) < \infty$ such that,
for all $\lambda \in (0,\lambda_0]$, for all $k \in \N$ and all $t \in \RP$,
\begin{equation}
\label{eq:cycle-tails}
 \Prl \bigl( \eta_k - \sigma_k \geq t  \bigmid \cF_{\sigma_k} \bigr) \leq C_1 \exp ( - \delta M_{\sigma_k}^{-1} t^{1/2}  ), \text{ $\Prl$-a.s.}
 \end{equation}
\end{itemize}
\end{lemma}
\begin{proof}
For $\lambda >0$, it is easy to see that $\eta_k < \infty$ implies $\sigma_{k+1} < \infty$, $\Prl$-a.s.;
since $\eta_0 = 0$ this means $\sigma_1 < \infty$.
Thus to show that $\eta_k, \sigma_k < \infty$ for all $k$,
it suffices to fix $k \in \N$ and to establish~\eqref{eq:cycle-tails} supposing that $\sigma_k < \infty$, a.s.
This is how we proceed.

Let $\lambda_0 \in (0,\infty)$ and $\lambda \in (0,\lambda_0]$. Write $M := M_{\sigma_k} \in (0,1]$. 
For $t\in \RP$, set 
\[ \tcI_t := \cI_{\sigma_k + M^2 t}, ~\tcA_t := \cA_{\sigma_k + M^2 t}, ~ \tA_t := A_{\sigma_k+M^2 t}, ~ \tcF_t := \cF_{\sigma_k+M^2t}, \]
and $\tx_j (t) := x_j (\sigma_k + M^2 t)$. In the rest of this proof, when we refer to `time' we mean the value of $t$ in the index $\sigma_k+M^2t$.
Let $B_t(j)$ be the event that $\tx_j(s)$, $j \in \tcA_t$, hits $\tcI_t$ at some time $s \in [t,t+1]$.
If $B_t(j)$ occurs, then particle $j$ is no longer active at time $t+1$, because either
it has been captured by an existing island, or it has collided with another active particle in the meantime. Also,
$| \tcA_{t+1} \setminus \tcA_t |$,
the number of new active particles at time $t+1$ compared to time~$t$, is bounded
by the number $\tZ_t$ of Poisson arrivals in time interval $[t,t+1]$. 
Thus
\[ \tA_{t+1} - \tA_t \leq \tZ_t - \sum_{j \in \tcA_t} \2 { B_t(j) } ,\]
and, by construction, $\tZ_t$ and the $B_t(j)$ are conditionally independent, given $\tcF_t$.
(This bound ignores nucleations, which can also eliminate active particles.) 
For $\delta >0$,
\[ \Expl \bigl( \re^{\delta(\tA_{t+1} - \tA_t) } \bigmid \tcF_t \bigr)
\leq \Expl \bigl( \re^{\delta \tZ_t} \bigmid \tcF_t \bigr) \prod_{j \in \tcA_t} \Expl \bigl( \re^{-\delta \2 { B_t(j) } } \bigmid \tcF_t \bigr) .\]
Given $\tcF_t$, $\tZ_t$ is Poisson with mean $\lambda M^2$, so 
 $\Expl ( \re^{\delta \tZ_t} \mid \tcF_t ) = \exp ( \lambda M^2 (\re^\delta-1) )$, while
\[ \Expl \bigl( \re^{-\delta \2 { B_t(j) } } \bigmid \tcF_t \bigr)
= 1   - (1- \re^{-\delta})\Prl ( B_t (j) \mid \tcF_t ) .\]
We claim that there is a constant $q >0$ such that
\begin{equation}
\label{eq:non-survival}
\Prl ( B_t (j) \mid \tcF_t ) \geq q , \as, \text{ for all $t \in \RP$ and all $j \in \tcA_t$}.
\end{equation} 
Indeed,
if $(w_t, t \geq 0)$ is standard Brownian motion on $\R$ with $w_0 = 0$,
then the claim~\eqref{eq:non-survival} holds with
$q = \Pr ( \sup_{0 \leq t \leq 1 } w_t \geq 1 ) = \Pr (\sup_{0 \leq t \leq M^2} w_t \geq M)$,
since $\tx_j(t)$ has at least one island within distance $M$ at time~$t$. 
By the reflection principle for Brownian motion~\cite[p.~45]{MoerPer},
$q = 2 \Pr ( w_1 \geq 1) \approx 0.317$.

Using~\eqref{eq:non-survival}, since $1 - z \leq \re^{-z}$ and $\lambda M^2 \leq \lambda_0$, we get
\begin{equation}
\label{eq:exponential-bound1}
 \Expl \bigl( \re^{\delta(\tA_{t+1} - \tA_t)} \bigmid \tcF_t \bigr)
\leq \exp \bigl( \lambda_0  (\re^\delta-1) - q \tA_t (1- \re^{-\delta}) \bigr) .\end{equation}
There is an absolute constant $\delta_0$ such that
$\re^\delta-1 \leq 2\delta$ and $1- \re^{-\delta} \geq \delta/2$
for all $\delta \in [0,\delta_0]$. Fix $\delta \in [0,\delta_0]$, and
let $a_0 := \lceil 6 \lambda_0 /q \rceil$, so $a_0 \in \N$. If $\tA_t \geq a_0$, then
\[ \lambda_0 (\re^\delta-1) - q \tA_t (1- \re^{-\delta}) \leq  2 \delta \lambda_0 - \frac{q a_0 \delta}{2} \leq - \delta \lambda_0 .\]
Thus we obtain from~\eqref{eq:exponential-bound1} that
\begin{align}
\label{eq:exponential-bound2} 
 \Expl \bigl( \re^{\delta(\tA_{t+1} - \tA_t)} \bigmid \tcF_t \bigr) & \leq \exp \left( - \delta \lambda_0 \right), \text{ on } \{ \tA_t \geq a_0 \}; \\
\label{eq:exponential-bound3} 
 \Expl \bigl( \re^{\delta(\tA_{t+1} - \tA_t)} \bigmid \tcF_t \bigr) & \leq \exp (2 \delta \lambda_0), \text{ on } \{ \tA_t < a_0\}. \end{align}

Set $\tau_0 := 0$ and define, for $r \in \N$, the stopping times
\begin{align*}
\gamma_r  := \inf \{ t \in \RP : t > \tau_{r-1} +1 , \, \tA_{t} \geq a_0 \}, \text{ and } \tau_r := \inf \{ t \in \RP : t > \gamma_r, \, \tA_t < a_0 \} .\end{align*}
Also define $\tau'_r  := \min \{  \gamma_r + n : n \in \N , \, \tA_{\gamma_r+n} < a_0 \}$. Then, a.s., $\tA_{\tau_0} = A_{\sigma_k} = 1$
and $\tau_r \leq \tau'_r$; also, $\tA_{\gamma_r}$ is bounded above by $a_0$ plus a Poisson random variable with mean~$\lambda_0$.
The Foster--Lyapunov drift bounds~\eqref{eq:exponential-bound2} and~\eqref{eq:exponential-bound3}
show that we may apply Theorem~2.3 of~\cite{hajek} to the discrete-time process $\tA_{\gamma_r}, \tA_{\gamma_r+1}, \ldots$ 
and its stopping time $\tau'_r - \gamma_r$ to show that, 
for some constants $\theta >0$ and $C < \infty$, depending only on $\lambda_0$,
\begin{equation}
\label{eq:hajek-bound}
\Expl \bigl( \re^{\theta  (\tau_r - \gamma_r)} \bigmid  \tcF_{\gamma_r} \bigr) \leq \Expl \bigl( \re^{\theta  (\tau'_r - \gamma_r)} \bigmid  \tcF_{\gamma_r} \bigr) \leq C , \as, \text{ for all } r \in \N .
\end{equation}

Next, with $\tZ_t$ again the number of depositions during time interval $[t,t+1]$,
we have that a sufficient condition for $\tA_{t+1} =0$
is that $\tZ_t = 0$ (no new arrivals) and all the active particles at time $t$
become inactive before time $t+1$; hence, by~\eqref{eq:non-survival},
\begin{align}
\label{eq:dies-out}
 \Prl ( \tA_{t+1} = 0 \mid \tcF_t ) 
\geq 
\Prl \biggl( \bigl\{ \tZ_t = 0 \bigr\} \cap \bigcap_{j \in \tcA_t} B_t (j) \biggmid \tcF_t \biggr)  
 \geq
q^{a_0} \re^{-\lambda_0}, \text{ on } \{ \tA_t \leq a_0 \} .\end{align}
Define for $r \in \N$ and $m \in \N$ the event 
\[ E(r,m) = \{ M^{-2} ( \eta_k - \sigma_k) > \tau_{r-1} + m, \, \gamma_r - \tau_{r-1} > m \} \in \tcF_{\tau_{r-1}+m}. \]
Then, since $E(r,m+1) \subseteq E(r,m)$,
\begin{align*}
 \Prl ( E(r,m+1) \mid \tcF_{\tau_{r-1}} )  
  = \Expl \Bigl[ \Prl  ( E(r,m+1) \mid \tcF_{\tau_{r-1}+m} )
\2 { E(r,m) } \Bigmid \tcF_{\tau_{r-1}} \Bigr] , \end{align*}
where, since $\tA_{\tau_{r-1}+m+1} = 0$
implies that $M^{-2} (\eta_k - \sigma_k) \leq \tau_{r-1} +m+1$,
\[  \Prl  ( E(r,m+1) \mid \tcF_{\tau_{r-1}+m} )   
 \leq    1 - \Prl ( \tA_{\tau_{r-1}+m+1} = 0 \mid \tcF_{\tau_{r-1}+m} )  
  . \]
Thus, by~\eqref{eq:dies-out} and the fact that $E(r,m)$
implies $\tA_{\tau_{r-1}+m} < a_0$,  with $\eps_0 = q^{a_0} \re^{-\lambda_0}$,
\[  \Prl ( E(r,m+1) \mid \tcF_{\tau_{r-1}} ) \leq \Expl \bigl[ ( 1 - \eps_0  )
\2 { E(r,m) } \bigmid \tcF_{\tau_{r-1}} \bigr] \leq \re^{-\eps_0} \Prl ( E(r,m) \mid \tcF_{\tau_{r-1}} ) .\]
Iterating this bound gives $\Prl ( E(r,m) \mid \tcF_{0} ) \leq \re^{-\eps_0 (m-1)}$, a.s., for all $m \in\N$.

Fix $t \in \N$.
Let $K = \max \{ r :   \tau_r \leq M^{-2} ( \eta_k -\sigma_k )\}$
and $L = \min \{ r :  \gamma_r - \tau_{r-1} >  t \}$. Then
$\tau_K \leq M^{-2} (\eta_k -\sigma_k )$
and we cannot have $\gamma_{K+1} \leq M^{-2} (\eta_k -\sigma_k )$, or else we would also have $\tau_{K+1} \leq M^{-2} (\eta_k -\sigma_k )$ too.
Thus $M^{-2} (\eta_k -\sigma_k ) \leq  \gamma_{K+1}$, so $M^{-2} (\eta_k -\sigma_k ) \leq \tau_K + (\gamma_{K+1} - \tau_K)$.
For $r < L$ we have
\[ \tau_r = \sum_{j=1}^{r} (\tau_j - \gamma_j) + \sum_{j=1}^r (\gamma_j - \tau_{j-1} ) \leq r t + \sum_{j=1}^{r} (\tau_j - \gamma_j) .\]
On the event
\[ \{ K < r \} \cap \left\{ \sum_{j=1}^{r} (\tau_j - \gamma_j)  \leq t \right\} \cap \bigcap_{j=1}^{r} \left( \{ M^{-2} ( \eta_k - \sigma_k ) \leq \tau_{j-1} + t \} \cup
\{ \gamma_j - \tau_{j-1} \leq  t \} \right) ,\]
we have that either $L >r$, in which case 
\[ M^{-2} (\eta_k -\sigma_k ) \leq \gamma_{K+1} \leq \gamma_r \leq \tau_{r-1} + \max_{1 \leq j \leq r} (\gamma_{j} - \tau_{j-1}) \leq (r+1)t, \]
or else $L \leq r$ and $M^{-2} (\eta_k -\sigma_k ) \leq \tau_{L-1} + t \leq (r+1)t$ also.
 Thus
\begin{align}
\label{eq:big-bound}
& {}  \Prl ( \eta_k - \sigma_k > (r+1) M^2 t \mid \cF_{\sigma_k} ) \\
& \quad {} \leq \Prl ( K \geq r \mid \cF_{\sigma_k} )
+ \Prl \biggl( \sum_{j=1}^r ( \tau_j - \gamma_{j} ) \geq t \biggmid \cF_{\sigma_k} \biggr)   + \Prl \biggl( \bigcup_{j=1}^{r} E(j,t) \biggmid \cF_{\sigma_k} \biggr) .
\nonumber
 \end{align}
For the first term on the right-hand side of~\eqref{eq:big-bound}, we have
\begin{align*}
\Prl ( K \geq m+1 \mid \cF_{\sigma_k} ) & = \Expl \bigl[ \Prl ( K \geq m+1 \mid \tcF_{\tau_m} ) \1 { K \geq m} \bigmid \tcF_{0} \bigr] \\
& \leq \Expl \bigl[ (1 - \Prl ( \tA_{\tau_m+1} = 0 \mid \tcF_{\tau_m} ) ) \1 { K \geq m} \bigmid \tcF_{0} \bigr] ,\end{align*}
since $\tA_{\tau_m+1} = 0$
implies $M^{-2} (\eta_k -\sigma_k ) \leq \tau_m +1 < \tau_{m+1}$.
By~\eqref{eq:dies-out}, $\Prl ( K \geq m+1 \mid \cF_{\sigma_k} ) \leq (1-\eps_0) \Prl (  K \geq m \mid \cF_{\sigma_k} )$,
and   $\Prl ( K \geq m \mid \cF_{\sigma_k} ) \leq \re^{-\eps_0m}$, a.s., for all $m \in \ZP$.

For the second term on the right-hand side of~\eqref{eq:big-bound}, we have from~\eqref{eq:hajek-bound} that
\begin{align*}
\Expl \Bigl[ \re^{\theta \sum_{j=1}^{r} ( \tau_j - \gamma_j) } \Bigmid \tcF_0 \Bigr]
& = \Expl \Bigl[ \re^{\theta \sum_{j=1}^{r-1} ( \tau_j - \gamma_j) } \Expl \bigl( \re^{\theta (\tau_r - \gamma_r ) } \bigmid \tcF_{\tau_{r-1}} \bigr)  \Bigmid \tcF_0  \Bigr] 
\\
& \leq C \Expl \Bigl[ \re^{\theta \sum_{j=1}^{r-1} ( \tau_j - \gamma_j) }  \Bigmid \tcF_0  \Bigr],
  \end{align*}
where $\theta, C$ are as in~\eqref{eq:hajek-bound}, and depend only on $\lambda_0$. Iterating this argument gives, for all $r \in \N$ and a constant $D < \infty$,
$\Expl( \re^{\theta \sum_{j=1}^{r} ( \tau_j - \gamma_j) } \mid \cF_{\sigma_k} ) \leq \re^{Dr}$, a.s.
By Markov's inequality,
\[ \Prl \biggl( \sum_{j=1}^r ( \tau_j - \gamma_{j} ) \geq t \biggmid \cF_{\sigma_k} \biggr) 
\leq \re^{Dr-\theta t} .\]
Choose $r = \lfloor \frac{\theta}{2D} \rfloor t$. 
Putting all the bounds together,  we obtain from~\eqref{eq:big-bound} that
\[  \Prl ( \eta_k - \sigma_k > c M^2 t^2 \mid \cF_{\sigma_k} ) \leq \re^{-\eps t} ,\]
for some constants $c>0$ and $\eps >0$. Then~\eqref{eq:cycle-tails} follows.
This completes the proof of (ii), and hence (i), as explained at the start of this proof.
\end{proof}

The next result shows that the cycles
 $[\sigma_{k},\eta_{k}]$ do not accumulate in finite time.

\begin{lemma}
\label{lem:regeneration-cycles-cover-all-time}
Let $\lambda \in (0,\infty)$. As $k \to \infty$, $\nu_k, \eta_k, \sigma_k \to \infty$, $\Prl$-a.s.
\end{lemma}
\begin{proof}
By construction, $\sigma_1, \sigma_2, \ldots$ is a subsequence of the Poisson
arrival times $s_1, s_2, \ldots$, and hence $\eta_k \geq \sigma_k \geq s_k$ for all $k \in \N$.
Similarly, the $k$th nucleation can only occur after $2k$ particles have been deposited, so $\nu_k \geq s_{2k}$ for all $k \in \N$.
But $\lim_{k \to \infty} s_k = \infty$, a.s.
\end{proof}

Lemma~\ref{lem:regeneration-cycles-cover-all-time}
  shows that $\nu_0,\nu_1, \ldots$ does not have a finite accumulation point,
so we can talk about the \emph{first} nucleation in any time interval which contains nucleations;
we have not yet proved that $\nu_k$ is finite for all $k$,
but we will do so later in this section.
Let $E_{k}$ be the event that at least one nucleation occurs 
in time interval~$[\sigma_k,\eta_k]$.
For $k \in \N$, let 
\[ \alpha_k := \min \{ j \in \N : E_{k+j} \text{ occurs} \} ,\]
the number of cycles after~$\eta_k$ until the first nucleation in time interval $(\eta_k, \infty)$,
where $\alpha_k = \infty$ if and only if there is no nucleation after time $\eta_k$. 

On the event $E_1$, 
the first nucleation occurs at some location $\zeta \ell$ for $\zeta  \in (0,1)$.
Recall that $\cB$ denotes the Borel subsets of $[0,1]$, and
denote the probability that the first nucleation occurs
during the first cycle and at spatial location in $\ell B$ by
\begin{equation}
    \label{eq:nu-def}
     \nu (\ell, \lambda ; B) := \Prll ( E_1 \cap \{ \zeta \in B \} ) , \text{ for } B \in \cB, \end{equation}
and set $\nu (\lambda ; B) := \nu (1, \lambda ; B)$. 
The probability of nucleation during the first cycle is
\[ \mu (\ell, \lambda ) := \Prll (E_1) = \nu (\ell, \lambda ; [0,1]), \text{ and } \mu (\lambda ) := \Prl (E_1) = \nu (\lambda; [0,1]) = \mu (1,\lambda).\]
The next result gives an important scaling property,
which is a consequence of the scaling properties of the Poisson process and of Brownian motion.
For a scalar $a >0$, let $a \cY_t = ( a \cI_t, \cA_t, a \cX_t)$,
i.e., scalar multiplication of all spatial variables.

\begin{lemma}
\label{lem:scaling}
Let $\ell, \lambda \in (0,\infty)$. Then
\begin{equation}
\label{eq:y-scaling}
 \Prll \left( \bigl(\ell^{-1} \cY_{\ell^2 t} \bigr)_{t \geq 0} \in \, \cdot \, \right) = \Pr_{\ell^3\lambda} \left( \bigl(  \cY_{t} \bigr)_{t\geq 0} \in \, \cdot \, \right) .\end{equation}
As a consequence, for all $\ell, \lambda \in (0,\infty)$,
\begin{equation}
\label{eq:nu-scaling}
 \nu (\ell, \lambda ; \, \cdot\, ) = \nu ( \ell^3 \lambda ; \, \cdot\,) .\end{equation}
\end{lemma}
\begin{proof}
For $\ell >0$, define the space-time scaling operation $S_\ell : \cC_0 \to \cC_0$ by $S_\ell (f) (t) = \ell^{-1} f (\ell^2 t)$, $t \in \RP$.
Then define the function $T_\ell : [0,\ell] \times \RP \times \cC_0 \to [0,1] \times \RP \times \cC_0$ by
\[ T_\ell ( x, s, f ) = \left( \frac{x}{\ell} , \frac{s}{\ell^2} , S_\ell (f) \right) .\]
We claim that 
\begin{equation}
\label{eq:mapping}
T_\ell (\cP_{\ell,\lambda} ) \text{ has the same law as } \cP_{1,\ell^3\lambda}. 
\end{equation}
To see this, 
view $\cP_{\ell,\lambda}$ as a Poisson point process with intensity measure
$\lambda \Lambda_\ell \otimes W$,
where $\Lambda_\ell$ is Lebesgue measure on $[0,\ell] \times \RP$
and $W$ is Wiener measure on $\cC_0$.
Brownian scaling~\cite[p.~12]{MoerPer} shows that~$W$ is preserved by the transformation $S_\ell$,
since   $b$ is standard Brownian motion on $\R$ if and only if $S_\ell(b)$
is too. Thus, by the mapping theorem~\cite[p.~38]{lp},
$T_\ell (\cP_{\ell,\lambda} )$ is a Poisson point process with intensity measure
$\ell^3 \lambda \Lambda_1 \otimes W$. This verifies~\eqref{eq:mapping}.

Let
$(\cY_s,s\in\RP)$
and
 $(\cY'_t,t \in \RP)$
denote the processes constructed from $\cP_{\ell,\lambda}$
and
$T_\ell (\cP_{\ell,\lambda} )$
using the algorithm described above.
The action of the map $T_\ell$ shows that
\begin{equation}
\label{eq:T-transform}
 \cY'_t = \frac{1}{\ell} \cY_{\ell^2 t} = \left( \frac{1}{\ell} \cI_{\ell^2 t} , \cA_{\ell^2 t} , \frac{1}{\ell} \cX_{\ell^2 t} \right) ,\end{equation}
since $T_\ell$ scales space by $1/\ell$ and time by $1/\ell^2$. For instance, the 
arrival time of the  $i$th particle
in $\cY'$ is $s_i' = \ell^{-2} s_i$, and the $i$th particle's 
trajectory $x_i'$ is given for $r \geq s'_i$ by 
\[ x_i' (r) = \frac{1}{\ell} \xi_i + S_\ell (b_i) ( r - \ell^{-2} s_i ) = \frac{1}{\ell} \xi_i + \frac{1}{\ell} b_i (\ell^2 r - s_i ) ,\]
so that $x_i' ( r ) = \ell^{-1} x_i ( \ell^2 r )$. Combining~\eqref{eq:T-transform}
with~\eqref{eq:mapping}, we see that $\ell^{-1} \cY_{\ell^2t}$ under $\Prll$
has the same law as $\cY'_t$ under $\Pr_{1,\ell^3\lambda}$. This proves~\eqref{eq:y-scaling}.
See Figure~\ref{fig:scaling} for a schematic.

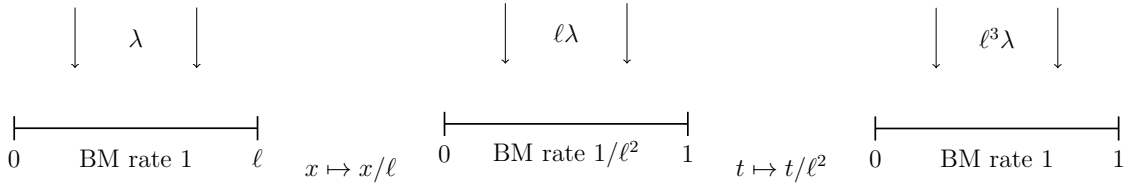
\begin{figure}[!h]
\centering
\scalebox{0.8}{
\begin{tikzpicture}
\draw[thick] (1,0) -- (5,0);
\draw[->] (2,2) -- (2,1);
\draw (3,1.5) node  {$\lambda$};
\draw[->] (4,2) -- (4,1);
\draw[thick] (1,0.2) -- (1,-0.2);
\draw[thick] (5,0.2) -- (5,-0.2);
\draw (1,-0.5) node  {$0$};
\draw (5,-0.5) node  {$\ell$};
\draw (3,-0.5) node  {BM rate $1$};
\end{tikzpicture}
\quad $x \mapsto x/\ell$ \quad
\begin{tikzpicture}
\draw[thick] (1,0) -- (5,0);
\draw[->] (2,2) -- (2,1);
\draw (3,1.5) node  {$\ell\lambda$};
\draw[->] (4,2) -- (4,1);
\draw[thick] (1,0.2) -- (1,-0.2);
\draw[thick] (5,0.2) -- (5,-0.2);
\draw (1,-0.5) node  {$0$};
\draw (5,-0.5) node  {$1$};
\draw (3,-0.5) node  {BM rate $1/\ell^2$};
\end{tikzpicture}
\quad $t \mapsto t/ \ell^2$ \quad
\begin{tikzpicture}
\draw[thick] (1,0) -- (5,0);
\draw[->] (2,2) -- (2,1);
\draw (3,1.5) node  {$\ell^3\lambda$};
\draw[->] (4,2) -- (4,1);
\draw[thick] (1,0.2) -- (1,-0.2);
\draw[thick] (5,0.2) -- (5,-0.2);
\draw (1,-0.5) node  {$0$};
\draw (5,-0.5) node  {$1$};
\draw (3,-0.5) node  {BM rate $1$};
\end{tikzpicture}}
\caption{Illustration of the scaling argument in the proof of Lemma~\ref{lem:scaling}.}
\label{fig:scaling}
\end{figure}

The event defining
$\nu (\ell, \lambda ; \, \cdot\, )$ in~\eqref{eq:nu-def},
namely $E_1 \cap \{ \zeta \in B \}$, 
is invariant under time-scaling,
and $\zeta$ is already scaled so as to be in $[0,1]$.
Then, by~\eqref{eq:y-scaling}, $\Prll ( E_1 \cap \{ \zeta \in B \} ) = \Pr_{\ell^3\lambda} ( E_1 \cap \{ \zeta \in B \} )$, 
which establishes~\eqref{eq:nu-scaling}.
\end{proof}

Fix $t \in \RP$.
Let $\zeta'_t \in (0,\ell)$ denote the location of the \emph{earliest} nucleation in the time interval~$(t, \infty)$, if there is one, otherwise
set $\zeta'_t = \infty$;
 we will shortly be able to prove Lemma~\ref{lem:tau-nice},
which says that there will a.s.~always be such a nucleation.
At time $t$ there are $I_{t} +1$ gaps,
and gap $j \in [I_t+1]$ is given by $[ Z_{I_{t},j-1} , Z_{I_{t},j}]$.
Define
\[ \zeta_t := \sum_{j \in [I_t+1]} \bbind \left\{ \zeta'_t \in ( Z_{I_{t},j-1} , Z_{I_{t},j} ) \right\} \left( \frac{\zeta'_t - Z_{I_{t},j-1} }
{Z_{I_{t},j}-Z_{I_{t},j-1} } \right) ,\]
so $\zeta_t \in (0,1)$ as long as $\zeta'_t$ is finite.  
For $t \in \RP$, $j \in [I_{t}+1]$, and $B \in \cB$, define the event
\begin{equation}
    \label{eq:D-def}
 D_{t} (j , B) := \left\{  \zeta'_t \in  [ Z_{I_{t},j-1} , Z_{I_{t},j}  ], \, \zeta_t \in B \right\} ,\end{equation}
which says that the next nucleation after time $t$ occurs in gap $j$ and at relative location in $B$.

By the strong Markov property, there is a measurable $\pi_\lambda$ such that, $\Prl$-a.s.,
\begin{equation}
\label{eq:pi-def}
  \Prl ( \{ \alpha_k = 1 \} \cap D_{\eta_k} (j , B) \mid \cF_{\eta_k} ) 
  = \pi_\lambda ( \cZ_{I_{\eta_k}} ; j, B)
  .\end{equation}
Similarly, there is a measurable $\Pi_\lambda$ such that 
\begin{equation}
\label{eq:Pi-def}
\Prl ( \alpha_k = 1 \mid \cF_{\eta_k} ) = \Pi_\lambda ( \cZ_{I_{\eta_k}} )
= \sum_{j \in [ I_{\eta_k}+1 ] } \pi_\lambda ( \cZ_{I_{\eta_k}} ; j, [0,1]) .\end{equation}
The next result gives a lower bound for $\pi_\lambda$ on a certain set;
in particular, it shows that $\Pi_\lambda  ( \cZ_{I_{\eta_k}} ) >0$. We will see in \S\ref{sec:sparse} that this bound is of the correct order as $\lambda \to 0$.

\begin{lemma}
\label{lem:intervals-shrink}
Let $B_0 := [ 1/8,7/8]$. For any $\lambda_0 \in (0,\infty)$ 
there exists a constant $\eps_0 = \eps_0 (\lambda_0) >0$ such that, for all $\lambda \in (0,\lambda_0]$ and all $k \in \ZP$, 
\[ \pi_\lambda (\cZ_{\eta_k}  ; j ,B_0 ) \geq \eps_0 \lambda L_{I_{\eta_k},j}^4, \text{ for all } j \in [ I_{\eta_k}+1 ].\]
\end{lemma}
\begin{proof}
Fix $k \in \ZP$.
To simplify notation, write $I = I_{\eta_k}$
for the number of interior islands, and, for $1 \leq j \leq I+1$, 
$Z_j = Z_{I,j}$ for the island locations 
and $L_{j} = L_{I,j}$ for the lengths of the gaps.
Take $j \in [I+1]$.
Define nested subintervals of $[Z_{j-1},Z_{j}]$ by
\[ \Lambda_{j,k} = [ Z_{j-1} + \tfrac{k+1}{8} L_{j}, Z_{j} - \tfrac{k+1}{8} L_{j} ], \text{ for } k \in \{0,1,2\} . \]
We will define a series of events whose intersection implies that
nucleation occurs in $\Lambda_{j,0}$.
Let $T := \min \{ i \in \ZP : s_i \geq \eta_k \}$.
Take a constant $t_0 \in (1,\infty)$ to be chosen later. 
Let 
\[ F_1 := \{ \xi_T \in \Lambda_{j,2} \} \cap \Bigl\{ \sup_{0 \leq t \leq L_{j}^2} | x_T (s_T+t) - \xi_T | < \tfrac{1}{8} L_{j} \Bigr\}, \]
the event that the next deposition occurs in $\Lambda_{j,2}$, and that $x_T$
stays in $\Lambda_{j,1}$ through time interval $[s_T,s_T+L_{j}^2]$. Define  the event 
\[ F_2 := \{ s_{T+1} \leq s_T + L_{j}^2 \}  \cap \{ \xi_{T+1} \in \Lambda_{j,1} \} \cap \{ s_{T+2} > s_T + t_0 L_{j}^2\} ,\]
that a single arrival occurs during time interval $(s_T,s_T+L_{j}^2]$ and
 at location in~$\Lambda_{j,1}$, and no arrival occurs during time interval $( s_T+L_{j}^2,s_T+t_0 L_{j}^2]$.

On $F_1 \cap F_2$,
at time $s_{T+1}$
both particles $T, T+1$ are active and are at locations in $\Lambda_{j,1}$, since neither can have encountered another active particle or an existing island.
Suppose (without loss of generality) that the leftmost of the two active particles $T, T+1$ at time $s_{T+1}$ is the particle labelled $T$: i.e., $x_T (s_{T+1} ) < x_{T+1} (s_{T+1})$.
Let $F_3$ denote the event that 
both $x_T$ visits $[ Z_{j} - \frac{1}{8} L_{j} , Z_{j} ]$ before
visiting $[Z_{j-1}, Z_{j-1} + \frac{1}{8} L_{j} ]$, and $x_{T+1}$ visits $[Z_{j-1}, Z_{j-1} + \frac{1}{8} L_{j} ]$ before visiting $[ Z_{j} - \frac{1}{8} L_{j} , Z_{j} ]$.
Also let $F_4$ denote the event that both $x_T$ and $x_{T+1}$ exit the interval $\Lambda_{j,0}$
before time $s_{T+1} + t_0  L_{j}^2$. If $F_3 \cap F_4$ occurs, then  the particles $T, T+1$ must meet in the interval $\Lambda_{j,0}$ before time $s_{T+1}+t_0  L_{j}^2$,
and, still being active, nucleate.
Therefore,
\begin{equation}
    \label{eq:four-events}
 \pi_\lambda (\cZ_{I}   ; j ,B_0 ) \geq \Prl (F_1 \cap F_2 \cap F_3 \cap F_4 \mid \cF_{\eta_k} ) 
 .\end{equation}

We bound the probability on the right-hand side of~\eqref{eq:four-events}: for concreteness, we give a quantitative estimate,
although we make no attempt to optimize the constants. We have
\[ \Prl ( F_1 \mid \cF_{s_T} )
\geq \frac{4}{\pi} \exp \left\{ -8\pi^2 \right\} - \frac{4}{3\pi}   \exp \left\{ - 72 \pi^2 \right\} =: q_1 > 10^{-35}, \text{ on } \{ \xi_T \in \Lambda_{j,2} \} ,\]
using bounds on two-sided exit times from e.g.~\cite[p.~1047]{jp}. 
Also, 
\begin{align*}
\Prl ( F_2 \mid \cF_{s_T} )
& \geq \frac{1}{2} L_{j} \cdot \lambda L_{j}^2 \re^{- L_{j}^2  \lambda}  \cdot \re^{-(t_0-1)  L_{j}^2  \lambda}  \geq p_\lambda L_{j}^3 ,\end{align*}
where  $p_\lambda :=  \frac{\lambda}{2} \re^{- \lambda t_0}$,
 and,
given $\cF_{s_T}$, $F_1$ and $F_2$ are independent.
So 
\[ \Prl (F_1 \cap F_2 \mid \cF_{\eta_k} ) =
\Expl \bigl[ \Prl ( F_1 \mid \cF_{s_T} ) \Prl ( F_2 \mid \cF_{s_T} ) \1 { \xi_T \in \Lambda_{j,2} } \bigmid \cF_{\eta_k} \bigr]
\geq \frac{q_1 p_\lambda}{4}  L_{j}^4 .\]
Brownian motion started at $x \in (a,b)$ hits $b$ before $a$ with probability $\frac{x-a}{b-a}$,
so 
\begin{equation}
\label{eq:f5}
\Prl (F_3 \mid \cF_{s_{T+1}} ) = \left( \frac{8 x_T (s_{T+1}) - 8 Z_{j-1} - L_j}{6L_j} \right) \left( \frac{8Z_j - 8 x_{T+1} (s_{T+1}) - L_j}{6L_j} \right),
\end{equation}
which is at least $1/36$ on $F_1 \cap F_2$.
Let $\tau$ be the first exit time of a Brownian motion started at $x \in [2/8,6/8]$
from the interval $[1/8,7/8]$. The minimal distance from $x$ to the set $\{1/8,7/8\}$
is at most $3/8$, so, if $w$ is Brownian motion on $\R$,
\[ \Pr ( \tau \geq t ) \leq \Pr \left( \sup_{0 \leq s \leq t} w_s \leq \frac{3}{8} \right)
= 1 - 2 \Pr \left( w_1 \geq \frac{3}{8 \sqrt{t}} \right) ,\]
by the reflection principle and scaling. Hence
\[ \Pr ( \tau \geq t ) \leq 2 \int_0^{3/(8\sqrt{t})} \frac{1}{\sqrt{2\pi}} \re^{-u^2 /2} \ud u
\leq \frac{3}{4 \sqrt{6 t}} .\]
Taking $t = t_0 = 1944$ ensures that $\Pr ( \tau \geq t_0 ) \leq 1/144$, so that,
by Brownian scaling,
\begin{equation}
\label{eq:f6} \Prl ( F^\rc_4 \mid \cF_{s_{T+1}} ) \leq 2 \Pr ( \tau \geq t_0 ) \leq  \frac{1}{72} , \text{ on } F_1 \cap F_2.
\end{equation}
Combining~\eqref{eq:f5} and~\eqref{eq:f6} we get
\[ \Prl ( F_3 \cap F_4 \mid \cF_{s_{T+1}} ) 
\geq \Prl ( F_3 \mid \cF_{s_{T+1}} ) - \Prl ( F^\rc_4 \mid \cF_{s_{T+1}} )
\geq \frac{1}{72}, \text{ on } F_1 \cap F_2.
\]
Hence we conclude that
\begin{align*}
     \Prl (F_1 \cap F_2 \cap F_3 \cap F_4  \mid \cF_{\eta_k}) & \geq
\Expl \bigl[  \Prl ( F_3 \cap F_4 \mid \cF_{s_{T+1}} ) \2 { F_1 \cap F_2 }  \bigmid \cF_{\eta_k} \bigr] \\
& \geq \frac{1}{72} \Prl ( F_1 \cap F_2 \mid \cF_{\eta_k} ) 
\geq \frac{q_1 p_\lambda}{288} L_{j}^4,\end{align*}
which, with~\eqref{eq:four-events}, completes the proof on setting $\eps_0 = \frac{q_1}{576} \re^{-1944\lambda_0}$.
\end{proof}

Now we can complete the proof of Lemma~\ref{lem:tau-nice}.

\begin{proof}[Proof of Lemma~\ref{lem:tau-nice}.]
Fix $\lambda_0 = \lambda > 0$, and let $\eps_0$ be as in Lemma~\ref{lem:intervals-shrink}.
By~\eqref{eq:Pi-def},
\begin{align*}
    \Prl (\alpha_k = 1 \mid \cF_{\eta_k} )   = \Pi_\lambda ( \cZ_{I_{\eta_k}} )    \geq \sum_{j \in [ I_{\eta_k}+1]} \pi_\lambda (\cZ_{I_{\eta_k}} ; j , B_0 )  
     \geq \eps_0 \lambda \sum_{j \in [ I_{\eta_k}+1]} L_{I_{\eta_k},j}^4 ,
\end{align*}
by Lemma~\ref{lem:intervals-shrink}.
Then, by Jensen's inequality, $\Prl (\alpha_k = 1 \mid \cF_{\eta_k} ) \geq \eps_0 \lambda (1 + I_{\eta_k})^{-3}$, since $\sum_{j \in [ I_{\eta_k}+1]} L_{I_{\eta_k},j} = 1$.
Also, $\{ \alpha_k = 1 \} \in \cF_{\eta_{k+1}}$. 
By L\'evy's extension of the Borel--Cantelli lemma
(e.g.~\cite[Corollary~7.20]{kall}), it follows that
$\sum_{k \in \ZP} (1+ I_{\eta_k})^{-3} = \infty$ implies that $\alpha_k =1$ for infinitely
many $k$.
On the other hand, if $\sum_{k \in \ZP} (1+ I_{\eta_k})^{-3} < \infty$,
then $I_{\eta_k} \to \infty$. In either case, there are infinitely many nucleations.
\end{proof}

The next result shows how the regenerative structure
leads to a description of 
the joint distribution of the gap which nucleates and the nucleation location
in terms of single-cycle distributions. Recall the definition of $D_t(j,B)$ from~\eqref{eq:D-def}.

\begin{lemma}
\label{lem:splitting-geometric} 
Let $\lambda \in (0,\infty)$. 
For all $k \in \ZP$, all $j \in [I_{\eta_k}+1]$, and all $B \in \cB$,
\[ \Prl ( D_{\eta_k} (j , B ) \mid \cF_{\eta_k} ) 
= \frac{\pi_\lambda ( \cZ_{I_{\eta_k}}  ; j , B )}{\Pi_\lambda ( \cZ_{I_{\eta_k}}  ) }, \text{ $\Prl$-a.s.}
\]
\end{lemma}
\begin{proof}
Fix $k \in \ZP$ and write $I = I_{\eta_k}$.
For $m \in \ZP$,
\begin{align}
&  {} \Prl ( \{ \alpha_k = m+1 \} \cap D_{\eta_k} (j , B) \mid \cF_{\eta_k} ) \nonumber\\
& \quad {} = \Expl \Bigl[ \Prl ( \{ \alpha_k = m+1 \} \cap D_{\eta_{k+m}} (j , B ) \mid \cF_{\eta_{k+m}} ) \1 { \alpha_k > m } \Bigmid \cF_{\eta_k} \Bigr] \nonumber\\
& \quad {} = \Expl \Bigl[  \pi_\lambda ( \cZ_I  ; j, B )  \1 { \alpha_k > m } \Bigmid \cF_{\eta_k} \Bigr] \nonumber\\
& \quad {} = \pi_\lambda ( \cZ_I ; j, B )  \Prl ( \alpha_k > m \mid \cF_{\eta_k} )  , \label{eq:j-nuleate3} 
\end{align}
using the regeneration at time~$\eta_{k+m}$ and~\eqref{eq:pi-def}.
Taking~$B=[0,1]$ and summing over $j \in [I+1]$, we get
$\Prl ( \alpha_k = m +1 \mid \cF_{\eta_k} )
= \Pi_\lambda ( \cZ_I ) \Prl ( \alpha_k > m \mid \cF_{\eta_k} ) $.
In other words,
\begin{align*} \Prl ( \alpha_k > m+1 \mid \cF_{\eta_k} )
& = \Prl ( \alpha_k >  m \mid \cF_{\eta_k} ) - \Prl ( \alpha_k = m +1 \mid \cF_{\eta_k} )\\
& = \left( 1 - \Pi_\lambda ( \cZ_I  ) \right)  \Prl ( \alpha_k >  m \mid \cF_{\eta_k} )   .\end{align*}
Iterating this gives $\Prl ( \alpha_k >  m \mid \cF_{\eta_k} ) = \left( 1 - \Pi_\lambda ( \cZ_I  ) \right)^m$.
Thus, by~\eqref{eq:j-nucleate3},
\[ \Prl ( \{ \alpha_k = m+1 \} \cap D_{\eta_k} (j , B) \mid \cF_{\eta_k} ) 
= \left( 1 - \Pi_\lambda ( \cZ_I  ) \right)^m \pi_\lambda ( \cZ_I  ; j, B ) .\]
Summing over $m \in \ZP$ gives the result.
\end{proof}

\section{Splitting distribution estimates}
\label{sec:splitting}

Define $\kappa_n : \Delta_n \times [n+1] \times \cB \to [0,1]$, $n \in \ZP$, by
\begin{equation}
\label{eq:kappa-def}
 \kappa_n ( z ; j, B) := \frac{ ( z_j - z_{j-1} )^4}{\sum_{i \in [n+1]}  (z_i - z_{i-1} )^4} \Phi_0 (B) ,\end{equation}
the interval-splitting kernel in~\eqref{eq:interval-splitting} specialized to the parameters $r_0$ and $\Phi_0$ as appearing in Theorem~\ref{thm:small-lambda}.
The main result of this section, as follows, shows that
  the evolution of the island locations in our nucleation process is approximated by the kernel~\eqref{eq:kappa-def}.
	This result will serve both for fixed time as $\lambda \to 0$, and for fixed $\lambda$ in the long-time limit. In the supremum in Proposition~\ref{prop:splitting-distribution}, and subsequent similar instances,  $B_j \in \cB$ for each $j$.

\begin{proposition}
\label{prop:splitting-distribution} 
For any $\lambda_0 \in (0,\infty)$ 
there exists a constant $C_{2} = C_{2} (\lambda_0) < \infty$
such that, for all $\lambda \in (0,\lambda_0]$ and all $k \in \ZP$, $\Prl$-a.s.,
\[    \sup_{ B_1 , \ldots, B_{I_{\eta_k}+1} } \, \left|
\sum_{j \in [I_{\eta_k}+1]} \! \Prl ( D_{\eta_k} (j, B_j) \mid \cF_{\eta_k} ) -  \!
\sum_{j \in [I_{\eta_k}+1]} \! \kappa_{I_{\eta_k}} ( \cZ_{I_{\eta_k}} ; j , B_j ) 
  \right| \leq
C_{2} \lambda^{1/2} M^{3/2}_{\eta_k} .	\] 
\end{proposition}

The rest of this section will develop the proof of Proposition~\ref{prop:splitting-distribution},
which is built on the regeneration structure in Lemma~\ref{lem:splitting-geometric}.
First, we explain the origin of $\Phi_0$.

Recall from~\eqref{eq:nu-def} that $\nu ( \lambda ; B ) = 
 \Prl ( E_1 \cap \{ \zeta \in B \} )$,
where $E_1$ is the event that at least one nucleation occurs in time interval $[\sigma_1,\eta_1]$,
and $\zeta$ is the location of the first nucleation.
For the $\lambda \to 0$
asymptotics of $\nu(\lambda;B)$ we need some more notation.

Let $w$ denote standard Brownian motion on $\R$, started at $x \in [0,1]$,
and set $\tau := \inf \{ t \in \RP : w_t \notin (0,1) \}$,
the first exit time from the interval $(0,1)$. Then for $B \in \cB$,
\[ \Pr ( w_t \in B, \, t \leq \tau \mid w_0 = x ) = \int_B q_t (x,y) \ud y, \]
where
\begin{equation}
\label{eq:q-def}
 q_t (  x,y ) :=
\frac{1}{\sqrt{2\pi t}} \sum_{k \in \Z} \left\{ \exp \left( - \frac{ (y-x+2k)^2}{2t} \right)
- \exp \left( - \frac{(y+x+2k)^2 }{2t} \right) \right\} ; \end{equation}
see e.g.~\cite[pp.~341--342]{feller2} or~\cite[pp.~122, 174]{bosa}. The density $q_t ( x ; \, \cdot \, )$
corresponds to a (defective) distribution with total mass $\Pr ( t \leq \tau \mid w_0 = x)$.

Let $W$ denote a standard Brownian motion in $\R^2$ given in  components as  $W_t = (W_t^{(1)}, W_t^{(2)})$, 
and let
$S := \partial [0,1]^2$ and $D := \{ (x,y) \in [0,1]^2 : x = y\}$
denote the boundary and diagonal of the unit square, respectively.
For measurable $A \subseteq \R^2$, define $\tau_A := \inf \{ t \in \RP : W_t \in A\}$.
For $u,v \in [0,1]^2$ and $B \in \cB$, set
\begin{equation}
\label{eq:H-def}
H ( u ,v ; B ) := \Pr ( \tau_D < \tau_S, \, W^{(1)}_{\tau_D} \in B \mid W_0 = (u,v) ) , \end{equation}
so that $H(u,v ; \, \cdot \,)$ is a measure on $([0,1],\cB)$ with total mass 
$H(u,v; [0,1]) = \Pr ( \tau_D < \tau_S \mid W_0 = (u,v) )$.
Define 
\begin{equation}
\label{eq:Phi1}
\Phi_1 (B) := \int_0^1 \ud z \int_0^1 \ud y \int_0^1 \ud x \int_0^\infty q_t (x, y) H ( y, z ; B) \ud t .\end{equation}
The proof of the following result is given in~\S\ref{sec:triangle}.
\begin{proposition}
\label{prop:Phi0-Phi1}
We have that $\Phi_1 = \mu \Phi_0$, where $\mu$
is given by~\eqref{eq:mu-def} and $\Phi_0$ is defined at~\eqref{eq:Phi0-def}.
In particular, $\Phi_1 ( [0,1] ) = \mu$.
\end{proposition}

We will use the simple fact that if $Z$ is Poisson with mean $\theta \in \RP$, then for all $k \in \N$,
\begin{align}
\label{eq:poisson-bound}
 k \Pr ( Z \geq k ) 
& \leq \Exp \bigl[ Z \1 { Z \geq k } \bigr]  =   \re^{-\theta} \sum_{\ell = k-1}^\infty \frac{\theta^{\ell+1}}{\ell!} \nonumber\\
& \leq \theta^k \re^{-\theta} \sum_{\ell = k-1}^\infty \frac{\theta^{\ell-k+1}}{(\ell - k+1)!}
= \theta^k . \end{align}
The next result shows how $\Phi_1$, and hence, by Proposition~\ref{prop:Phi0-Phi1}, $\Phi_0$,
 arises in our model.

\begin{lemma}
\label{lem:hitting-distribution-small-lambda}
For any  $\lambda_0 \in (0,\infty)$  there is a constant $C_{3} = C_{3} (\lambda_0) < \infty$ such that,
\begin{equation}  
\label{eq:lambda1}
\sup_{B \in \cB} \left|  \nu (\lambda ; B ) - \lambda \Phi_1 (B) \right| \leq C_{3} \lambda^{5/3}, \text{ for all } \lambda \in (0,\lambda_0]. 
\end{equation}
\end{lemma}
\begin{proof}
In order for $E_1$ to occur,  the particle that arrives at time $s_1 = \sigma_1$
must remain active until the second particle arrives at time $s_2$ (or else the number of active particles
would fall to zero). 
Define events $F_1 (B) = \{  1 \in \cA_{s_2}, \, x_1 (s_2) \in B \}$, and $F_1 = F_1 ((0,1))$,
the event that the first particle is still active when the second one arrives. 
We have
\[ \Prl ( F_1 (B) ) = \Expl [ q_{s_2-s_1} ( \xi_1 , B ) ]
= \int_B  \ud y \int_0^1 \ud x \int_0^\infty \lambda \re^{-\lambda t} q_t (x, y) \ud t ,\]
since $\xi_1$ is uniform on $[0,1]$, $s_2 - s_1$ is exponential with parameter $\lambda$,
and the two are independent. From time~$s_2$, on $F_1$,
there are active particles at  $x_1 (s_2) = y$ (say) and $x_2 (s_2) = \xi_2 = z$ (say);
if these two particles meet in $B$ before either exits $[0,1]$ (call this event $F_2(y,z; B)$), and no other particle is deposited in the meantime,
then $E_1 \cap \{ \zeta \in B \}$ occurs.
Any other way for $E_1 \cap \{ \zeta \in B \}$ to occur requires that a third particle arrive before time $\eta_1$.
Thus if $F_3 = \{ s_3 > \eta_1 \}$, we have 
\[ F_1 \cap F_2 (x_1(s_2), \xi_2;B)  \cap F_3 \subseteq E_1 \cap \{ \zeta \in B \} \subseteq \left( F_1 \cap F_2 (x_1(s_2), \xi_2;B) \right) \cup F_3^\rc .\]
It follows that
\begin{equation} 
\label{eq:F3-bound}
\left| \Prl (  E_1 \cap \{ \zeta \in B \}  ) - \Prl (  F_1 \cap F_2 (x_1(s_2), \xi_2;B)  ) \right| \leq \Prl ( F_3^\rc ) .\end{equation}
Here
\[ 
\Prl (  F_1 \cap F_2 (x_1(s_2), \xi_2; B) )
= \int_0^1 \ud z \int_0^1  \Prl ( F_1 (\ud y) ) \Prl ( F_2 (y,z; B ) ) ,\]
using the Markov property at time $s_2$, and the fact that $\xi_2$ is uniform on $[0,1]$.
Thus
\begin{align*}
\Prl (  F_1 \cap F_2 (x_1(s_2), \xi_2 ; B) )
& = \lambda  \int_0^1 \ud z \int_0^1 \ud y \int_0^1 \ud x \int_0^\infty \re^{-\lambda t} q_t (x, y) H ( y, z ; B) \ud t , \end{align*}
and hence, by~\eqref{eq:Phi1} and the fact that $H ( y, z ; B) \leq 1$ and $1 - \re^{-z} \leq z$,
\begin{align}
\label{eq:lambda-zero}
\sup_{B \in \cB} \left| \Prl (  F_1 \cap F_2 (x_1(s_2), \xi_2 ; B) ) - \lambda \Phi_1 (B) \right|
& \leq \lambda  \int_0^1 \ud x \int_0^\infty \bigl( 1 -  \re^{-\lambda t} \bigr) \Pr ( \tau \geq t \mid w_0 = x)  \ud t \nonumber\\
& \leq \lambda^2 \int_0^1   \Exp ( \tau^2 \mid w_0 = x)  \ud x,
\end{align}
which is $O (\lambda^2)$.
Let $\eps \in (0,1)$. If $Z$ is the number of arrivals in time interval $(\sigma_1,\sigma_1+\lambda^{-\eps}]$,
then, since $Z$ is Poisson with mean $\lambda^{1-\eps}$,
$\Prl ( Z \geq 2 ) \leq \lambda^{2-2\eps} $ by~\eqref{eq:poisson-bound}, and 
\begin{align}
\label{eq:F3-prob}
\Prl ( F_3^\rc ) 
  \leq \Prl ( \eta_1 - \sigma_1 \geq \lambda^{-\eps} ) + \Prl ( Z \geq 2 )   \leq C_1 \exp ( - \delta \lambda^{-\eps/2} ) +   \lambda^{2-2\eps}   ,\end{align}
by Lemma~\ref{lem:regeneration-times-finite}.
The result follows from \eqref{eq:F3-bound}, \eqref{eq:lambda-zero}, and~\eqref{eq:F3-prob}.
\end{proof}

Consider the end of a cycle at time $\eta_k$. Denote by $J_k \in [I_{\eta_k}+1]$
the index such that the arrival at time $\sigma_{k+1}$ lands in gap $[Z_{I_{\eta_k}, J_k-1},Z_{I_{\eta_k}, J_k}]$.
Let $\cF'_{\eta_k}$ denote the $\sigma$-algebra
generated by $\cF_{\eta_k}$ and the value~$J_k$, so $\cF'_{\eta_k}$
identifies the gap occupied by the first arrival after~$\eta_k$, but not that arrival's 
location in the gap. 

Let $G_{k} (j,s)$ be the event that
during time interval $[\sigma_{k+1},\sigma_{k+1}+s]$ at least one nucleation occurs in gap $j \in [I_{\eta_k} +1]$.
The next result gives an upper bound on nucleation occurring outside gap $J_k$
during a fixed time horizon.

\begin{lemma}
\label{lem:nucleation-bound}
Let $\lambda_0 \in (0,\infty)$.
There exists a constant  $C_4 = C_4 (\lambda_0) < \infty$ 
 such that,
for all $\lambda \in (0,\lambda_0]$, all $s \in [ 0, \frac{1}{2\lambda} ]$, 
all $k \in \ZP$, 
and all $j \in [I_{\eta_k} +1] \setminus \{ J_k \}$,
\[   \Prl ( G_{k} (j,s) \mid \cF'_{\eta_k} ) \leq C_4 \lambda^2 s L_{I_{\eta_k},j}^4  , \text{ $\Prl$-a.s.} \]
\end{lemma}
\begin{proof}
Fix $k \in \ZP$ and write $I = I_{\eta_k}$, $J = J_k$, and, for $1 \leq j \leq I+1$,
 $Z_j = Z_{I,j}$ and $L_j = L_{I,j}$. 
Given $\cF_{\eta_k}'$, take $j \in [I+1] \setminus \{ J \}$.
For the process restricted to the interval $[Z_{j-1},Z_j]$,
define `local cycles' $[\sigma_{j,\ell},\eta_{j,\ell}]$ by $\eta_{j,0} := \sigma_{k+1}$ (at which point
there are no active particles in gap $j$) and, for $\ell \in \N$,
\begin{align*} \sigma_{j,\ell} 
& = \inf \Bigl\{ t > \eta_{j,\ell-1} : \sum_{i \in \cA_t} \1 { x_i (t) \in [ Z_{j-1},Z_j ] }  = 1 \Bigr\} , \\ 
\eta_{j,\ell} & = \inf \Bigl\{ t > \sigma_{j,\ell} :  \sum_{i \in \cA_t} \1 { x_i (t) \in [ Z_{j-1},Z_j ] } = 0 \Bigr\} .\end{align*}
Nucleation in $[Z_{j-1},Z_j]$ can only occur during time intervals $[\sigma_{j,\ell},\eta_{j,\ell}]$.
In order for $G_k (j,s)$ to occur via nucleation during $[\sigma_{j,\ell},\eta_{j,\ell}]$,
there must have been at least $\ell$ arrivals in $[Z_{j-1},Z_j]$ during time interval $[\sigma_{k+1}, \sigma_{k+1} + s]$,
and then nucleation must occur during that cycle, an event of probability $\mu ( L_j, \lambda )$. Together with~\eqref{eq:poisson-bound}
this gives
\begin{align*}
\Prl ( G_{k} (j,s) \mid \cF'_{\eta_k} ) 
& \leq \sum_{\ell =1}^\infty \lambda^\ell L_j^\ell s^\ell \mu (L_j, \lambda ) = \lambda L_j s \mu (L_j^3 \lambda) \sum_{\ell =0}^\infty \lambda^\ell L_j^\ell s^\ell,
\end{align*}
by~\eqref{eq:nu-scaling}. Here $\lambda L_j s \leq 1/2$, provided $s \leq \frac{1}{2\lambda}$. Then $\Prl ( G_{k} (j,s) \mid \cF'_{\eta_k} ) \leq 2 \lambda L_j s \mu (L_j^3 \lambda)$,
and Lemma~\ref{lem:hitting-distribution-small-lambda} completes the proof.
\end{proof}

Now we can give the proof of Proposition~\ref{prop:splitting-distribution}.

\begin{proof}[Proof of Proposition~\ref{prop:splitting-distribution}.]
Fix $k \in \ZP$ and write $I = I_{\eta_k}$, $J=J_k$, $M = M_{\eta_k}$, and, for $1 \leq j \leq I+1$,
$Z_j = Z_{I,j}$
and $L_{j} = L_{I,j}$.
The new arrival at time $\sigma_{k+1}$ is deposited in gap~$J$. We show that the main contribution to
$\Prl ( D_{\eta_k} (j, B) \mid \cF_{\eta_k} )$ comes from $J = j$.

Fix $\lambda_0 \in (0,\infty)$.
Define event $G_k (s) := \cup_{j \in [I+1] \setminus \{ J\}} G_{k} ( j , s)$,
that there is at least one nucleation outside interval $[Z_{J-1}, Z_{J}]$
during time interval $[\sigma_{k+1}, \sigma_{k+1}+s]$.
By Lemma~\ref{lem:nucleation-bound},
\begin{align*}
\Prl ( G_k(s) \mid \cF_{\eta_k} ) 
& \leq \Expl \biggl[  \sum_{j \in [I+1] \setminus \{ J \} } \Pr ( G_{k} ( j , s) \mid \cF_{\eta_k}' ) \biggmid \cF_{\eta_k} \biggr] \\
& \leq  C_4 \lambda^2 s \Expl \biggl[  \sum_{j \in [I+1] \setminus \{ J \} }  L_{j}^4 \biggmid \cF_{\eta_k} \biggr]  \leq C_4 \lambda^2 s \sum_{j \in [I+1]}  L_{j}^4 ,
\end{align*}
for all $\lambda \in (0,\lambda_0]$ and all $s \leq \frac{1}{2\lambda}$.
Moreover, from Lemma~\ref{lem:regeneration-times-finite} we have that
\begin{align*} \Prl  ( \eta_{k+1} - \sigma_{k+1} \geq s \mid \cF_{\eta_k} ) & = \Expl \bigl[  \Prl  ( \eta_{k+1} - \sigma_{k+1} \geq s \mid \cF_{\sigma_{k+1}} ) \bigmid \cF_{\eta_k} \bigr] \\
& \leq C_1 \Expl \bigl[ \exp ( - \delta M_{\sigma_{k+1}}^{-1} s^{1/2} )  \bigmid \cF_{\eta_k} \bigr] . \end{align*}
Then, if $G_\star := G_k ( \eta_{k+1} - \sigma_{k+1})$,  since $M_{\sigma_{k+1}} \leq M_{\eta_k} = M$, we get
\begin{align}
\label{eq:E-star-bound}
 \Prl  ( G_\star \mid \cF_{\eta_k} )
& \leq \Prl ( \eta_{k+1} - \sigma_{k+1} \geq s \mid \cF_{\eta_k} ) + \Prl ( G_k(s) \mid \cF_{\eta_k} )  \nonumber\\
& \leq C_1 \exp ( - \delta M^{-1} s^{1/2} ) + C_4 \lambda^2 s \sum_{i \in [I+1]}  L_{i}^4 , 
\end{align}
for   all $\lambda \in (0,\lambda_0]$ and all $s \leq \frac{1}{2\lambda}$.
For $\eps \in (0,1)$, take 
$s = M^{2-\eps} \min ( \lambda^{-1/2} , \frac{1}{2} \lambda^{-1} )$. By~\eqref{eq:E-star-bound} and noting that  $\sum_{i \in[I+1]} L_i^4 \geq M^4$,
we get, for some $C<\infty$ and all $\lambda \in (0,\lambda_0]$,
\begin{equation}
\label{eq:E-star}
 \Prl ( G_\star \mid \cF_{\eta_k} ) \leq C \lambda^{3/2} M^{2-\eps} \sum_{i \in[I+1]} L_i^4  , \as
 \end{equation}

For $t\in\RP$, let $A'_t$ denote the number
of active particles in gap $[Z_{J-1},Z_J]$
at time $\sigma_{k+1} + t$, and let $\eta' := \inf \{ t > 0 : A'_t = 0\}$;
note $A'_0 = 1$.
Let $\zeta^\star \in (0,1)$ denote the relative location
of the first nucleation for 
the process restricted to gap $[Z_{J-1},Z_J]$.
Observe that on the event $G^\rc_\star$,
 we have $\zeta_{\eta_k} = \zeta^\star$ and nucleation occurs in gap~$J$.
Let $E'$ be the event that nucleation occurs in gap $[Z_{J-1},Z_J]$
during time interval $[\sigma_{k+1},\sigma_{k+1}+\eta']$. Then
\begin{align}
\label{eq:j-nucleate1}
  \left|\, \sum_{j \in [I+1]} \Prl ( \{ \alpha_k = 1 \} \cap D_{\eta_k} (j , B_j) \mid \cF_{\eta_k} ) 
   \right. & \left. \! {} - {} \!\! \sum_{j \in [I+1]} \Prl ( E' \cap \{ J =j \} \cap \{ \zeta^\star \in B_j \}  \mid \cF_{\eta_k} ) \, \right|  \nonumber\\
& {}    \qquad \qquad\qquad \qquad  {} \leq \Prl ( G_\star  \mid \cF_{\eta_k} ) .
\end{align}
Here
\begin{align*}
  \Prl ( E' \cap \{ J = j \} \cap \{ \zeta^\star \in B_j \} \mid \cF_{\eta_k} ) = 
	\Expl \Bigl[ \Prl ( E' \cap \{ \zeta^\star \in B_j \} \mid \cF'_{\eta_k} ) \1 { J = j } \Bigmid \cF_{\eta_k} \Bigr] .\end{align*}
	The event $E' \cap \{ \zeta^\star \in B\}$ depends only on the
	process restricted to the interval $[Z_{J-1},Z_J]$ after time $\sigma_{k+1}$,
	which has the same law as the process
	on interval $[0,L_J]$ after time $\sigma_1$, for which the event
	 $E' \cap \{ \zeta^\star \in B \}$ translates as $E_1 \cap \{ \zeta \in B \}$. Thus, 
\begin{align*}
\Prl ( E' \cap \{ \zeta^\star \in B \} \mid \cF'_{\eta_k} ) & = \Pr_{L_J,\lambda}  ( E_1 \cap \{  \zeta \in B \} )   = \nu ( L_J , \lambda ; B) , \end{align*}
by~\eqref{eq:nu-def}.
Then, since $\nu ( L_j , \lambda ; B)$ is $\cF_{\eta_k}$-measurable and  $\Prl ( J = j \mid \cF_{\eta_k} ) = L_j$, we obtain
\begin{equation}
\label{eq:j-nucleate4}
\Prl ( E' \cap \{ J = j \} \cap \{ \zeta^\star \in B_j \} \mid \cF_{\eta_k} )
=  L_j \nu ( L_j , \lambda ; B_j ) .\end{equation}
Then from~\eqref{eq:pi-def} with~\eqref{eq:E-star}, \eqref{eq:j-nucleate1}, \eqref{eq:j-nucleate4}, 
and the scaling property~\eqref{eq:nu-scaling}, 
there is a constant $C < \infty$ such that, a.s., for all $\lambda \in (0,\lambda_0]$,
\begin{equation}
\label{eq:pi-approx}
      \sup_{B_1 , \ldots, B_{I+1} } 
    \left| \sum_{j \in [I+1]} \pi_\lambda ( \cZ_I ; j  , B_j )  - \sum_{j \in [I+1]} L_j \nu ( L^3_j \lambda ; B_j) \right|
	\leq C \lambda^{3/2} M^{2-\eps} \sum_{j \in[I+1]} L_j^4  .
	\end{equation}
	 Now applying~\eqref{eq:lambda1}, we have from~\eqref{eq:pi-approx} that,
	for all $\lambda \in (0,\lambda_0]$,
	\begin{align}
	    \label{eq:pi-approx2}
	 & {} 
	  \sup_{B_1, \ldots, B_{I+1}} 
	  \left| \sum_{j \in [I+1]} \pi_\lambda ( \cZ_I ; j  , B_j )  -  \lambda  \sum_{j \in [I+1]} L_j^4 \Phi_1 (B_j) \right| \nonumber\\
	 & {} \qquad\qquad\qquad\qquad\qquad {} 	\leq C_{3} \lambda^{5/3} \sum_{j \in [I+1]} L_j^{6} + 
	C \lambda^{3/2} M^{2-\eps}  \sum_{j \in [I+1]} L_j^4 \nonumber\\
		 & {} \qquad\qquad\qquad\qquad\qquad {} \leq C \lambda^{3/2} M^{2-\eps} \sum_{j \in [I+1]} L_j^4 , 
		\end{align}
	redefining $C< \infty$ as necessary, 
		since 
$ L_j^{6}  
	  \leq  M^2 L_j^4$.
Taking all the $B_j = [0,1]$ in~\eqref{eq:pi-approx2},  and  using the fact that $\Phi_1 ([0,1]) = \mu$ (see Proposition~\ref{prop:Phi0-Phi1}), we get
	\begin{align}
	 \label{eq:Pi-approx}
\bigg| \Pi_\lambda (\cZ_I ) - \mu \lambda   \sum_{j \in [I+1]} L_j^4 \bigg| &
	 \leq C \lambda^{3/2}  M^{2-\eps} \sum_{j \in [I+1]} L_j^4 . \end{align}
		For the constants $\eps>0$ and $C<\infty$ as appearing in~\eqref{eq:pi-approx2} and~\eqref{eq:Pi-approx}, 
	let $\eps_0 = \frac{\mu}{2C}$ and 
	define the event $F := \{ \lambda^{1/2}  M^{2-\eps} \leq \eps_0 \}$.  
	Then, from~\eqref{eq:Pi-approx},	
	 \begin{equation}
	 \label{eq:Pi-on-F}
	 \Pi_\lambda (\cZ_I ) \geq \frac{\mu \lambda}{2}  \sum_{j \in [I+1]} L_j^4, \text{ on } F  .\end{equation}
	 Thus we have from~\eqref{eq:pi-approx2} and~\eqref{eq:Pi-on-F} that, for all  $\lambda \in (0,\lambda_0]$,
	 \[ 
	   \sup_{B_1, \ldots, B_{I+1}} 
	 \left| \frac{\sum_{j \in [I+1]} \pi_\lambda (\cZ_I; j, B_j)}{\Pi_\lambda (\cZ_I)} - \frac{\lambda \sum_{j \in [I+1]} L_j^4  \Phi_1 (B_j)}{ \Pi_\lambda (\cZ_I)} \right| \leq 
	C \lambda^{1/2}  M^{2-\eps} 
	, \text{ on } F .\]
	Moreover, since $\sum_{j \in [I+1]} L_j^4 \Phi_1 (B_j) \leq \mu \sum_{j \in [I+1]} L_j^4$, we have that, on $F$,
	\begin{align*} 
	\left|  \frac{\lambda \sum_{j \in [I+1]} L_j^4  \Phi_1 (B_j)}{ \Pi_\lambda (\cZ_I)} - 
	  \frac{\lambda \sum_{j \in [I+1]} L_j^4  \Phi_1 (B_j)}{  \mu \lambda \sum_{j \in [I+1]} L_j^4}
	\right|  \leq
	\left| \frac{\Pi_\lambda (\cZ_I) -  \mu \lambda \sum_{j \in [I+1]} L_j^4}{\Pi_\lambda (\cZ) } \right| 
	 \leq C \lambda^{1/2}  M^{2-\eps} ,  \end{align*}
	by~\eqref{eq:Pi-approx} and~\eqref{eq:Pi-on-F}. Combining the bounds in the last two displays
	we get
	\[  \sup_{B_1, \ldots, B_{I+1}} 
	\left| \frac{\sum_{j \in [I+1]} \pi_\lambda (\cZ_I; j, B_j)}{\Pi_\lambda (\cZ_I)} -\frac{ \sum_{j \in [I+1]} L_j^4 \Phi_1 (B_j)}{\mu \sum_{j \in [I+1]} L_j^4} \right| \leq
	C \lambda^{1/2}  M^{2-\eps}, \text{ on } F .\]
Lemma~\ref{lem:splitting-geometric}
	and the fact that  $\Phi_1(B) = \mu \Phi_0 (B)$ (Proposition~\ref{prop:Phi0-Phi1}) finish the proof.
\end{proof}

\section{Sparse deposition regime}
\label{sec:sparse}

In this section we focus on the $\lambda \to 0$ regime, and prove Theorem~\ref{thm:small-lambda}. 
Proposition~\ref{prop:splitting-distribution} refers
to the next nucleation after time $\eta_k$.
For the convergence of finite-dimensional distributions in Theorem~\ref{thm:small-lambda},
we need  to consider the next nucleation
after time $\nu_n$, the previous nucleation time.
This is the purpose of the next result.

\begin{lemma}
\label{lem:splitting-distribution-nucleations} 
For any $n \in \ZP$, we have
\[ \lim_{\lambda \to 0} \Expl 
 \sup_{B_1 , \ldots, B_{n+1} } 
\, \left| \sum_{j \in [n+1]} \Prl ( D_{\nu_n} ( j , B_j ) \mid \cF_{\nu_n} ) 
- \sum_{j \in [n+1]} \kappa_n ( \cZ_n ; j , B_j )
\right| =0.
\]
\end{lemma}

For the proof of this result, and later, it is useful to define
\begin{equation}
\label{eq:k-n-def}
k_n := \min \{ k \in \ZP : \eta_k \geq \nu_n \} , \text{ for } n \in \ZP.\end{equation}
Then $k_0 = 0$, and, for all $n\in \N$,  $\sigma_{k_n} < \nu_n \leq \eta_{k_n}$
for $k_n \in \N$. Note that $\eta_{k_n}$ is a stopping time, but $\sigma_{k_n}$, $n \in \N$, is not a stopping time.

\begin{proof}[Proof of Lemma~\ref{lem:splitting-distribution-nucleations}.]
With $k_n$ as defined at~\eqref{eq:k-n-def}, we have
\begin{align*}
\Prl ( D_{\nu_n} (j , B_j ) \mid \cF_{\nu_n} ) 
& = \Expl \bigl[ \Prl ( D_{\nu_n} ( j , B_j ) \mid \cF_{\eta_{k_n}} ) \bigmid \cF_{\nu_n} \bigr] .
\end{align*}
Let $F_{k}$ be the event that there are two or more nucleations in time interval $[\sigma_k, \eta_k]$.
On $F_{k_n}^\rc$, there is no nucleation in the interval $(\nu_n, \eta_{k_n}]$, and so $D_{\nu_n} (j , B ) = D_{\eta_{k_n}} (j , B )$.
Thus
\begin{equation}
\label{eq:eta-to-nu-F}
\sup_{B_1 , \ldots, B_{n+1} } 
 \left| \sum_{j \in [n+1]} \Prl ( D_{\nu_n} ( j , B_j ) \mid \cF_{\eta_{k_n}} ) 
- \sum_{j \in [n+1]} \kappa_n ( \cZ_n ; j , B_j )  \right| \leq C_2 \lambda^{1/2} + \2 { F_{k_n} } ,\end{equation}
by Proposition~\ref{prop:splitting-distribution}  and the fact that $I_{\eta_{k_n}} = n$
on~$F_{k_n}^\rc$.
If $G_k = \cup_{i=1}^{k} F_i$, then 
\begin{align}
\label{eq:eta-to-nu}
& 
\sup_{B_1 , \ldots, B_{n+1} }  \left| \sum_{j \in [n+1]} \Prl ( D_{\nu_n} ( j , B_j ) \mid \cF_{\nu_n} ) 
- \sum_{j \in [n+1]} \kappa_n ( \cZ_n ; j , B_j )   \right| \nonumber\\
& {} \qquad\qquad\qquad\qquad\qquad \qquad \qquad \qquad  {} \leq C_2 \lambda^{1/2} + \Prl ( G_{k_n}  \mid \cF_{\nu_n} ), 
\end{align}
by~\eqref{eq:eta-to-nu-F} and the fact that $\kappa_n ( \cZ_n ; j , B )$ is $\cF_{\nu_n}$-measurable.
Next, we bound $\Prl (F_k)$ and hence $\Prl (G_{k_n})$.
In order for there to be two (or more) nucleations in time interval $[\sigma_k, \eta_k]$,
there must be at least three deposition events during time interval  $(\sigma_k, \eta_k]$. 
Let $Z$ denote the number of deposition events 
during time $(\sigma_k, \sigma_k + \lambda^{-1/6}]$. Then,  
\[ \Prl (  F_{k} \mid \cF_{\sigma_k} ) \leq \Prl ( \eta_k - \sigma_k > \lambda^{-1/6} \mid \cF_{\sigma_k} )
+  \Prl ( Z  \geq 3 \mid \cF_{\sigma_k} ) ,
\]
and, since, given $\cF_{\sigma_k}$, $Z$ is Poisson with mean $\lambda^{5/6}$,
$\Prl (Z \geq 3 \mid \cF_{\sigma_k} ) \leq \lambda^{5/2}$, by~\eqref{eq:poisson-bound}.
Together with the tail bound in Lemma~\ref{lem:regeneration-times-finite},
this shows that $\Prl (  F_{k} \mid \cF_{\sigma_k} ) \leq  C \lambda^{5/2}$,
for some $C< \infty$ and all $\lambda \in (0,1]$, say. 
For fixed $n \in \ZP$ and $\eps>0$, choose $k$ sufficiently
large so that $\Prl ( k_n > k ) \leq \eps$.
Then
\[ \Prl ( G_{k_n} ) \leq \Prl (k_n > k ) + \sum_{i=1}^k \Prl (F_i) \leq \eps + C k \lambda^{5/2} .\]
Thus, for fixed $n$ and $\eps >0$, we may choose $\lambda$ small enough so that $\Prl ( G_{k_n} ) \leq 2\eps$. Hence $\lim_{\lambda \to 0} \Prl ( G_{k_n} ) =0$.
Together with~\eqref{eq:eta-to-nu}, this completes the proof.
\end{proof}

Now we are ready to prove Theorem~\ref{thm:small-lambda}.
Recall the definition of the splitting function $\Gamma_n$ from~\eqref{eq:splitting-map},
and that, from~\eqref{eq:interval-splitting} and~\eqref{eq:kappa-def}, the interval-splitting process $\cS = (\cS_0, \cS_1, \ldots)$ with parameters $r_0$ and $\Phi_0$ has
\[ \Pr ( \cS_{n+1} \in \Gamma_n ( \cS_n ; j, B ) \mid \cS_0, \cS_1, \ldots, \cS_n ) 
= \kappa_n ( \cS_n ; j, B)  .\]
Define the transition kernel $P_n : \Delta_n \times \cB_{n+1} \to [0,1]$ by
$\Pr ( \cS_{n+1} \in A \mid \cS_0, \cS_1, \ldots, \cS_n ) = P_n ( \cS_n , A )$, 
where $A \in \cB_{n+1}$ and $\cB_n$ denotes the Borel sets on $\Delta_n$.
Then, if $\Gamma^{-1}_n (z ; j,A) := \{ v \in [0,1] : \Gamma_n (z ; j,v) \in A \}$
for $A \in \cB_{n+1}$, we have from~\eqref{eq:kappa-def} that 
\[ P_n ( z, A) = \sum_{j \in [n+1]} \kappa_n ( z ; j , \Gamma^{-1}_n (z ; j, A ) ) =  \frac{\sum_{j \in[n+1]} ( z_{j} - z_{j-1} )^4 \Phi_0 ( \Gamma^{-1}_n (z ; j, A ) ) }{\sum_{j \in [n+1]} ( z_{j} - z_{j-1} )^4} ,\]
where $z = (z_{0}, \ldots, z_{n+1} ) \in \Delta_n$.

\begin{proof}[Proof of Theorem~\ref{thm:small-lambda}.]
Define for $n \in \N$ and $A_1 \in \Delta_1, \ldots, A_n \in \Delta_n$,
\[ K_n (A_1, \ldots, A_n) := \Pr ( \cS_1 \in A_1, \ldots, \cS_n \in A_n) ,\]
where $\cS$ is the interval-splitting process with parameters $r_0$ and $\Phi_0$.
Then $K_1 (A) = P_0 ( \cZ_0, A)$ and, for $n \in \N$,
\begin{align}
\label{eq:Kn}
    K_{n+1} (A_1, \ldots, A_{n+1} ) & = \Exp \bigl[ \2 {\{ \cS_1 \in A_1 , \ldots, \cS_n \in A_n \}} \Pr ( \cS_{n+1} \in A_{n+1} \mid \cS_0, \ldots, \cS_n ) \bigr] \nonumber\\
    & =  \Exp \bigl[ \2 {\{ \cS_1 \in A_1 , \ldots, \cS_n \in A_n \}} P_n (\cS_n , A_{n+1} ) \bigr] \nonumber\\
    & = \int_{A_n} K_n (A_1, \ldots, A_{n-1}, \ud z ) P_n (z , A_{n+1} ). 
\end{align}
We wish to prove that for any $n \in \N$,
\begin{align}
\label{eq:tv-fdd}
  \lim_{\lambda \to 0}  \sup_{A_1,\ldots,A_n} \left| \Prl ( \cZ_1 \in A_1, \ldots, \cZ_n \in A_n )
    - K_n ( A_1, \ldots, A_n) \right| = 0,
\end{align}
the supremum over $A_1 \in \cB_1, \ldots, A_n \in \cB_n$.
We establish~\eqref{eq:tv-fdd} by induction on $n$. First, 
\begin{align}
\label{eq:kernel-to-kernel}
 & {}   \left| \Prl ( \cZ_{n+1} \in A \mid \cF_{\nu_n} ) - P_n (\cZ_n , A) \right|  \nonumber\\
 & {}  \quad {} 
    = \left| \sum_{j \in [n+1]} \Prl \bigl( D_{\nu_n} (j , \Gamma_n^{-1} ( \cZ_n ; j, A ) ) \bigmid \cF_{\nu_n} \bigr) - \sum_{j \in [n+1]} \kappa_n ( \cZ_n ; j,  \Gamma_n^{-1} ( \cZ_n ; j, A ) ) \right|.
\end{align}
By~\eqref{eq:kernel-to-kernel}, 
a consequence of 
Lemma~\ref{lem:splitting-distribution-nucleations} is that, for all $n \in \ZP$,
\begin{equation}
    \label{eq:induction-basis}
 \lim_{\lambda \to 0} \Expl  \sup_{A \in \cB_{n+1}} \left| \Prl ( \cZ_{n+1} 
\in A \mid \cF_{\nu_n} ) 
- P_n (\cZ_n , A ) \right| =0.
\end{equation}
In particular, taking $n=0$ in~\eqref{eq:induction-basis},
we get the $n=1$ case of~\eqref{eq:tv-fdd},
the basis for the induction.
For the inductive step, suppose that~\eqref{eq:tv-fdd} holds for some given $n \in \N$.
Then
\begin{align*}
    \Prl ( \cZ_1 \in A_1, \ldots , \cZ_{n+1} \in A_{n+1} ) & = \Expl \bigl[
    \2 {\{  \cZ_1 \in A_1, \ldots, \cZ_n \in A_n \}} \Prl ( \cZ_{n+1} \in A_{n+1} \mid \cF_{\nu_n} ) \bigr] .
\end{align*}
By~\eqref{eq:induction-basis}, it follows that
\[ \lim_{\lambda \to 0} \sup_{A_1, \ldots, A_{n+1}}  \!\!\!
\left| \Prl ( \cZ_1 \in A_1, \ldots , \cZ_{n+1} \in A_{n+1} ) \!
- \! \Expl\! \bigl[
   \2 {\{  \cZ_1 \in A_1, \ldots, \cZ_n \in A_n\}} P_n (\cZ_n, A_{n+1} ) \bigr]\! \right|\! = 0 .\]
 Now by inductive hypothesis~\eqref{eq:tv-fdd}
 and the relationship between total-variation distance and coupling, 
for any $\eps>0$ we can choose $\lambda >0$ sufficiently small, and work on  a suitable probability space
in which $\Pr ( (X_1, \ldots, X_n) \neq (Y_1, \ldots, Y_n) ) \leq \eps$ and
$(X_1, \ldots, X_n)$ has law $\Prl ( \cZ_1 \in \, \cdot \, , \ldots, \cZ_n \in \, \cdot\, )$
and $(Y_1, \ldots, Y_n)$ has law $K_n$. Hence
\[ 
\sup_{A_1,\ldots,A_{n+1}} \!\!\! \left| \Expl \bigl[
   \2 {\{  \cZ_1 \in A_1, \ldots, \cZ_n \in A_n\}} P_n (\cZ_n, A_{n+1} ) \bigr] \!
    - \! \Exp \bigl[ \2 {\{  Y_1 \in A_1, \ldots, Y_n \in A_n \} } P_n ( Y_n , A_{n+1}) \bigr]\! \right| \!
    \leq \eps ,\]
    and the expectation involving the $Y_i$s is, by~\eqref{eq:Kn}, equal to $K_{n+1} (A_1, \ldots, A_{n+1})$. This completes the inductive step.
\end{proof}

\section{Fixed-rate deposition regime}
\label{sec:fixed-rate-regime}

In this section, we will prove Theorem~\ref{thm:gap-statistics}, which says that,
roughly speaking, the long-term asymptotics of the fixed-$\lambda$
process are governed by the interval-splitting process that arises as the $\lambda\to 0$ limit
established in Theorem~\ref{thm:small-lambda}. The intuition for this is that as time goes on, the gaps get smaller
and so capture of active particles by existing islands gets faster, which has a similar effect
as driving down the deposition rate.

The proof of Theorem~\ref{thm:gap-statistics} 
 uses coupling,   based 
on the following result.

\begin{proposition}
\label{prop:fixed-lambda-approx}
For any $\lambda \in (0,\infty)$,
\[
 \Expl 
\sum_{n \in \ZP}
\sup_{B_1, \ldots, B_{n+1}} \, \left| 
\sum_{j \in [n+1]}
\Prl ( D_{\nu_n} (j, B_j) \mid \cF_{\nu_n} ) - 
\sum_{j \in [n+1]} \kappa_n (\cZ_n ; j , B_j )  
\right|    < \infty. \]
\end{proposition}

To obtain Proposition~\ref{prop:fixed-lambda-approx}, we need an improved
 version of the bound in Lemma~\ref{lem:splitting-distribution-nucleations},
and this requires control of the chance
of additional nucleations occurring in time interval $(\nu_n,\eta_{k_n}]$, where $k_n$ is as defined at~\eqref{eq:k-n-def}. This is the purpose of the next lemma.

Let $\chi_k$ denote the number of nucleations during time interval $[\sigma_k,\eta_k]$.

\begin{lemma}
\label{lem:multiple-nucleations}
For any $\lambda_0 \in (0,\infty)$, there exists $C_5 = C_5 (\lambda_0) < \infty$ such that, for all $\lambda \in (0,\lambda_0]$
and all $k \in \ZP$,
\[ \Expl \bigl( \chi_{k+1} \1 { \chi_{k+1} \geq 2 } \bigmid \cF_{\eta_k} \bigr) \leq C_5 \lambda^{5/2} M_{\eta_k}^{13/2}   , \text{ $\Prl$-a.s.} \]
\end{lemma}
\begin{proof}
Fix $k \in \ZP$.
Write $I = I_{\eta_k}$, $J = J_k$, $M= M_{\eta_k}$, and $L_j = L_{I,j}$ for $1 \leq j \leq I+1$.
Fix $s >0$. 
During time interval  $(\sigma_{k+1}, \sigma_{k+1} + s]$,
let $Y$ denote the number of depositions in gap $J$, and let $Y'$ denote the number   elsewhere in the interval. Given $\cF'_{\eta_k}$, $Y$ and $Y'$ are independent
Poisson random variables with $\Exp ( Y \mid \cF'_{\eta_k} ) = \lambda s L_J \leq \lambda s M$
and $\Exp ( Y' \mid \cF'_{\eta_k} ) \leq \lambda s$. 
Since each nucleation consumes two active particles in the same interval, 
in order for there to be (at least) two nucleations
during time interval  $(\sigma_{k+1}, \sigma_{k+1} + s]$,
we must either have (i) at least $3$ depositions in gap~$J$, (ii) at least one deposition in gap~$J$, and at least two depositions elsewhere,
or (iii) no depositions in gap~$J$, and at least $4$ depositions elsewhere. In any case, the number of nucleations is not more than the number of depositions.
Hence
\begin{align}
\label{eq:two-nucleations}
 & {} \Expl ( \chi_{k+1} \1 { \chi_{k+1} \geq 2 , \, \eta_{k+1} - \sigma_{k+1} \leq s }  \mid \cF'_{\eta_k} ) \nonumber\\
& \leq   \Expl (  ( Y + Y' ) \1 { Y \geq 3 } \mid \cF'_{\eta_k} ) + \Expl (  ( Y + Y' ) \1 { Y \geq 1, Y' \geq 2 } \mid \cF'_{\eta_k} ) \nonumber\\
& {} \qquad {} +  \Expl (  Y'  \1 { Y' \geq 4 } \mid \cF'_{\eta_k} ) 
\nonumber\\
& \leq 
\Expl ( Y \1 { Y \geq 3 } \mid \cF'_{\eta_k} ) + \Expl ( Y' \mid \cF'_{\eta_k} ) \Prl ( Y \geq 3 \mid \cF'_{\eta_k} ) \nonumber\\
& {} \qquad {} + 2 \Expl (    Y \1 { Y \geq 1 } \mid \cF'_{\eta_k} ) \Expl ( Y' \1 { Y' \geq 2 } \mid \cF'_{\eta_k} ) +  \Expl (  Y'  \1 { Y' \geq 4 } \mid \cF'_{\eta_k} ) \nonumber\\
& \leq ( \lambda s M)^3 + \lambda s \cdot ( \lambda s M)^3 + 2 \lambda s M \cdot (\lambda s)^2 + (\lambda s)^4 \nonumber\\
& \leq 3 \lambda^3 s^3 M + 2 \lambda^4 s^4,  \end{align}
by~\eqref{eq:poisson-bound} and the fact that $M \leq 1$. 
On the other hand, by Cauchy--Schwarz,
\[ \Expl \bigl( \chi_{k+1} \1 { \eta_{k+1} - \sigma_{k+1} > s } \bigmid \cF_{\eta_k} \bigr)
\leq \bigl( \Expl ( \chi_{k+1}^2 \mid \cF_{\eta_k} ) \bigr)^{1/2} \bigl( \Prl ( \eta_{k+1} - \sigma_{k+1} > s \mid \cF_{\eta_k} ) \bigr)^{1/2} .\]
If $Y''$ denotes the number of Poisson arrivals during time interval $(\sigma_{k+1}, \sigma_{k+1} + x ]$,
then  
\begin{align*}
 \Prl ( \chi_{k+1} \geq 3 \lceil \lambda_0 x \rceil   \mid \cF_{\eta_k} ) & \leq 
\Pr ( \eta_{k+1} - \sigma_{k+1} \geq x \mid \cF_{\eta_k} )
+ \Pr ( Y'' \geq  3 \lceil \lambda_0 x \rceil    \mid \cF_{\eta_k} ) . \end{align*} 
The Poisson variable $Y''$ has $\Expl ( \re^{Y''} \mid \cF_{\eta_k} ) \leq \re^{2 \lambda_0 x}$, a.s., so,
 by Lemma~\ref{lem:regeneration-times-finite} and
 Markov's inequality,
$\Prl ( \chi_{k+1} \geq 3 \lceil \lambda_0 x \rceil   \mid \cF_{\eta_k} )  \leq 
C_1 \exp ( - c x^{1/2} ) +   \exp ( - \lambda_0 x )$, where $c >0$.
It follows that $\Expl ( \chi_{k+1}^2 \mid \cF_{\eta_k} ) \leq C$
for some $C < \infty$. With Lemma~\ref{lem:regeneration-times-finite}, this shows that
\[ \Expl \bigl( \chi_{k+1} \1 { \eta_{k+1} - \sigma_{k+1} > s } \bigmid \cF_{\eta_k} \bigr)
\leq C \exp ( - c M^{-1} s^{1/2} ) ,\]
where the constants $C < \infty$ and $c >0$ depend on $\lambda_0$.
Thus we obtain
\[ \Expl \bigl( \chi_{k+1} \1 { \chi_{k+1} \geq 2} \bigmid \cF_{\eta_k} \bigr) \leq 3 \lambda^3 s^3 M + 2 \lambda^4 s^4 + C \exp ( - c M^{-1} s^{1/2} ) ,\]
provided $\lambda  \leq \lambda_0$. 
Taking $s =  \lambda^{-1/6} M^{11/6}$ we get the result.
\end{proof}

The next result shows that $M_{\nu_n} \to 0$, a.s., and gives some quantification of the rate. One expects that $M_{\nu_n}$, the
length of the largest gap when there are $n$ interior islands, is not much greater than~$1/n$,
and Lemma~\ref{lem:max-vanishes} is, in a rough sense, a bound of $O (n^{\eps-(3/4)})$.
On the basis of the upper tail of $g_0$ in Theorem~\ref{thm:gap-statistics},
we conjecture that the correct order for~$M_{\nu_n}$ is $( \log n)^{1/4} / n$.

\begin{lemma}
\label{lem:max-vanishes}
For any $\lambda \in (0,\infty)$ and any $\gamma > 4/3$, we have 
$\Expl \sum_{n =0}^\infty M_{\nu_n}^\gamma < \infty$.
\end{lemma}
\begin{proof}
Consider the process $(W_t , t \in \RP)$ defined by
\[ W_t  = \sum_{i \in [I_t+1]} L^\alpha_{I_t,i} ,\]
where $\alpha > 1$. Note that $0 < W_t \leq W_s \leq W_0 =1$ for $0 \leq s \leq t < \infty$.
Lemma~\ref{lem:intervals-shrink} says
\[ 
\Prl ( \{ \alpha_k = 1\} \cap D_{\eta_k} (j,B_0) \mid \cF_{\eta_k} ) \geq \eps_0   L_{I_{\eta_k},j}^4 ,\]
where $B_0 = [\frac{1}{8},\frac{7}{8}]$ and $\eps_0 = \eps_0 (\lambda) >0$. On $\{ \alpha_k = 1\} \cap D_{\eta_k} (j,B_0)$, nucleation
occurs during time interval $[\eta_k,\eta_{k+1}]$ 
at relative location $v \in B_0$
 in gap $j \in [I_{\eta_k} +1]$, and any subsequent nucleation
before time $\eta_{k+1}$ only decreases $W_{\eta_{k+1}}$.
Hence,
on $\{ \alpha_k = 1\} \cap D_{\eta_k} (j,B_0)$,
\begin{equation}
\label{eq:W-change}
 W_{\eta_{k+1}} - W_{\eta_k} \leq  - \inf_{v \in B_0} \Lambda ( v, \alpha) L_{I_{\eta_k},j}^\alpha , \text{ where } \Lambda (v,\alpha ) := 1 - v^\alpha - (1-v)^\alpha .\end{equation}
 Here $\inf_{v \in B_0} \Lambda ( v, \alpha)  = \delta >0$
 depending only on $\alpha >1$. From Lemma~\ref{lem:intervals-shrink} with~\eqref{eq:W-change}, 
\begin{align}
\label{eq:W-increment}
 \Expl ( W_{\eta_{k+1}} - W_{\eta_k} \mid \cF_{\eta_k} )   \leq - \delta \eps_0  \sum_{i \in [I_{\eta_k}+1]}  L_{I_{\eta_k},i}^{4+\alpha}  
  \leq - \delta \eps_0   M_{\eta_k}^{4+\alpha} .
\end{align}
Taking expectations and summing, since $W_0 = 1$, we obtain, for every~$\alpha >1$,
\begin{align}
\label{eq:M-first-bound}    \Expl \sum_{k \in \ZP} M_{\eta_{k}}^{4+\alpha} \leq  \frac{1}{\delta \eps_0 } < \infty .\end{align}
In particular~\eqref{eq:M-first-bound} shows that $\lim_{t \to \infty} M_t = 0$, a.s.
As in the proof of Lemma~\ref{lem:splitting-distribution-nucleations},
let $F_k = \1 { \chi_k \geq 2 }$, the event that there are two or move nucleations during $[\sigma_k,\eta_k]$.
Recall the definition of $k_n$ at~\eqref{eq:k-n-def}.
From Lemma~\ref{lem:multiple-nucleations} and~\eqref{eq:M-first-bound} (take $\alpha = 5/2$), we have
\begin{equation}
\label{eq:F-k-n-summable}
\sum_{n \in \N} \Prl ( F_{k_n} )
= \sum_{k \in \N } \Expl \sum_{n : k_n = k} \1 { \chi_k \geq 2 }
= \Expl \sum_{k \in \N}  \chi_k \1 { \chi_k \geq 2 } 
 < \infty ,\end{equation}
since $k_n = k \in \N$ if and only if $\sigma_k \leq \nu_n \leq \eta_k$.

Now we extend the argument to get the statement in the lemma.  Take $\alpha = 4$ in the definition of~$W_t$. 
For $k \in \ZP$ let $\rho_k = \min \{ n \in \N : \nu_n > \eta_k \}$.
As above, we have that $\Lambda ( v, 4) \geq \delta >0$ for all $v \in B_0$.
On the event $D_{\eta_k} (j , B_0 )$, we have $W_{\nu_{\rho_k}} - W_{\eta_k} \leq - \delta L_{I_{\eta_k,j}}^4$. 
If $\nu_1$ and $\nu_2$ are finite measures on a countable set $S$, 
then, for $A = \{ j \in S : \nu_1 (j) \geq \nu_2 (j) \}$, supposing, without loss of generality, that $\nu_1(S) \geq \nu_2 (S)$,
\begin{equation}
    \label{eq:total-variation}
 \sum_{j \in S} \bigl| \nu_1 (j) - \nu_2(j) \bigr| = 2 ( \nu_1 (A) - \nu_2 (A) ) + \nu_2 (S) - \nu_1 (S) \leq 2 \sup_{\cJ \subseteq S} \bigl| \nu_1 (\cJ) - \nu_2 (\cJ) \bigr| . \end{equation}
From Proposition~\ref{prop:splitting-distribution}, taking $B_j = B$ for all $j \in \cJ$ and $B_j = \emptyset$ for $j \notin \cJ$,  we obtain
\[ \sup_{B \in \cB} \left| \sum_{j \in \cJ} \Prl ( D_{\eta_k} (j, B ) \mid \cF_{\eta_k} ) - \sum_{j \in \cJ} \kappa_{I_{\eta_k}} ( \cZ_{I_{\eta_k}} ; j, B ) \right| \leq C_2 \lambda^{1/2} M_{\eta_k}^{3/2}. \] 
Then using~\eqref{eq:kappa-def}
and~\eqref{eq:total-variation}, it follows that
\[ \sum_{j \in [I_{\eta_k}+1]} \left| \Prl ( D_{\eta_k} (j , B_0 ) \mid \cF_{\eta_k} ) L_{I_{\eta_k},j}^4
- \frac{L_{I_{\eta_k,j}}^8 \Phi_0 (B_0)}{\sum_{i \in[I_{\eta_k}+1]} L_{I_{\eta_k},i}^4}
\right| \leq 2 C_{2} \lambda^{1/2} M_{\eta_k}^{11/2} .\]
Consequently, for $\eps = \Phi_0 (B_0) > 0$,
\begin{align*} 
\Expl \bigl( W_{\nu_{\rho_k}} - W_{\eta_k} \bigmid \cF_{\eta_k} \bigr)
& \leq 2 C_{2} \lambda^{1/2} W_{\eta_k}^{11/8} - \delta \eps \frac{\sum_{i \in [I_{\eta_k} +1]} L_{I_{\eta_k},i}^{8}}
{\sum_{i \in [I_{\eta_k} +1]} L_{I_{\eta_k},i}^{4}}
.\end{align*}
Let $A_n$ denote the event that $\nu_{n+1} > \eta_{k_n}$; 
on $A_n$, $I_{\eta_{k_n}} = I_{\nu_n} = n$ and $\rho_{k_n} = n+1$.
Then taking $k = k_n$ (noting that $\eta_{k_n}$ is a stopping time) and using the monotonicity of $W_t$,  
\begin{align*} 
\Expl \bigl( W_{\nu_{n+1}} - W_{\nu_n} \bigmid \cF_{\eta_{k_n}} \bigr)
& \leq 
\Expl \bigl( W_{\nu_{\rho_{k_n}}} - W_{\eta_{k_n}} \bigmid \cF_{\eta_{k_n}} \bigr) \2 { A_n }  \\
& \leq 2 C_{2} \lambda^{1/2} W_{\nu_n}^{11/8} - \delta \eps \frac{\sum_{i \in [n +1]} L_{n,i}^{8}}
{\sum_{i \in [n +1]} L_{n,i}^{4}} \2 {A_n }
. \end{align*}
Since $\sum_{i\in [n+1]} L_{n,i} =1$, Jensen's inequality gives
$\sum_{i \in [n+1]} L_{n,i}^8 \geq W_{\nu_n}^{7/3}$. Hence
\[ \Expl ( W_{\nu_{n+1}} - W_{\nu_n} \mid \cF_{\eta_{k_n}}  )  \leq 
C W_{\nu_n}^{11/8}  
-  \delta \eps W_{\nu_n}^{4/3}  \2 {A_n } , \]
where $C< \infty$. Since $4/3 < 11/8$, and $W_{\nu_{n+1}} \leq W_{\nu_n} \leq 1$,
 there exists $\eps>0$ such that a.s., 
\begin{equation}
\label{eq:W-n-drift}
  \Expl ( W_{\nu_{n+1}} - W_{\nu_n} \mid \cF_{\eta_{k_n}}  )  \leq 
   -  \eps W_{\nu_n}^{4/3} \1 { W_{\nu_n} < \eps } \2 {A_n } . \end{equation}
Let $\beta \in (0,1)$. Then $(1+x)^\beta \leq 1 + \beta x$ for all $ x \in [-1,0]$, so
\begin{align*}
\Expl ( W_{\nu_{n+1}}^\beta - W_{\nu_n}^\beta \mid \cF_{\eta_{k_n}}  ) 
& = W_{\nu_n}^\beta \Expl \biggl[   \Bigl(1 + \frac{W_{\nu_{n+1}} - W_{\nu_n}}{W_{\nu_n}} \Bigr)^\beta  - 1 \biggmid \cF_{\eta_{k_n}}  \biggr] \\
& \leq \2 {A^\rc_n }  - \beta \eps W_{\nu_n}^{\beta+(1/3)} \1 { W_{\nu_n} < \eps } , \end{align*}
by~\eqref{eq:W-n-drift}.
Taking expectations, summing, and using the fact that $W_0 = 1$, we get
\[ \beta \eps \sum_{m=0}^{n-1} \Expl \bigl( W_{\nu_m}^{\beta+(1/3)} \1 { W_{\nu_m} < \eps } \bigr) \leq 1 + \sum_{m=0}^{n-1} \Prl ( A_m^\rc) , \text{ for all } n \in \N.\]
Now $A_n^\rc \subseteq F_{k_n}$, so, by~\eqref{eq:F-k-n-summable}, $\sum_{n=0}^\infty \Prl ( A_n^\rc) < \infty$. Thus
\begin{equation}
\label{eq:W-n-beta-sum1}
 \sum_{n =0}^\infty \Expl \bigl( W_{\nu_n}^{\beta+(1/3)} \1 { W_{\nu_n} < \eps } \bigr) < \infty .\end{equation}
On the other hand, since $W_{\nu_n} \in [0,1]$ and $\beta >0$, 
\begin{align*}
\sum_{n =0}^\infty \Expl \bigl( W_{\nu_n}^{\beta+(1/3)} \1 { W_{\nu_n} \geq  \eps } \bigr) 
&  \leq \eps^{-2} \sum_{n =0}^\infty \Expl \bigl( W_{\nu_n}^{7/3} \bigr)  \\
& \leq \eps^{-2} \Expl \sum_{k \in \ZP} \sum_{n: k_n = k}  \bigl( W^{7/3}_{\eta_k} \1 { \chi_k = 1} + \1 { \chi_k \geq 2 } \bigr) ,\end{align*}
since $W_{\nu_n} = W_{\eta_{k_n}}$ on $\{ \chi_{k_n} = 1 \}$. Here we have that
\begin{align*}
  \Expl \sum_{k \geq 1} \sum_{n: k_n = k} \bigl( W^{7/3}_{\eta_k} \1 { \chi_k = 1} + \1 { \chi_k \geq 2 } \bigr)
	& \leq \Expl \sum_{k \geq 1} W_{\eta_k}^{7/3} + \Expl \sum_{k \geq 1} \chi_k \1{ \chi_k \geq 2} \\
	& \leq \Expl \sum_{k \geq 1} M_{\eta_k}^{7} + C \Expl \sum_{k \geq 1} M_{\eta_k}^{13/2} ,
\end{align*}
	by Lemma~\ref{lem:multiple-nucleations} and the fact that $W_{\eta_k} \leq M_{\eta_k}^3$.
	Then, by~\eqref{eq:M-first-bound}, since $7 > \frac{13}{2} > 5$,
\begin{equation}
\label{eq:W-n-beta-sum2}
\sum_{n =0}^\infty \Expl \bigl( W_{\nu_n}^{\beta+(1/3)} \1 { W_{\nu_n} \geq  \eps } \bigr) < \infty .\end{equation}
	Combining~\eqref{eq:W-n-beta-sum1} and~\eqref{eq:W-n-beta-sum2}, since $W_{\nu_n} \geq M_{\nu_n}^4$  we conclude
\[ \Expl \sum_{n \in \ZP} M_{\nu_n}^{4\beta + (4/3)} \leq   \Expl \sum_{n \in \ZP} W_{\nu_n}^{\beta + (1/3)} < \infty ,   \]
which gives the result, since $\beta \in (0,1)$ was arbitrary.
\end{proof}

Now we can give the proof of Proposition~\ref{prop:fixed-lambda-approx}.

\begin{proof}[Proof of Proposition~\ref{prop:fixed-lambda-approx}.]
On $F^\rc_{k_n}$,  $I_{\eta_{k_n}} = I_{\nu_n} = n$, and so, by Proposition~\ref{prop:splitting-distribution},
\[\sup_{B_1 , \ldots, B_{n+1} } \, \left| 
\sum_{j \in [n+1]}
\Prl ( D_{\nu_n} (j, B_j ) \mid \cF_{\eta_{k_n}} )   
- \sum_{j \in [n+1]} \kappa_n (\cZ_n ; j , B_j )
\right| \leq \2 { F_{k_n} } +
C_{2} \lambda^{1/2} M^{3/2}_{\eta_{k_n}} .	\] 
Since $M_{\eta_{k_n}} \leq M_{\nu_n}$, it follows on taking conditional expectations given~$\cF_{\nu_n}$ that
\begin{align*}
& {} \sup_{B_1 , \ldots, B_{n+1} } \, \left| 
\sum_{j \in [n+1]}
\Prl ( D_{\nu_n} (j, B_j) \mid \cF_{\nu_n} ) - 
\sum_{j \in [n+1]}
\kappa_n (\cZ_n ; j , B_j ) 
\right| \\
& {} \qquad 
 \qquad 
  \qquad 
   \qquad 
    \qquad {}
 \leq \Prl ( F_{k_n} \mid \cF_{\nu_n} ) +
C_{2} \lambda^{1/2} M^{3/2}_{\nu_n} .	\end{align*}
Here we know from Lemma~\ref{lem:max-vanishes} that $\Expl \sum_{n\in \ZP}  M^{3/2}_{\nu_n} < \infty$,  and we know from~\eqref{eq:F-k-n-summable} that $\sum_{n \in \ZP} \Prl ( F_{k_n}  ) < \infty$. 
The result follows. 
\end{proof}

Finally we present the proof of Theorem~\ref{thm:gap-statistics}. This is based on a coupling argument,
using Proposition~\ref{prop:fixed-lambda-approx},
together with appropriate asymptotic results for interval-splitting processes,
which we defer to~\S\ref{sec:gap-tails}.

\begin{proof}[Proof of Theorem~\ref{thm:gap-statistics}.]
Recall from the proof of Theorem~\ref{thm:small-lambda} that $P_n ( \cS_n , A) = \Pr ( \cS_{n+1} \in A \mid \cS_0, \ldots, \cS_n )$, $A \in \cB_{n+1}$,
is the kernel associated with the 
the interval-splitting process $\cS = (\cS_0, \cS_1, \ldots)$ with parameters $r_0$ and $\Phi_0$.
 Proposition~\ref{prop:fixed-lambda-approx} with~\eqref{eq:kernel-to-kernel} shows that
\begin{equation}
\label{eq:fixed-lambda-approx} \Expl 
\sum_{n \in \ZP} \sup_{ A \in \cB_{n+1}}  \, \bigl| \Prl ( \cZ_{n+1} \in A \mid \cF_{\nu_n} ) - P_n (\cZ_n, A) \bigr| < \infty .\end{equation}
Let $\cF^{\cZ}_n := \sigma (\cZ_0, \cZ_1, \ldots, \cZ_n)$.
Since
$\Prl ( \cZ_{n+1} \in A \mid \cF^{\cZ}_n ) = \Expl [ 
\Prl (\cZ_{n+1} \in A \mid \cF_{\nu_n} ) \mid \cF^{\cZ}_n ]$,
\begin{align}
\label{eq:fixed-lambda-approx-Z}
&  {} \Expl 
\sum_{n \in \ZP} \sup_{ A \in \cB_{n+1}} \, \bigl| \Prl ( \cZ_{n+1} \in A \mid \cF^{\cZ}_{n} ) - P_n (\cZ_n, A) \bigr| \\
& {} \quad {} \leq \Expl 
\sum_{n \in \ZP} \sup_{ A \in \cB_{n+1}} \Expl \Bigl[ \bigl| \Prl ( \cZ_{n+1} \in A \mid \cF_{\nu_n} ) - P_n (\cZ_n, A) \bigr| \Bigmid \cF_n^{\cZ} \Bigr]  < \infty
 ,
\nonumber
\end{align}
 by~\eqref{eq:fixed-lambda-approx}. 
Fix $n_0 \in \ZP$. We couple $(\cZ_0,\cZ_1,\ldots)$ and $(\cS_{n_0}, \cS_{n_0+1}, \ldots)$, an interval-splitting process with parameters $r_0$ and $\Phi_0$ and initial configuration $\cS_{n_0} = \cZ_{n_0}$.
On a common probability space, where we still denote probabilities by $\Prl$,
construct $(\cZ_0, \ldots, \cZ_{n_0})$ according to the law
given by $\Prl$, and then set $\cS_{n_0} = \cZ_{n_0}$.
Let $n \geq n_0$. 
Let $\cF_n^{\cZ,\cS} = \sigma (\cZ_0,  \ldots, \cZ_n, \cS_{n_0} , \ldots , \cS_n)$.
Given $\cF_n^{\cZ,\cS}$, 
if $\cZ_n \neq \cS_n$, then generate $\cS_{n+1}$ independently of $\cZ_{n+1}$ according to $P_n (\cS_n , \, \cdot \,)$.
If $\cZ_n = \cS_n$, then 
generate $(\cZ_{n+1}, \cS_{n+1})$ by maximal coupling of
 $\Prl ( \cZ_{n+1} \in \, \cdot \, \mid \cF^{\cZ}_{n} )$
and $P_n (\cZ_n , \, \cdot \,)$.
Then
\begin{align}
\label{eq:coupling-estimate}
   \Prl (  \cZ_{n} = \cS_{n} \text{ for all } n \geq n_0  ) \geq 1 - \eps_{n_0} ,\end{align}
where 
\begin{align*}
\eps_{n_0} = \Expl  \sum_{n \geq n_0} \sup_{ A \in \cB_{n+1}} \left| \Prl ( \cZ_{n+1} \in A \mid \cF^{\cZ}_{n} ) - P_n (\cZ_n, A) \right|   .\end{align*}
By~\eqref{eq:fixed-lambda-approx-Z}, for any $\eps>0$ we can choose $n_0$ large enough  that 
$\eps_{n_0} < \eps$; fix such an $n_0$.

We apply Theorem~\ref{thm:general-splitting-distribution} to the interval-splitting
process with parameters $r_0$ and $\Phi_0$;  in the hypotheses
 we take $\alpha = 4$ and $\beta =2$,
using Lemma~\ref{lem:psi-near-zero} for the behaviour of $\phi_0$~near zero.
Theorem~\ref{thm:general-splitting-distribution}(i) shows that
$\Prl ( n^{-1} C_n (x) \to x \mid \cS_{n_0} = z) = 1$ for all $z$, and hence
$\Prl ( n^{-1} C_n (x) \to x \mid \cF_{n_0}^{\cZ,\cS} ) = 1$, a.s.
On the event $ \cZ_{n} = \cS_{n} \text{ for all } n \geq n_0$, we have $n^{-1} | N_n (x) - C_n (x) | \to 0$, a.s.
Hence
\begin{align*}
 \Prl ( n^{-1} N_n (x) \to x \mid \cF_{n_0}^{\cZ,\cS} ) 
& \geq \Prl ( n^{-1} C_n (x) \to x, \,  \cZ_{n} = \cS_{n} \text{ for all } n \geq n_0 \mid \cF_{n_0}^{\cZ,\cS} ) \\
& = \Prl (  \cZ_{n} = \cS_{n} \text{ for all } n \geq n_0 \mid \cF_{n_0}^{\cZ,\cS} ) .\end{align*}
Taking expectations and using~\eqref{eq:coupling-estimate}, we get
$\Prl ( n^{-1} N_n (x) \to x ) \geq 1 - \eps_{n_0} \geq 1 - \eps$.
Since $\eps >0$ was arbitrary, we establish part~(i) of the theorem. 
The a.s.~convergence in part~(ii) is deduced from Theorem~\ref{thm:general-splitting-distribution}(ii) in a similar way,
and   convergence of $\Prl ( \tilde L_n \leq x) = \Expl \cE_n (x)$ follows from the bounded convergence theorem.
The  asymptotics for $g_0$ follow from  Theorem~\ref{thm:general-splitting-distribution}(iii), noting that
there $\alpha = 4$, $\beta =2$, and $a=0$.
\end{proof}

\section{Brownian motion exiting a right-angled triangle}
\label{sec:triangle}

This section provides a proof of Proposition~\ref{prop:Phi0-Phi1}.
 Recall the notation $S =
\partial[0,1]^2$ and $D = \{(x,y) \in [0,1]^2 : x=y\}$ for the boundary and diagonal of
the unit square, and the definition of $H(u,v;B)$ from~\eqref{eq:H-def}.  
Let $U := \{(x,y) \in [0,1]^2 : x\geq y\}$,
the right-angled triangle with side-lengths $1,1,\sqrt{2}$. Then for $B \in \cB$ we can write
\begin{align*} H(u,v;B) = \Pr ( \tau_{\partial U} = \tau_{D} , \, W^{(1)}_{\tau_{D}} \in B \mid W_0 = (u,v) ) , \text{ for } (u,v) \in U;
\end{align*}
 the symmetry $H(u,v;B) = H(v,u;B)$ gives $H(u,v;B)$ for all $(u,v) \in [0,1]^2$. 
An old result of Smith \& Watson~\cite{sw} states that the probability that planar Brownian
motion started from a \emph{uniform random} point in $U$ exits via the diagonal is given by
\begin{align}
\label{eq:sw}   2 \int_U H(u,v;[0,1]) \ud u \ud v = \int_{[0,1]^2} H(u,v;[0,1]) \ud u \ud v  & =   1 - \frac{16}{\pi^3}\sum_{\text{$n$ odd}} \frac{\coth(
    n\pi /2)}{n^3} \\
		& \approx 0.41063 \text{ to 5 decimal places.} \nonumber \end{align}
		(The term $(\coth(3\pi/2)-1)/9$ on p.~484 of~\cite{sw} should be $(\coth(3\pi /2)-1)/27$,
which leads to an error in the 5th decimal place of their numerical approximation.)
The method of~\cite{sw} could be adapted to find $\int_0^1 \int_0^1  H(u,v;B) \ud u \ud v$,
but we want to evaluate $H(u,v;B)$ integrated against a different measure, as in~\eqref{eq:Phi1}.
We use a variation on the classical \emph{method of images} to evaluate $H(u,v;B)$ explicitly
for \emph{fixed} $u,v$. This is the content of Theorem~\ref{thm:H-formula}, which appears to be new,
and from which we deduce Proposition~\ref{prop:Phi0-Phi1}.

For $n \in \Z$, set $s_n (x,y) := \sin (n\pi x ) \sinh (n\pi y)$.
For $x,y  \in (0,1)$ and $z \in [0,1]$, define
\[
h(x,y,z) := \sum_{n \in \N } \frac{ 2 \sin( n\pi  (1-z)) }{\sinh( n\pi  )}  \Big(  s_n (x,y) + s_n (1-x,1-y) -s_n (y,x) - s_n (1-y,1-x) \Big);
\]
 provided $x, y \in (0,1)$, the sum here converges absolutely, uniformly for $z \in [0,1]$.

\begin{theorem}
\label{thm:H-formula}
  For all $(u,v) \in U \setminus D$ and all $B \in \cB$,
  \begin{equation}
      \label{eq:H-formula}
    H(u,v; B)  = \int_B h \left( \frac{u+v}{2} , \frac{u-v}{2} ,w \right) \ud w .
  \end{equation}
\end{theorem}
\begin{remark}
As a corollary to the theorem, following
a similar (but simpler) series of calculations to those in the proof of
Proposition~\ref{prop:Phi0-Phi1} below,
one can derive
\begin{align}
    \label{eq:sw2}
 2 \int_{U} H(u,v; [0,1]) \ud
u\ud v & =  \int_{[0,1]^2} H(u,v; [0,1]) \ud u \ud v \nonumber\\
& = \frac{32}{\pi^3} \sum_{\text{$n$ odd}} \frac{ (-1)^{\frac{n-1}{2}} }{n^3}
\sech(n\pi/2), 
\end{align}
which converges much faster than~\eqref{eq:sw}. 
Equality of~\eqref{eq:sw} and~\eqref{eq:sw2} entails the identity
\[
  1 - \frac{16}{\pi^3} \sum_{\text{$n$ odd}} \frac{\coth(n\pi /2)}{n^3} = \frac{32}{\pi^3} \sum_{\text{$n$ odd}} \frac{ (-1)^{\frac{n-1}{2}} }{n^3}
    \sech(n\pi/2),
  \]
for which we have not been able to find a reference.
\end{remark}

We prove Theorem~\ref{thm:H-formula} by solving an appropriate Dirichlet problem. For a 
domain $\cD \subset \R^2$ with boundary $\partial \cD$ and $g : \partial \cD \to \R$,  a twice-differentiable
$f : \R^2 \to \R$ solves the Dirichlet problem $(\cD, g)$ if $\nabla^2 f  =0$ on $\cD$
and $f = g$ on $\partial \cD$. 
We will show that $H$ solves the Dirichlet problem $(U,g)$ where $g$ depends on $u,v$ and $B$.
Since $g$ is not continuous, we approximate it by continuous functions.
Then we appeal to the explicit eigenfunctions of the Laplacian on $[0,1]^2$, and an application of the method of images~\cite{Keller},
 to solve the modified Dirichlet problem, and then take a limit. While we believe that Theorem~\ref{thm:H-formula} is new,
the idea has a long history, and we refer to~\cite{MorFesh} for some similar examples.

\begin{proof}[Proof of Theorem~\ref{thm:H-formula}.]
Provided $(u,v)  \in U \setminus D$, both $(u \pm v)/2$ are in $(0,1)$.
Note also that $h((1+v)/2,(1-v)/2,w) = h(u/2,u/2,w) = 0$ because of the antisymmetries
  $h(x,y,z) = -h(y,x,z) = -h(1-y,1-x,z)$, so   the integral in~\eqref{eq:H-formula} is~$0$ if $u=1$
  or $v=0$.  Thus it remains to prove~\eqref{eq:H-formula} for $(u,v) \in U \setminus \partial U$.
  Moreover, since Brownian motion started in the interior of $U$ hits $(0,0)$ or $(1,1)$ with
  probability 0, and the value of the integral is
 unaffected by the addition of points $0$ or $1$ to $B$, 
  it suffices to suppose that $B \subseteq (0,1)$.

Set $V := \{ (x,y) \in [0,1]^2 : x+y \leq 1, x \geq y \}$. 
	Define the matrix $M$ and associated linear transformation $m$ by
	\[ M:= \frac{1}{2} \begin{pmatrix} 1 & 1 \\ 1 & -1 \end{pmatrix} , \text{ and } m (u,v ) := M \begin{pmatrix} u \\ v \end{pmatrix} . \]
	Then $m$ maps $U$ to $V$, and $m (x,x) = (x,0)$.
Since $M M^\tra = \frac{1}{2} I$, where $I$ is the identity,
	the process $MW$ is a constant time-change of Brownian motion, so that
	\begin{equation}
	\label{eq:H-small-triangle}
	H(u,v;B) = \Pr \bigl( W_{\tau_{\partial V}}^{(1)} \in B, \, W_{\tau_{\partial V}}^{(2)} = 0 \bigmid W_0 =  m (u,v) \bigr).\end{equation}	
	Take $\gamma : [0,1] \to [0,1]$
	  continuous with $\gamma(0)=\gamma(1) =0$. Then $g_V  : \partial V \to [0,1]$ 
	with $g_V (x,y) = 0$ for $y > 0$ 
	and $g_V (x,0) = \gamma (x)$ is continuous on $\partial V$.
Define $g : S \to [0,1]$ by
\begin{equation}
\label{eq:g-def}
  g(x,y) = \begin{cases} \gamma (x) &\text{if $y=0$},\\
 \gamma(1-x) &\text{if $y=1$},\\
 -\gamma(y) &\text{if $x=0$},\\
 -\gamma(1-y) &\text{if $x=1$},
\end{cases}
\end{equation}
then $g$ is continuous on $S$ and satisfies $g(x,y) = - g(y,x) =  g(1-x,1-y)$ for all $(x,y) \in S$.
There is a unique twice-differentiable function $f : [0,1]^2 \to [0,1]$ that solves the
Dirichlet problem $([0,1]^2, g)$.
Moreover, $ f$ inherits from $g$ the symmetries 
 $f(x,y) = - f(y,x) = f(1-x,1-y)$ for all $(x,y) \in [0,1]^2$
(to see this, note for instance that $f(x,y) + f(y,x)$ solves the Dirichlet problem with zero boundary condition,
and hence is identically zero).
In particular, $ f(x,x) =  f(x,1-x) = 0$ for all $x \in [0,1]$.  Hence the function
$f_V :=  f|_V$ solves the Dirichlet problem $(V,g_V)$.  

For the simple region $[0,1]^2$, solutions to the Dirichlet problem can easily be described in terms of
combinations of functions $\sin (n \pi  x) \sinh(n \pi  y)$ for $n \in \N$ and
their images under the transformations $x \leftrightarrow y$ and $x \leftrightarrow 1-y$.
In particular, the solution to the Dirichlet problem $([0,1]^2,g)$, where $g$
is of the form~\eqref{eq:g-def}, is $f$ given by
\begin{align*}
   f(x,y) = \sum_{n\in \N} \frac{A_n s_n (x,y) + B_n s_n (1-x,1-y)+ C_n s_n (y,x) + D_n s_n(1-y,1-x)}{\sinh( n\pi  )}   ,
\end{align*}
where the boundary condition gives $A_n = B_n = -C_n = -D_n = 2\int_0^1 \gamma (z ) \sin( n\pi  (1-z))
\ud z$.

Given $B \in \cB$, with $B \cap \{0,1\} = \emptyset$,
  consider a sequence $\gamma_k$ of bounded continuous functions on $[0,1]$ with $\gamma_k (0) = \gamma_k(1) =0$
and $\lim_{k \to \infty} \gamma_k (x) = \2{B} (x)$ for every $x \in [0,1]$. Then let $f_k$ denote the solution to the
Dirichlet problem $([0,1]^2,g_k)$, where $g_k$ is constructed from $\gamma_k$ according to~\eqref{eq:g-def}.
Then
\begin{align*}
   f_k (x,y) = \sum_{n\in \N}  \frac{A_{k,n}}{\sinh( n\pi  )} \bigl( s_n (x,y) +  s_n (1-x,1-y) - s_n (y,x) - s_n(1-y,1-x) \bigr)  ,
\end{align*}
where 
  $A_{k,n} = 2 \int_0^1 \gamma_k(z) \sin(n\pi (1-z)) \ud z$. 
	As described above, restricting $f_k$ to $V$ gives the (unique)
	solution to the Dirichlet problem $(V,g_{V,k})$,
	where $g_{V,k} (x,y) = \gamma_k (x) \1 {y=0}$.
	By the connection between the Dirichlet problem 
	with continuous boundary conditions and stopped Brownian motion (see e.g.~Theorem~3.12 of~\cite{MoerPer}) we have
\[ \Exp \bigl( g_{V,k} ( W_{\tau_{\partial V}} ) \bigmid W_0 =  (x,y) \bigr) = f_k ( x,y ) ,
\text{ for all } (x,y) \in V. \]
	Now, by choice of $\gamma_k$,
	\[ \lim_{k \to \infty} g_{V,k} (W_{\tau_{\partial V}}) = \1 { W_{\tau_{\partial V}}^{(1)} \in B, \, W_{\tau_{\partial V}}^{(2)} = 0 }, \as ,\]
	so, by bounded convergence, 
	\begin{equation}
	\label{eq:dirichlet-prob}
	\Pr \bigl(  W_{\tau_{\partial V}}^{(1)} \in B, \, W_{\tau_{\partial V}}^{(2)} = 0 \bigmid W_0 =  (x,y) \bigr)
	= \lim_{k \to \infty} f_k ( x,y ) , \text{ for all } (x,y) \in V. \end{equation}
		By bounded convergence, 
	$\lim_{k \to \infty} A_{k,n} = 2 \int_B \sin(n \pi (1-z)) \ud z$. 
	For fixed $(x,y) \in (0,1)^2$, the series expression for $f_k$ is absolutely convergent,
and  $f(x,y) = \lim_{k \to \infty} f_k (x,y)$ satisfies
$f (x, y )   = \int_B h (x,y,z) \ud z$, with $h$ as defined in the display above Theorem~\ref{thm:H-formula}.
Then combining~\eqref{eq:H-small-triangle} and~\eqref{eq:dirichlet-prob}, we obtain the result. 
\end{proof}

Now we can complete the proof of Proposition~\ref{prop:Phi0-Phi1}.

\begin{proof}[Proof of Proposition~\ref{prop:Phi0-Phi1}.]
Recall the definition of $\Phi_1$ from~\eqref{eq:Phi1}.  For $y, z \in (0,1)$, $y \neq z$,
we have $H(y,z;B) = H( y \vee z, y \wedge z ; B)$ 
is given by the formula~\eqref{eq:H-formula}, so that
\begin{align*}
    \Phi_1 (B) = \int_B \ud w \int_0^1 \ud z \int_0^1 \ud y \int_0^1 \ud x \int_0^\infty q_t(x,y) h \bigl( \tfrac{y+z}{2} , \tfrac{(y\vee z)-(y \wedge z)}{2} ,w \bigr) \ud t,
\end{align*}
where we may define $h(x,0,w)$ and $h(1,y,w)$ arbitrarily.
We now proceed to show that $\Phi_1(B) = \int_B \psi(w) \ud w$, where $\psi$ is given by \eqref{eq:psi-def}.  Let 
\[ Q (y) := \int_0^1 \ud x \int_0^\infty q_t (x,y) \ud t .\]
Then
\begin{align*}
    \Phi_1 (B) & = \int_B\! \ud w \! \int_0^1 \! \ud y\!  \int_0^y\! Q(y) h \bigl( \tfrac{y+z}{2}, \tfrac{y-z}{2},w \bigr) \ud z
    + \int_B \! \ud w \! \int_0^1 \! \ud z\!  \int_0^z\! Q(y) h \bigl( \tfrac{y+z}{2}, \tfrac{z-y}{2},w \bigr) \ud y \\
    & = \int_B \! \ud w\! \int_0^1\! \ud y\! \int_0^y \bigl[ Q(y) + Q(z) \bigr]  h \bigl( \tfrac{y+z}{2}, \tfrac{y-z}{2},w \bigr) \ud z .
\end{align*}
Changing variables from $(y,z) \in U$ to $(u,v) = ( \tfrac{y+z}{2}, \tfrac{y-z}{2}) \in V$ we get
\begin{align*}
    \Phi_1 (B) & = 2 \int_B \ud w \int_0^1 \ud u \int_0^{u \wedge (1-u)}  \bigl[ Q(u+v) + Q(u-v) \bigr]  h ( u , v, w) \ud v .
\end{align*}
A useful alternative expression for $q_t$ (see~\cite[p.~122]{bosa}) is the spectral representation
\[ q_t (x,y) = 2 \sum_{m \in \N} \exp \left( - \frac{m^2 \pi^2 t}{2} \right) \sin (m\pi x) \sin (m\pi y) .\]
Hence
\[ Q(y) = \frac{8}{\pi^3} \sum_{m \text{ odd}} \frac{\sin (m\pi y)}{m^3} =  y (1-y) ;\]
see~\eqref{eq:s3odd} below.
It follows that
\begin{align*}
    \Phi_1 (B) & =  
    4 \int_B \ud w  \int_0^{1} \ud u  \int_0^{u \wedge (1-u)}  \bigl[ u (1-u) - v^2 \bigr] h(u,v,w) \ud v .
\end{align*}

Decomposing $h(u,v,w)$ into sums over even and odd $n$, and using that $h_{\textrm{even}}(u,v,w) = -h_{\textrm{even}}(1-u,v,w)$ and $h_{\textrm{odd}}(u,v,w) = h_{\textrm{odd}}(1-u,v,w) = h_{\textrm{odd}}(u,v,1-w)$, gives
\[
\int_0^{1}\! \ud u \! \int_0^{u \wedge (1-u)} \!\!\!  \bigl[ u (1-u) - v^2 \bigr] h(u,v,w) \ud v
 = 2\int_0^{1/2} \!\!\!\! \ud u \! \int_0^{u} \! \bigl[ u (1-u) - v^2 \bigr] h_{\textrm{odd}}(u,v,1-w) \ud v,\]
so that $\Phi_1(B) = 16 \displaystyle\int_B \sum_{\text{$n$ odd}} \dfrac{ I_n }{\sinh(n \pi)}  \sin(n \pi w) \ud w$, where
\[
I_n = \int_0^{1/2} \ud u \int_0^u [ u(1-u) -v^2 ]( s_n(u,v) + s_n(1-u,1-v) - s_n(v,u) - s_n(1-v,1-u) )\ud v.
\]
It remains to evaluate the integral $I_n$ for $n$ odd.  To simplify the calculation, observe that for $n$ odd, the angle-sum formulae for the trigonometric and hyperbolic sines imply 
\[
s_n(1/2-x, 1/2-y) + s_n(1/2+x,1/2+y) = 2 s_n(1/2,1/2) c_n(x,y),
\]
where $c_n(x,y) := \cos(n \pi x)\cosh(n \pi y)$.  Hence, changing variables from $(u,v)$ to $(x,y) = (1/2-v,1/2-u)$, the integral $I_n$ becomes
\[
I_n = 2s_n(1/2,1/2) \int_0^{1/2} \ud x \int_0^x [ x(1-x) - y^2] (c_n(y,x)- c_n(x,y))\ud y. \]
We can write $I_n = 2 s_n(1/2,1/2) (I_{n,1} + I_{n,2} )$, where, for $\tilde c_n (x,y) := c_n(y,x)- c_n(x,y)$,
\[ I_{n,1} := \int_0^{1/2} \ud x \int_0^x x(1-2x) \tilde c_n(x,y)\ud y, \text{ and } I_{n,2}:= \int_0^{1/2} \ud x \int_0^x ( x^2-y^2 ) \tilde c_n(x,y) \ud y .\]
Then, since $n\pi \int_0^x \tilde c_n(x,y) \ud y = \mathfrak{R}( (1+i)\sin( n\pi(1+i)x) )$, integration by parts of the (complex) integral $\int_0^{1/2}x(1-2x)\sin( n\pi(1+i)x)\ud x$ yields
\[
I_{n,1} = \frac{\sin(n\pi/2)}{\pi^4} \left(\frac{2\sinh(n\pi/2)}{n^4} - \frac{\pi \cosh(n\pi/2)}{2n^3} \right), \text{ for odd } n. 
\]
For $I_{n,2}$, notice that $ - \int_0^{1/2} \ud x \int_0^{x} y^2 \tilde c_n(x,y) \ud y = \int_0^{1/2 }\ud y\int_0^{y} x^2 \tilde c_n(x,y) \ud x$, 
so that $I_{n,2} = \int_0^{1/2}\int_0^{1/2} x^2 \tilde c_n(x,y) \ud x \ud y$, and integration by parts yields
\[
I_{n,2} = \frac{\sin(n\pi/2)}{\pi^4}\left(\frac{4\sinh(n\pi/2)}{n^4} -\frac{\pi\cosh(n\pi/2)}{n^3}\right), \text{ for odd } n. 
\]
Hence, using that $\sin^2(n\pi/2) = 1$ for $n$ odd, we have
\[
\frac{16 I_n}{\sinh(n\pi)} = \frac{32\sinh(n\pi/2)}{\pi^4 \sinh(n\pi)}\left( \frac{6\sinh(n\pi/2)}{n^4} - \frac{3\pi\cosh(n\pi/2)}{2n^3} \right) = 
\frac{24}{\pi^4} a_n,
\]
for odd $n$, where $a_n$ is given by \eqref{eq:psi-def}, and therefore $\Phi_1(B) = \int_B \psi(w) \ud w$.
\end{proof}

\section{Analysis of the splitting density}
\label{sec:numerics}

In this section we present 
analytical and numerical results on the probability density
$\phi_0$ appearing in~\eqref{eq:Phi0-def}. 
We start by discussing 
efficient numerical approximation of the function~$\psi$ defined at~\eqref{eq:psi-def} and the
constant $\mu$ defined at~\eqref{eq:mu-def}. 

First we establish the final equality in~\eqref{eq:mu-def}.
This will follow from the identity
\begin{equation}
\label{eq:sech-tanh}
 4\sum_{n \text{ odd}}  \frac{\tanh \left( n\pi / 2 \right)}{n^5}
= \frac{\pi^5}{96} + \pi \sum_{n \text{ odd}}  \frac{\sech^2 \left( n\pi / 2 \right)}{n^4} .\end{equation}
The equality~\eqref{eq:sech-tanh} may be deduced from the fact that, for $\alpha, \beta >0$ with $\alpha \beta = \pi^2$,
\begin{equation}
\label{eq:tanh-tanh}
 \alpha^{-2} \sum_{n \text{ odd}} \frac{\tanh \left(  n \alpha / 2 \right)}{n^5}
- ( - \beta)^{-2}  \sum_{n \text{ odd}} \frac{\tanh \left(   n \beta / 2 \right)}{n^5} = \frac{\alpha \beta (\beta-\alpha)}{192} ,\end{equation}
 a formula attributed to de~Saint-Venant in 1856~\cite[p.~294]{berndt}. It follows from~\eqref{eq:tanh-tanh} that
\[ \frac{\alpha+\beta}{\pi^4}  \sum_{n \text{ odd}} \frac{\tanh \left(  n \alpha / 2 \right)}{n^5}
- \frac{1}{\beta^2}  \sum_{n \text{ odd}} \frac{\tanh \left(  n \beta / 2 \right)-\tanh \left(   n \alpha / 2 \right)}{n^5 (\beta - \alpha)} = \frac{\pi^2}{192} .\]
Taking $\alpha - \beta \to 0$ gives~\eqref{eq:sech-tanh}. Then
from the first series in~\eqref{eq:mu-def} with~\eqref{eq:psi-def}, we have that
\[ \mu = \frac{192}{\pi^5} \sum_{n \text{ odd}} \frac{\tanh \left(   n \pi / 2 \right)}{n^5} - \frac{48}{\pi^4} \sum_{n \text{ odd}} \frac{1}{n^4} .\]
Writing $\zeta (s) := \sum_{n \in \N} n^{-s}$, note that
$\sum_{n \text{ odd}} n^{-s} = (1-2^{-s}) \zeta(s)$ for $s >1$. 
Thus we obtain the final series in~\eqref{eq:mu-def}, using~\eqref{eq:sech-tanh} and the fact that $\zeta (4) = \pi^4/90$.

Truncating the second series in~\eqref{eq:mu-def}, 
 we can write, for any odd integer $n$,
\[ \mu = \mu_n  + \frac{48}{\pi^4} r_{n} , \text{ where } \mu_n := \frac{48}{\pi^4} \sum_{\substack{k \leq n \\ k \text{ odd}}}  \frac{\sech^2 \left(  k\pi / 2 \right)}{k^4}
\text{ and } r_n  :=   \sum_{\substack{k > n \\ k \text{ odd}}} \frac{\sech^2 \left(  k\pi / 2 \right)}{k^4}.
\]
Since $\sech x \leq 2 \re^{-x}$ for all $x \in \R$ we have, for $n$ odd,
\begin{align}
\label{eq:r-bound} 
r_n  
\leq \frac{4}{(n+2)^4}   \sum_{\substack{k \geq n+2 \\ k \text{ odd}}} \re^{-k\pi}  \leq \frac{4 \re^{-(n+2)\pi}}{(n+2)^4} \cdot \frac{1}{1-\re^{-2\pi}} .\end{align}
In particular, the bound $r_3 < 10^{-9}$ guarantees that $\mu$ is approximated by
$\mu_3$
 to within $5 \times 10^{-10}$. Since $\mu_3 \approx 0.078268954659$, this 
suffices to evaluate the first 8 decimal digits of $\mu$ as $\mu \approx 0.07826895$.
 
This idea can be extended to compute moments of $\Phi_0$. Set  $m_k :=  \int_0^1 z^k \phi_0 (z) \ud z$.

\begin{proposition}
\label{prop:moments}
We have that $m_1 =1/2$, 
\[ m_2 = \frac{1}{2} - \frac{1}{60\mu}, ~ m_3 = \frac{1}{2} - \frac{1}{40\mu}, 
\text{ and } m_4 = \frac{1}{2} - \frac{11}{280\mu} + \frac{576}{\mu \pi^8} \sum_{n \text{ odd}}  \frac{\sech^2 \left( n\pi/ 2 \right)}{n^8} .\]
\end{proposition}
\begin{proof}
If
\[ \omega_{k,n} := \int_0^1 z^k \sin ( n \pi z ) \ud z , \]
then we have from~\eqref{eq:psi-def} and the fact that $\phi_0 (z) = \mu^{-1} \psi (z)$ that
$m_k = \frac{24}{\mu \pi^4} \sum_{n \text{ odd}} a_n \omega_{k,n}$. 
For example, $\omega_{1,n} = \frac{1}{n\pi} (-1)^{n+1}$ so
$ m_1 = \frac{24}{\mu \pi^5} \sum_{n \text{ odd}} \frac{a_n}{n} = \frac{1}{2}$, by~\eqref{eq:mu-def},
as is to be  expected, due to the symmetry of $\phi_0$ around $1/2$. Also, $\omega_{2,n} = \frac{1}{n \pi} - \frac{4}{n^3 \pi^3}$ for odd $n$,
so
\begin{align*}
     m_2 & = \frac{1}{2} - \frac{96}{\mu \pi^7} \sum_{n \text{ odd}} \frac{a_n}{n^3} = \frac{1}{2} - \frac{384}{\mu \pi^7} \sum_{n \text{ odd}} \frac{\tanh \left(  n\pi / 2 \right)}{n^7}
		+ \frac{96}{\mu \pi^6} \cdot \zeta(6) \cdot \frac{63}{64} . \end{align*}
Since $\sum_{n \text{ odd}} n^{-7} \tanh \left( \frac{n\pi}{2} \right) = \frac{7\pi^7}{23040}$~\cite[p.~293]{berndt}
and $\zeta (6) = \frac{\pi^6}{945}$, we get the claimed formula for $m_2$. The formula for $m_3$ follows from those for $m_1$ and $m_2$ by symmetry of~$\phi_0$.

Finally, for $n$ odd, $\omega_{4,n} = \frac{1}{n\pi} - \frac{12}{n^3 \pi^3} + \frac{48}{n^5\pi^5}$, and, similarly to before,
\begin{align*}
m_4 & = \frac{1}{2} - \frac{6}{35\mu} + \frac{4608}{\mu \pi^9}  \sum_{n \text{ odd}}  \frac{\tanh \left(  n\pi / 2 \right)}{n^9} .\end{align*}
The claimed formula for $m_4$ now follows from the identity
\[ 4096  \sum_{n \text{ odd}}  \frac{\tanh \left(  n\pi / 2 \right)}{n^9} = \frac{37 \pi^9}{315} + 512 \pi  \sum_{n \text{ odd}}  \frac{\sech^2 \left(  n\pi / 2 \right)}{n^8} ,\]
which can be obtained in a similar fashion to~\eqref{eq:sech-tanh},
but replacing~\eqref{eq:tanh-tanh} by the appropriate higher-order analogue from~\cite[p.~294]{berndt}.
\end{proof}

The  formulae in Proposition~\ref{prop:moments}
give $m_2$ and $m_4$ to 10 decimal places as
\begin{equation}
\label{eq:phi-two-moments}
 m_2 \approx 0.2870590372, \text{ and } m_4 \approx
0.1212564646.\end{equation}

\begin{corollary}
The distribution $\Phi_0$ is not a Beta distribution.
\end{corollary}
\begin{proof}
For $\beta >0$, the Beta$(\beta,\beta)$ distribution has density 
proportional to $x^{\beta-1} (1-x)^{\beta-1}$ for $x \in [0,1]$, and its $k$th moment
is $m_{\beta,k} = \prod_{j=0}^{k-1} \frac{\beta+j}{2\beta+j}$.
Thus $m_{\beta,1} = 1/2$. To fit $m_{\beta,2} = m_2$ as given by~\eqref{eq:phi-two-moments}
requires that $\beta = \beta_\star \approx 
2.8729963811$. But the Beta$(\beta_\star,\beta_\star)$
distribution has $m_{\beta_\star,4} \approx 0.1212665009$, which fails to match~$m_4$ from~\eqref{eq:phi-two-moments}.
\end{proof}

Now we turn to analysis of the density $\phi_0$. It is useful to write
\[ a_n =  b_n - \frac{4}{n^4} d_n, \text{ where } b_n := \frac{4-n\pi}{n^4} \text{ and } d_n := 1-  \tanh \left ( n \pi/2 \right) , \text{ for } n \in \N .\]
Note that $0 < d_n < 2 \re^{-n \pi}$. For $k \in \N$, define the functions
\begin{equation}
    \label{eq:S-C}
 S_{k} (x) := \sum_{n=1}^\infty \frac{\sin nx}{n^{k}} , \text{ and } C_k (x):=  \sum_{n=1}^\infty \frac{\cos nx}{n^{k}} .\end{equation}
It is known (see e.g.~equation 1.443.1 of~\cite[p.~47]{gr}) that
\begin{equation}
\label{eq:s3}
S_3 (x) = \frac{\pi^2}{6} x - \frac{\pi}{4} x^2 + \frac{1}{12} x^3 , \text{ for } 0 \leq x \leq 2\pi .\end{equation}
There is no closed form for $S_2$ or $S_4$, which are relatives of the \emph{Clausen function}~\cite{lewin}.
We will express $\psi$ in terms of the function
\begin{equation}
\label{eq:S-def}
 S(x) := \sum_{n \text{ odd}} \frac{\sin n\pi x}{n^4} = S_4 (\pi x) - \frac{1}{16} S_4 (2\pi x) .\end{equation}

\begin{lemma}
\label{lem:psi-clausen}
We have that
\begin{align}
\label{eq:psi3}
\psi (x) & = \frac{96}{\pi^4} S(x) - 3 x (1-x) - \frac{96}{\pi^4} \sum_{n \text{ odd}} \frac{d_n}{n^4} \sin n \pi x . \end{align}
Moreover, $\psi$ is twice continuously differentiable on $[0,1]$.
\end{lemma}

\begin{remark}
The third derivative of $\psi$ diverges as $x \to 0$ (to $-\infty$) and $x \to 1$ (to $+\infty$).
\end{remark}

\begin{proof}[Proof of Lemma~\ref{lem:psi-clausen}.]
By rearranging~\eqref{eq:psi-def}, we can write
\[
 \psi (x) = \frac{96}{\pi^4} \sum_{n \text{ odd}} \frac{\sin n\pi x}{n^4}
- \frac{24}{\pi^3} \sum_{n \text{ odd}} \frac{\sin n \pi x}{n^3} - \frac{96}{\pi^4} \sum_{n \text{ odd}} \frac{d_n}{n^4} \sin n \pi x . \]
From~\eqref{eq:s3} we have that, for $0 \leq x \leq 1$,
\begin{equation}
\label{eq:s3odd}
 \sum_{n \text{ odd}} \frac{\sin n \pi x}{n^3} = S_3 (\pi x) - \frac{1}{8} S_3 (2\pi x) = \frac{\pi^3}{8} x (1-x) .\end{equation}
 This yields~\eqref{eq:psi3}. 
 The series expression for $S(x)$ is evidently twice continuously differentiable, and hence the same
 is true for $\psi$.
 \end{proof}
 
 Although $S_4$ has no closed form, it has some numerically efficient series representations.
 We use one of these to obtain an efficient approximation for $\psi$, and hence $\phi_0$.
The (absolute) Bernoulli numbers are
$| B (2\ell) | := 2 \zeta (2 \ell) (2\ell)! / (2\pi)^{2\ell}$. 
For $k,m \in \ZP$ and $x \in [0,1]$,
 define
\begin{align*}
\psi^{k,m} (x ) & := \frac{84}{\pi^3} x \zeta(3) + \frac{8}{\pi} x^3 \log (\pi x) - \frac{8}{\pi} \left( \frac{11}{6} + \log 2 \right) x^3 - 3 x (1-x) \\
& {} \qquad {} + 48 \pi x^5 \sum_{n=0}^k \frac{ | B(2n+2) | \left( 2^{2n+1} - 1 \right)}{(n+1) (2n+5)!} \pi^{2n} x^{2n}
- \frac{96}{\pi^4} \sum_{\substack{n \text{ odd}\\ n \leq m}} \frac{d_n}{n^4} \sin n \pi x .\end{align*}
It turns out that $\psi^{k,m} \to \psi$ as $k,m \to \infty$, but the convergence
is poor as $x$ approaches 1. Thus we make use of the symmetry of $\psi$ and consider
the symmetrization
\[ \phi_0^{k,m} (x) := \begin{cases} \frac{1}{\mu} \psi^{k,m} (x) &\text{if } 0 \leq x \leq 1/2,\\
 \frac{1}{\mu} \psi^{k,m} (1-x) & \text{if } 1/2 < x \leq 1. \end{cases} \]
Then $\phi_0^{k,m}$ converges rather rapidly to $\phi_0$, as shown by the following estimate.

\begin{lemma}
\label{lem:psi-approx}
For all $k, m \in \ZP$, with $m$ odd,
\[  \sup_{0 \leq x \leq 1} \left| \phi_0^{k,m} (x) - \phi_0(x) \right| \leq \frac{4^{-k} \zeta(2k+4)}{\pi \mu (2k+4)^4}  +   \frac{2\re^{-(m+2)\pi}}{\mu (m+2)^4} 
.\]
For example, $\sup_{0 \leq x \leq 1}  | \phi_0^{9,5} (x) - \phi_0 (x) | < 10^{-10}$.
\end{lemma}
\begin{proof}
A standard series expansion, valid for $0 \leq x < 2\pi$, is
\begin{equation}
\label{eq:cl2}
 S_2 (x) = x - x \log x + \frac{x^3}{2} \sum_{n=0}^\infty \frac{| B (2n+2) |}{(n+1)(2n+3)!} x^{2n} ;\end{equation}
see e.g.~equation~(4.28) of~\cite{lewin} or Proposition~3.1 of the more readily accessible~\cite{lo}.
Differentiation in~\eqref{eq:S-C} gives
$S'_4 (x) = C_3 (x)$ and $C'_{3} (x) = - S_2 (x)$, so we may integrate~\eqref{eq:cl2}
twice, term by term, using the initial values 
$C_3 (0) = \zeta(3)$ and $S_4 (0) = 0$,
 to get
\begin{equation}
\label{eq:cl4}
 S_4 (x) = x \zeta (3) + \frac{x^3}{6} \log x - \frac{11}{36} x^3 - \frac{x^5}{2} \sum_{n=0}^\infty \frac{| B (2n+2) |}{(n+1)(2n+5)!} x^{2n} , \end{equation}
for $0 \leq x < 2\pi$.
It follows from~\eqref{eq:cl4} that, for $0 \leq x < 1$,
\begin{align}
\label{eq:s4odd}
  S_4 (\pi x) - \frac{1}{16} S_4 (2\pi x) 
& = \frac{7}{8} \pi x \zeta(3) + \frac{\pi^3}{12} x^3 \log (\pi x) - \frac{\pi^3}{12} \left( \frac{11}{6} + \log 2 \right) x^3 \nonumber\\
& {} \qquad {} + \frac{\pi^5 x^5}{2} \sum_{n=0}^\infty \frac{ | B(2n+2) | \left( 2^{2n+1} - 1 \right)}{(n+1) (2n+5)!} \pi^{2n} x^{2n} . 
\end{align}
Then substituting~\eqref{eq:s4odd} for~$S(x)$ in~\eqref{eq:S-def} and~\eqref{eq:psi3}, we get for $0 \leq x < 1$,
\begin{align}
\label{eq:psi-series}
\psi (x) & =  \frac{84}{\pi^3} x \zeta(3) + \frac{8}{\pi} x^3 \log (\pi x) - \frac{8}{\pi} \left( \frac{11}{6} + \log 2 \right) x^3 - 3 x (1-x) \nonumber\\
& {} \qquad {} + 48 \pi x^5 \sum_{n=0}^\infty \frac{ | B(2n+2) | \left( 2^{2n+1} - 1 \right)}{(n+1) (2n+5)!} \pi^{2n} x^{2n}
- \frac{96}{\pi^4} \sum_{n \text{ odd}} \frac{d_n}{n^4} \sin n \pi x . 
\end{align}
Since $\phi_0(x) = \mu^{-1} \psi (x)$, it follows that 
\begin{align*}
\sup_{0 \leq x \leq 1/2} \left| \phi_0 (x) - \phi_0^{k,m} (x) \right| & \leq \frac{3\pi}{\mu} \sum_{n>k} \frac{ | B(2n+2) | \pi^{2n}}{(n+1) (2n+5)!} 
+ \frac{96}{\pi^4 \mu} \sum_{\substack{n \text{ odd} \\ n>m}} \frac{d_n}{n^4} . \end{align*}
Here, since $| B (2\ell) | = 2 \zeta (2 \ell) (2\ell)! / (2\pi)^{2\ell}$,
\[
 \frac{3\pi}{\mu} \sum_{n>k} \frac{ | B(2n+2) | \pi^{2n}}{(n+1) (2n+5)!} 
\leq 
 \frac{3}{\pi \mu}  \sum_{n>k} \frac{  \zeta(2n+2) }{(2n+2)^4} 2^{-2n}
 \leq 
 \frac{3}{\pi \mu}   \frac{  \zeta(2k+4) }{(2k+4)^4} \sum_{n = k+1}^\infty 4^{-n},
\]
since $\zeta( \, \cdot \,)$ is decreasing. Moreover, a similar bound to~\eqref{eq:r-bound}
gives
\[ \sum_{\substack{n \text{ odd} \\ n>m}} \frac{d_n}{n^4} \leq 
2 \sum_{\substack{n \text{ odd} \\ n>m}} \frac{\re^{-n\pi}}{n^4}
\leq \frac{2}{1-\re^{-2\pi}} \frac{\re^{-(m+2)\pi}}{(m+2)^4}
 .\]
With the numerical bound $96 < \pi^4 ( 1 - \re^{-2\pi})$, this completes the proof.
\end{proof}

Important for the asymptotics of the normalized gap distribution given in Theorem~\ref{thm:gap-statistics}
is the behaviour of $\phi_0(x)$ as $x \to 0$ (see Theorem~\ref{thm:general-splitting-distribution} below). 
Here the expression~\eqref{eq:psi-series}
is misleading at first glance, as the next result shows.

\begin{lemma}
\label{lem:psi-near-zero}
We have that $\phi_0 (x) \sim (3/\mu) x^2$ as $x \to 0$.
\end{lemma}
\begin{proof}
First note that, using~\cite[p.~287]{berndt} to evaluate the sum involving $\tanh$,
\begin{equation}
\label{eq:psi-cancellation}
  \sum_{n \text{ odd}} n a_n =  
	4 \sum_{n \text{ odd}} \frac{\tanh (  n \pi / 2 )}{n^3} 
 - \pi  \sum_{n \text{ odd}} \frac{1}{n^2} = 
\frac{\pi^3}{8} - \frac{3\pi}{4} \zeta(2) = 0 .
\end{equation}
Then from~\eqref{eq:psi-def} with~\eqref{eq:S-def}, \eqref{eq:s3odd} and \eqref{eq:psi-cancellation} we obtain, as an alternative to~\eqref{eq:psi3},
\begin{align}
\label{eq:psi-near-zero} 
\psi (x) & = \frac{24}{\pi^4} \sum_{n \text{ odd}}  a_n \bigl(  \sin n \pi x - n \pi x \bigr) \nonumber\\
& = \frac{96}{\pi^4} S(x) + 3 x^2 - \frac{84\zeta(3)}{\pi^3} x - \frac{96}{\pi^4} \sum_{n \text{ odd}} \frac{d_n}{n^4} \bigl( \sin n \pi x - n\pi x \bigr) .
\end{align}
Here $S(x) = \frac{7}{8} \pi x \zeta(3)  + o(x^2)$ as $x \to 0$, by~\eqref{eq:s4odd}.
Since $| y - \sin y | = O (y^3)$ as $y \to 0$, and $d_n = O ( \re^{-n \pi})$,
the final sum in~\eqref{eq:psi-near-zero} is absolutely convergent, uniformly for $x \in [0,1]$, and hence is $O(x^3)$.
Thus~\eqref{eq:psi-near-zero} gives $\psi (x) \sim 3 x^2$  as $x \to 0$.
\end{proof}

\section{Limiting gap statistics}
\label{sec:gap-tails}

This section contributes to the proof of Theorem~\ref{thm:gap-statistics}, by establishing the
corresponding limit statements for the approximating interval-splitting process appearing in Theorem~\ref{thm:small-lambda}, building on work of Brennan \& Durrett~\cite{bd1,bd2}. We work in a more general setting to emphasize which elements of $r_0$ and $\phi_0$
contribute to the tail asymptotics of the normalized gap density $g_0$.
Also, because the approximation between the nucleation process and the interval-splitting limit works well only for large times (see \S\ref{sec:fixed-rate-regime}), we derive our results
on the interval-splitting process started from arbitrary initial conditions. To this end,
for $n_0 \in \ZP$ and $z \in \Delta_{n_0}$,
we write $\Prnz$ for the law of the interval-splitting process
$\cS = (\cS_{n_0}, \cS_{n_1}, \ldots, )$
with $\cS_{n_0} = z$ and evolving for $n \geq n_0$ according to~\eqref{eq:interval-splitting} with parameters $r$ and $\Phi$. Here is the main result of this section.

\begin{theorem}
\label{thm:general-splitting-distribution}
Let $\alpha, b \in (0,\infty)$, $\beta \in \RP$, $r(\ell) = \ell^\alpha$,
and $\phi$ be a bounded probability density on $[0,1]$
with $\phi(x) = \phi(1-x)$ for all $x \in [0,1]$ 
and $\phi (x) \sim b x^\beta$ as $x \to 0$.
Define $\Phi(B) = \int_B \phi(x) \ud x$ for all $B \in \cB$.
Let $\cS$ be an interval-splitting process with parameters $r$ and $\Phi$, and let $\ell_{n,i}$, $i \in [n+1]$,
denote the lengths of the gaps in $\cS_n$. For $x \in [0,1]$, let $C_n (x) = \max \{ m \in \{0, 1, \ldots, n+1 \} : \sum_{i=1}^m \ell_{n,i} \leq x \}$.
\begin{itemize}
\item[(i)] 
For all $n_0 \in \ZP$ and
  all $z \in \Delta_{n_0}$, $\lim_{n \to \infty} \sup_{x \in [0,1]} | n^{-1} C_n(x) -  x  | = 0$, $\Prnz$-a.s.
\item[(ii)] There exists a continuous probability density function~$g$ on $\RP$ such that for all $n_0 \in \ZP$, 
  all $z \in \Delta_{n_0}$, and all $x \in \RP$,
\begin{equation} 
\label{eq:gap-distribution-limit}
  \lim_{n \to \infty} \frac{1}{n+1} \sum_{i \in [n+1]} \1 { (n+1) \ell_{n,i} \leq x }  = \int_0^x g(y) \ud y  , \text{ $\Prnz$-a.s.~and in $L^1$.} \end{equation}
\item[(iii)]  There exist constants $c_{g,0}, c_{g,\infty}, \theta \in (0,\infty)$ such that 
\[ g (x) \sim c_{g,0} \, x^{\beta}, \text{ as $x \to 0$, and } 
 g (x) \sim c_{g,\infty}\,  x^{2a-2}   \exp ( - \theta x^\alpha) , \text{ as $x \to \infty$}, \]
where in the latter case, $a = \lim_{x \to 0} \phi(x)$.
\end{itemize}
\end{theorem}

\begin{remark}
\label{rem:uniform-g}
In the case of a uniform splitting distribution, where $\phi(x) \equiv 1$ for $x \in [0,1]$, one has the explicit expression (see Remark~\ref{rem:explicit-f-Z-q} below) that 
\[ g(x) = \frac{\alpha \Gamma (2/\alpha)}{\Gamma (1/\alpha)^2} \exp \left\{ - \left( \frac{ \Gamma (2/\alpha)}{\Gamma (1/\alpha)} \right)^\alpha x^\alpha \right\}. \]
\end{remark}

There are two parts to the proof of Theorem~\ref{thm:general-splitting-distribution}. One is to translate the results of~\cite{bd1,bd2},
which pertain to a continuous-time   interval-splitting model started from a unit interval,
to our setting, to obtain a characterization of the density~$g$
in terms of distributional fixed-point equations. The second part of the proof is 
an analysis of these fixed-point equations to obtain the tail asymptotics. We start with the second part.

The fixed-point description goes as follows. Let $X$ and $T$ be random variables on $\RP$
with probability density functions $f_X$ and $f_T$ respectively, given by
\[ f_X (x) = 2 \re^{-2x} \phi (\re^{-x} ), \text{ and } f_T (x) = c_T \int_0^{\re^{-x}} s \phi(s) \ud s, \text{ where } \frac{1}{c_T} = \int_0^\infty u \re^{-2u} \phi (\re^{-u} ) \ud u  .\]
Define the distribution of random variables $Q$ and $Z$ via the  fixed-point equation 
\begin{equation}
\label{eq:fixed-point}
 (Q ,Z ) \eqd (Z  \re^{-\alpha T}, Z \re^{- \alpha X} + \xi ) , ~~ Q \geq 0, \, Z \geq 0 ,\end{equation}
where the $Z, T, X$, and $\xi$ on the right-hand side are independent, and 
$\xi$ is exponentially distributed with unit mean.
The second coordinate equality in~\eqref{eq:fixed-point} determines uniquely the distribution of~$Z$  by e.g.~Theorem~1.5(i) and Lemma~1.4(a) of~\cite{vervaat};
the first coordinate equality then specifies the distribution of~$Q$. 
We will show that the $g$ in~\eqref{eq:gap-distribution-limit} is given 
in terms of the density $q$ of the random variable $Q^{1/\alpha}$; the next result gives
asymptotics for $q$.

\begin{lemma}
\label{lem:fixed-point}
Let $\alpha, b \in (0,\infty)$, $\beta \in \RP$, $r(\ell) = \ell^\alpha$,
and $\phi$ be a bounded probability density on $[0,1]$
with $\phi(x) = \phi(1-x)$ for all $x \in [0,1]$ 
and $\phi (x) \sim b x^\beta$ as $x \to 0$.
Then the random variable $Q^{1/\alpha}$
whose distribution is characterized by~\eqref{eq:fixed-point} has  a density~$q$ which is continuous on $\RP$, and 
there exist constants $c_{q,0}, c_{q,\infty} \in (0,\infty)$ such that 
\[ q (x) \sim c_{q,0} \, x^{1+\beta}, \text{ as $x \to 0$, and } 
 q (x) \sim c_{q,\infty} \, x^{2a-1}   \exp ( - x^\alpha) , \text{ as $x \to \infty$}, \]
where in the latter case, $a = \lim_{x \to 0} \phi(x)$.
\end{lemma}
\begin{proof}
Let $F_Z(r) := \Pr (Z \leq r)$. By~\eqref{eq:fixed-point}, conditioning on $\xi$ and then $X$, 
for $r \geq 0$,
\begin{align*}
 F_Z (r) & = \int_0^r \re^{-u} \Pr ( Z \re^{-\alpha X} \leq r - u )  \ud u \\
& = 2 \int_0^r \ud u \re^{-u}  \int_0^\infty \re^{-2x} \phi ( \re^{-x} ) F_Z ( (r-u) \re^{\alpha x} ) \ud x . \end{align*}
 With the change of variable $v = r-u$, this says
\[  F_Z (r) = 2 \re^{-r} \int_0^r \ud v  \re^v \int_0^\infty \re^{-2x} \phi ( \re^{-x} ) F_Z ( v \re^{\alpha x} ) \ud x , \]
which is  continuously differentiable,  so $f_Z (r) := F'_Z(r)$ exists and is continuous. Also
\begin{align*}
\Pr ( Z \re^{-\alpha X} \leq r )   = 2 \int_0^\infty \re^{-2x} \phi ( \re^{-x} ) F_Z ( r \re^{\alpha x} ) \ud x.\end{align*}
Since $F_Z$ is continuously differentiable, we can differentiate under the integral to get that $Y := Z \re^{-\alpha X}$  has a density $f_Y$ satisfying
\begin{align*} f_Y (r)  =  2 \int_0^\infty \re^{(\alpha-2)x} \phi ( \re^{-x} ) f_Z ( r \re^{\alpha x} ) \ud x.\end{align*}
Then since $Z$ is distributed as $Y + \xi$, we can write
\begin{align*}
F_Z (r)
 = \int_0^r \! f_Y (y) \Pr ( \xi \leq r -y ) \ud y 
  = 2 \int_0^r \! \ud y (1 - \re^{-(r-y)}) \int_0^\infty\! \re^{(\alpha-2)x} \phi ( \re^{-x} ) f_Z ( y \re^{\alpha x} ) \ud x .\end{align*}
Differentiating we obtain  
\begin{align*}
f_Z (r) = 2 \re^{-r} \int_0^r \ud y \re^y \int_0^\infty \re^{(\alpha-2)x} \phi ( \re^{-x} ) f_Z ( y \re^{\alpha x} ) \ud x .
\end{align*}
With the substitution $u = y \re^{\alpha x}$, we get, for $r \geq 0$,
\begin{align}
\label{eq:f-Z-fixed-point}
 f_Z (r)  = \frac{2}{\alpha} \re^{-r} \int_0^r \ud y \re^y y^{\frac{2-\alpha}{\alpha}} \int_y^\infty u^{-\frac{2}{\alpha}} \phi \bigl( (y/u)^{1/\alpha} \bigr) f_Z ( u ) \ud u . \end{align}

We use the relation~\eqref{eq:f-Z-fixed-point} to derive asymptotics of $f_Z(r)$ as $r \to 0$.
Fix $\eps > 0$. 
For $u \geq Ky$, $K>1$, we have $y/u \leq  1/K$. Choosing $K > 1$ large
 enough (depending on $\eps$), this means $\phi  ( (y/u)^{1/\alpha}  ) \leq ( b + \eps) (y/u)^{\beta/\alpha}$
for all $y >0$ and all $u \geq Ky$.
Hence
\[  \int_{Ky }^\infty u^{-\frac{2}{\alpha}} \phi \bigl( (y/u)^{1/\alpha} \bigr) f_Z ( u ) \ud u
\leq (b + \eps ) y^{\beta/\alpha} \int_{Ky}^\infty u^{-\frac{2+\beta}{\alpha}} f_Z ( u ) \ud u .\]
On the other hand, let 
$A := \sup_{x \in [0,1]} \phi (x)$, which is finite. 
Then, for all $y \in \RP$,
\[ \int_y^{Ky} u^{-\frac{2}{\alpha}} \phi \bigl( (y/u)^{1/\alpha} \bigr) f_Z ( u ) \ud u 
\leq A  B(K y) \int_y^{Ky} u^{-\frac{2}{\alpha}} \ud u, \text{ where } B(y) := \sup_{0 \leq u \leq y} f_Z (u) .\]
It follows from~\eqref{eq:f-Z-fixed-point} that for $C$ a finite constant depending on $K$,
for all $r \in \RP$,
\begin{equation}
\label{eq:f-Z-iteration}
 f_Z (r) \leq C  \int_0^r  B (Ky) \ud y
+ \frac{2}{\alpha} (b + \eps ) \int_0^r \ud y  y^{\frac{2+\beta-\alpha}{\alpha}} \int_{Ky}^\infty u^{-\frac{2+\beta}{\alpha}} f_Z ( u ) \ud u .\end{equation}
We apply~\eqref{eq:f-Z-iteration} successively to get a bound.
Let  $r_k = K^{-k}$.
Suppose that for constants $C_k, \gamma_k \in \RP$ we have $f_Z (r) \leq C_k r^{\gamma_k}$ for $r \in [0,r_k]$.
We bound the  $u$-integral in~\eqref{eq:f-Z-iteration} via
\begin{align*} 
\int_{Ky}^\infty u^{-\frac{2+\beta}{\alpha}} f_Z ( u ) \ud u 
& \leq C_k \int_{Ky}^{r_k} u^{\gamma_k - \frac{2+\beta}{\alpha}} \ud u
+ r_k^{-\frac{2+\beta}{\alpha}} \\
& \leq C_{k+1} + C_{k+1}  y^{1 +\gamma_k - \frac{2+\beta}{\alpha}} \log (1/y) ,\end{align*}
for some $C_{k+1} < \infty$ and all $y \in [0,r_{k+1}]$. 
Thus from~\eqref{eq:f-Z-iteration} we get, for all $r \in [0,r_{k+1}]$,
\[ f_Z (r) \leq C_{k+1}  r^{1+\gamma_k} \log (1/r)    + C_{k+1} r^{\frac{2+\beta}{\alpha}} \leq C_{k+1} r^{\gamma_{k+1}} , 
\text{ where } \gamma_{k+1}  = \left( \tfrac{1}{2} + \gamma_k \right) \wedge \left( \tfrac{2+\beta}{\alpha} \right) , \]
redefining $C_{k+1}$ as necessary. Starting with the bound
$f_Z(r) \leq C_0 = \sup_{x \in[0,1]} f_Z(x) < \infty$ for $r \in [0,r_0]$, we iterate this argument
from 
$\gamma_0 = 0$ to get, for some finite $k$,
\begin{equation}
\label{eq:Z-first-bound-near-zero}
f_Z (r) \leq C_k r^{\frac{2+\beta}{\alpha}}, \text{ for all } r \in [0,r_k]. \end{equation}
Using the bound~\eqref{eq:Z-first-bound-near-zero} now in~\eqref{eq:f-Z-iteration} shows that,
for all $r$ sufficiently small,
\begin{equation}
\label{eq:f-Z-second-bound-near-zero}  f_Z (r) \leq \frac{2}{2+\beta} c_1 (b + \eps ) r^{\frac{2+\beta}{\alpha}}  , \text{ where } c_1 = \int_{0}^\infty u^{-\frac{2+\beta}{\alpha}} f_Z ( u ) \ud u ,
\end{equation}
the $c_1$ being a finite positive constant, since~\eqref{eq:Z-first-bound-near-zero} shows that the integral does not blow up near zero.
The other direction is similar: from~\eqref{eq:f-Z-fixed-point} we have
\begin{align*}
  f_Z (r) & \geq \frac{2}{\alpha} (b -\eps)  \re^{-r} \int_0^r \ud y  y^{\frac{2+\beta-\alpha}{\alpha}} \int_{Ky}^\infty u^{-\frac{2+\beta}{\alpha}}  f_Z ( u ) \ud u \\
	& \geq \frac{2}{2+\beta} c_1 (b -\eps)  \re^{-r} r^{\frac{2+\beta}{\alpha}}
	- C \int_0^r \ud y  y^{\frac{2+\beta-\alpha}{\alpha}} \int_{0}^{Ky} u^{-\frac{2+\beta}{\alpha}}  f_Z ( u ) \ud u  .\end{align*}
With the upper bound from~\eqref{eq:f-Z-second-bound-near-zero} we get that the negative term here is $O (r^{1+\frac{2+\beta}{\alpha}} )$ as $r \to 0$.
Since $\eps >0$ was arbitrary, we  conclude that
\begin{equation}
\label{eq:f-Z-near-zero}
 f_Z (r) = ( c_{Z,0} +o(1))  r^{\frac{2+\beta}{\alpha}}, \text{ as } r \to 0 ,\end{equation}
where $c_{Z,0} := \frac{2b c_1}{2+\beta}$, with $c_1$ defined in~\eqref{eq:f-Z-second-bound-near-zero}.

Now we turn to the random variable $Q^{1/\alpha}$.
Note that the density $f_T$ of $T$ satisfies $f_T (t) \sim c_{T,\infty} \re^{-(2+\beta)t}$ as $t \to \infty$ where $c_{T,\infty} := \frac{b c_T}{2+\beta} \in (0,\infty)$.
It follows that 
$f_T (t) \leq C \re^{-(2+\beta)t}$ for some $C < \infty$ and all $t \in \RP$.
For $r \in \R$ we have from~\eqref{eq:fixed-point} that
\begin{align*} \Pr ( - \log Q > r\alpha ) & =  \Pr ( \alpha T - \log Z > r\alpha ) = \int_0^\infty f_T (t) F_Z ( \re^{(t - r)\alpha} ) \ud t .\end{align*}
Since $F_Z$ is continuously differentiable,   
and since $\Pr ( - \log Q > r\alpha ) = \Pr ( Q^{1/\alpha} \leq \re^{-r} )$,
we can
differentiate under the integral sign to see that $Q^{1/\alpha}$ has a density~$q$ which satisfies
\begin{align}
    \label{eq:f-q-formula}
 \re^{(\alpha-1)r } q( \re^{-r} ) & =   \alpha  \int_0^\infty \re^{\alpha t} f_T (t) f_Z ( \re^{\alpha (t-r)} ) \ud t  \nonumber\\
& = \re^{\alpha r} \int_{\re^{-\alpha r}}^\infty f_T \left( \frac{\log u}{\alpha} + r \right) f_Z (u) \ud u
 .
\end{align}
 Here $f_T$ is bounded and continuous, so the dominated convergence theorem shows that the second integral in~\eqref{eq:f-q-formula} is continuous over $r \in \R$, and hence $q(r)$ is continuous over $r \in (0,\infty)$.
We now use the first integral in~\eqref{eq:f-q-formula}
to derive the   asymptotics of $q$ near zero. 
By~\eqref{eq:f-Z-near-zero} there exists $C < \infty$ such that $f_Z ( r ) \leq C r^{\frac{2+\beta}{\alpha}}$
for all $r \in \RP$. 
 Thus
\[ 
\int_0^{r/2}   \re^{\alpha t} f_T (t)  f_Z ( \re^{\alpha (t - r)} ) \ud t 
\leq C \re^{- (2+\beta) r} \int_0^{r/2} \re^{\alpha t} \ud t
\leq C \re^{\alpha r /2} \re^{- (2+\beta) r} .\]
Similarly, for any $\eps >0$ and all $t > r/2$ with $r$ sufficiently large,
\begin{align*}
 \int_{r/2}^{\infty}  \re^{\alpha t} f_T (t) f_Z ( \re^{\alpha (t - r)} ) \ud t 
& \leq (c_{T,\infty} + \eps) \int_{r/2}^{\infty}  \re^{\alpha t} \re^{-(2+\beta) t} f_Z ( \re^{\alpha (t - r) } ) \ud t \\
& = (c_{T,\infty} + \eps)  \re^{\alpha r-(2+\beta) r} \int_{-r/2}^\infty \re^{\alpha s -(2+\beta)s} f_Z ( \re^{\alpha s} ) \ud s , \end{align*}
using the change of variable $s = t - r$.
The $r \to \infty$ limit of the $s$-integral here converges to $c_1/\alpha$, with $c_1$ the integral defined at~\eqref{eq:f-Z-second-bound-near-zero}. 
Thus we get, for all $r$ sufficiently large,
\[  \int_{r/2}^{\infty}  \re^{\alpha t}  f_T (t) f_Z ( \re^{\alpha (t - r) } ) \ud t  \leq \left( \frac{c_1 c_{T,\infty}}{\alpha} + \eps \right) 
 \re^{\alpha r -(2+\beta) r}  .\]
A similar argument in the other direction shows that, for all $r$ sufficiently large,
\[  \int_{r/2}^{\infty}  \re^{\alpha t}  f_T (t) f_Z ( \re^{\alpha (t - r)} ) \ud t
\geq \left( \frac{c_1 c_{T,\infty}}{\alpha} - \eps \right)  \re^{\alpha r -(2+\beta) r} .\]
It follows from~\eqref{eq:f-q-formula} and the above estimates that
\begin{equation}
\label{eq:F-Q-near-zero}
q(r)  = ( c_{q,0} + o(1) ) r^{1 + \beta} , \text{ as } r \to 0 ,
\end{equation}
where $c_{q,0} :=\frac{b c_1 c_T}{2+\beta}$, with $c_1$ as defined at~\eqref{eq:f-Z-second-bound-near-zero}.  

Next we turn to the upper tail estimates. In this case we will use Brennan \& Durrett's expression for the moment
generating function of~$Z$ and a Tauberian theorem. Brennan \& Durrett also give an expression for the moment generating
function of~$Q$, but monotonicity properties, helpful
for deducing density asymptotics via the Tauberian argument, are easier to demonstrate for~$Z$.
 Recalling that $Z$ has the same distribution as $Y + \xi$, for $Y, \xi$ independent and $\xi$ exponential with unit mean, 
we have
\[ \Pr ( Z \leq r ) = \int_0^r \re^{-s} \Pr ( Y \leq r -s ) \ud s = \re^{-r} \int_0^r \re^u \Pr ( Y \leq u ) \ud u .\]
Differentiation gives
$f_Z (r) = \Pr ( Y \leq r ) - \Pr ( Z \leq r)$, and so $f'_Z (r) = f_Y (r) - f_Z (r)$.
Thus
\[ \frac{\ud}{\ud r} \bigl( \re^r f_Z (r) \bigr) = \re^r \bigl( f'_Z(r) + f_Z (r) \bigr) = \re^r f_Y (r) \geq 0 .\]
Hence $\re^r f_Z(r)$ is non-decreasing; this is the helpful monotonicity property mentioned above.
Using Brennan \& Durrett's formula for the moments of $Z$~\cite[p.~114]{bd2},  we see
\begin{equation}
\label{eq:Z-mgf}
 m_Z (t) := \Exp ( \re^{tZ} ) =\sum_{k =0}^\infty t^k \prod_{j=1}^k \frac{1}{1-h(j \alpha)} ,\end{equation}
where 
\[ h (t ) := \int_0^\infty \re^{-tx} f_X (x) \ud x = 2 \int_0^\infty \re^{-(2+t)x} \phi (\re^{-x} ) \ud x.\]
Since $X$ is non-degenerate, $h(t) < 1$ for all $t >0$.
Moreover, since $\phi (\re^{-x} ) \sim \phi (x) \sim b x^\beta$ as $x \to 0$,
we may apply Laplace's method (see e.g.~\cite[pp.~55--58]{wong}) to obtain
\begin{equation}
\label{eq:h-asymptotics}
 h(t) \sim 2 b \Gamma (1+\beta) t^{-1-\beta}, \text{ as } t \to \infty .\end{equation}
 It follows from~\eqref{eq:h-asymptotics} that $m_Z(t) < \infty$ provided $| t | <1$; indeed, as we will see, the information we need is contained in the asymptotics of $m_Z(t)$ as $t \uparrow 1$.
Consider the Laplace transform $\tilde m_Z$ associated with $\re^r f_Z(r)$, namely
\[ \tilde m_Z (t) := \int_0^\infty \re^{-tx} \re^x f_Z (x) \ud x = m_Z (1-t ) ,\]
which is finite for $t \in (0,1)$. We will use a Tauberian theorem to relate the
$r \to \infty$ asymptotics of $\re^r f_Z (r)$ to the $t \to 0$ asymptotics
of $\tilde m_Z (t)$. From~\eqref{eq:Z-mgf}, we have
\begin{align}
\label{eq:Z-mgf-near-1}
 \tilde m_Z (t) & = \sum_{k =0}^\infty (1-t)^k 
\exp \sum_{j=1}^k \log \left( \frac{1}{1-h(j \alpha)} \right) .\end{align}
Here we have from~\eqref{eq:h-asymptotics} that, as $j \to \infty$,
\[  \log \left( \frac{1}{1-h(j \alpha)} \right) =  \log \left( 1 + \frac{h(j\alpha)}{1-h(j \alpha)} \right)
=  2 b \Gamma (1+\beta) (j \alpha)^{-1-\beta} + O (j^{-2-\beta}) .\]
It follows that, as $k \to \infty$,
\begin{equation}
\label{eq:Z-mgf-sum-asymptotics}
 \sum_{j=1}^k \log \left( \frac{1}{1-h(j \alpha)} \right) = \begin{cases}
\frac{2b}{\alpha} \log k + \log c_2 + o(1) & \text{if } \beta = 0, \\
\log c_2 + o(1)  & \text{if } \beta >0, \end{cases} \end{equation}
where  $c_2 \in (0,\infty)$ is a constant depending on $\alpha, \beta$, and $\phi$. 

If $\beta >0$, then~\eqref{eq:Z-mgf-near-1} and~\eqref{eq:Z-mgf-sum-asymptotics} show that $\tilde m_Z (t) = \sum_{k =0}^\infty (1-t)^k (c_2 + o(1))$,
where the $o(1)$ is as $k \to \infty$, and is uniform in $t >0$.
It is elementary to deduce that
\begin{equation}
\label{eq:mgf-Z-beta-non-zero}
 \tilde m_Z (t) \sim c_2/t , \text{ as } t \to 0 , \text{ if } \beta >0. \end{equation}
On the other hand, suppose that $\beta =0$. Then we have from~\eqref{eq:Z-mgf-near-1} and~\eqref{eq:Z-mgf-sum-asymptotics} that
$\tilde m_Z (t) = \sum_{k =0}^\infty (1-t)^k (c_2 + o(1) ) k^{2b/\alpha}$,
where the $o(1)$ is as $k \to \infty$, and is uniform in $t$.
It is a consequence of a standard Abelian theorem for power series that
$\sum_{k =0}^\infty (1-t)^k  k^{\rho}  \sim \Gamma(1+\rho) t^{-\rho-1}$ as $t \downarrow 0$.
Thus we deduce that
\begin{equation}
\label{eq:mgf-Z-beta-zero} \tilde m_Z (t) \sim \Gamma \left( 1 + \frac{2b}{\alpha} \right) c_2 t^{-\frac{2b}{\alpha} -1}, \text{ as } t \to 0 , \text{ if } \beta = 0. \end{equation}
Defining $a := \lim_{x \to 0} \phi (x)$, so $a=0$ if $\beta >0$ and $a = b$ if $\beta =0$,
we can combine the asymptotics~\eqref{eq:mgf-Z-beta-non-zero}
and~\eqref{eq:mgf-Z-beta-zero} into the single statement that, for some $\tilde c_{Z,0} \in (0,\infty)$,
\begin{equation}
\label{eq:mgf-Z-near-zero}
\tilde m_Z (t) \sim \tilde c_{Z,0} t^{-1-(2a/\alpha)}  , \text{ as } t \to 0 . \end{equation}
Together with the fact that $\re^r f_Z(r)$ is non-decreasing,
the asymptotics~\eqref{eq:mgf-Z-near-zero} 
 allow us to apply a monotone-density Tauberian theorem~(e.g.~\cite[p.~446]{feller2}) to deduce
\begin{equation}
\label{eq:Z-near-infinity}
 f_Z (r) = (c_{Z,\infty} +o(1)) r^{2a/\alpha} \re^{-r} , \text{ as } r \to \infty, \end{equation}
where $c_{Z,\infty} \in (0,\infty)$. 
Rewriting the first equality in~\eqref{eq:f-q-formula}, we have 
\begin{equation}
\label{eq:q-in-terms-of-f-Z}
 q (r) = \alpha r^{\alpha -1} \int_0^\infty \re^{\alpha t} f_T (t) f_Z ( \re^{\alpha t} r^\alpha ) \ud t .\end{equation}
From~\eqref{eq:q-in-terms-of-f-Z}, with the  change of variable $u = \re^{\alpha t}$ and
 the $f_Z$ asymptotics from~\eqref{eq:Z-near-infinity},  
\[ q(r) = (c_{Z,\infty} + o(1)) r^{2a + \alpha -1} \int_1 ^\infty f_T ( \alpha^{-1} \log u ) u^{2a/\alpha} \re^{- u r^\alpha} \ud u , \]
as $r \to \infty$.
Since $\lim_{x \to 0} f_T ( x ) = c_T/2$ (by symmetry of $\phi$), the asymptotics of the latter integral can be obtained by Laplace's method (e.g.~\cite[pp.~55--58]{wong}), which gives
\[ q (r) = ( c_{q,\infty} + o(1) ) r^{2a -1} \re^{-r^\alpha} , \text{ as } r \to \infty,\]
where $c_{q,\infty} := \frac{c_T c_{Z,\infty}}{2}$. This completes the proof.
\end{proof}

\begin{proof}[Proof of Theorem~\ref{thm:general-splitting-distribution}.]
Define the distribution function  
\begin{equation}
    \label{eq:G-def}
 G (x) :=\frac{1}{\rho} \int_0^x \frac{q (y)}{y}  \ud y, \text{ for } x \in \RP, \end{equation}
where 
$q$ 
is the density of the random variable $Q^{1/\alpha}$ and $\rho := \Exp ( Q^{-1/\alpha} ) \in (0,\infty)$. 
Brennan \& Durrett~\cite{bd2} consider a continuous-time
embedding of the interval-splitting process in which an interval
of length $\ell$ splits at rate $r(\ell) = \ell^\alpha$, and, when it splits, does so according to $\Phi$.
Starting at time $t=0$ with a single gap of length~$1$,
let $i_t$ denote the number of intervals at time $t \in \RP$ and let $e_{t,i}$, $i \in [i_t]$,
denote the lengths of those intervals, listed left to right. 
For $x \in [0,1]$, let $c_t(x) := \max \{ m \in \{ 0 , 1, \ldots, i_t \} : \sum_{i=1}^m e_{t,i} \leq x \}$,
the number of intervals wholly contained in $[0,x]$.  
 The result of~\cite[p.~113]{bd2} says that
\begin{equation}
    \label{eq:bd-limit}
    \lim_{t \to \infty} t^{-1/\alpha} i_t = \rho , \as ,
\text{ and } \lim_{t \to \infty} \frac{1}{i_t} \sum_{i \in [i_t]} \1 { t^{1/\alpha} e_{t,i} \leq x } = G(x) , \as ,\end{equation}
where $G$ is given by~\eqref{eq:G-def}, 
while Theorem~1.1 of~\cite[pp.~1027--8]{bd1} shows that
\begin{equation}
\label{eq:bd-N-limit}
\lim_{t \to \infty} \frac{c_t(x)}{i_t} = x, \as, \text{ for all } x \in [0,1] .
\end{equation}
Now we extend the  model to permit an arbitrary initial configuration $z \in \Delta_{n_0}$ at time $t = 0$.
Then the initial gaps $j \in [n_0+1]$ have lengths $u_1, \ldots, u_{n_0+1}$, say,
with $\sum_{j =1}^{n_0+1} u_j = 1$. The process
evolves independently on each gap. Let $i^z_{j,t}$ denote the number
of intervals at time $t \in \RP$ for the process restricted to initial gap $j$, and let $i^z_t = \sum_{j=1}^{n_0+1} i^z_{j,t}$
denote the total number of intervals. Also let $e^z_{j,t,i}$, $i \in [i^z_{j,t}]$, denote the interval
lengths for the process in interval~$j$.
The process in interval $j$ is a copy of the process on the single initial interval $[0,1]$, but with all lengths scaled by a factor of $u_j$,
which entails a time-scaling of $u_j^\alpha$;
in particular, $i^z_{j,t}$ has the same distribution as $i_{u^\alpha_j t}$,
and the collection $e^z_{j,t,i}$, $i \in [i^z_{j,t}]$,
 has the same distribution as $u_j e_{u^\alpha_j t, i}$, $i \in [i_{u^\alpha_j t}]$. 
Thus~\eqref{eq:bd-limit} implies that
\begin{equation}
    \label{eq:bd-limit-subgaps}
    \lim_{t \to \infty} t^{-1/\alpha} i^z_{j,t} = \rho u_j , \as ,
\text{ and } \lim_{t \to \infty} \frac{1}{i^z_{j,t}} \sum_{i \in [i^z_{j,t}]} \1 { t^{1/\alpha} e^z_{j,t,i} \leq x } = G(x) , \as \end{equation}
Also, if $e^z_{t,i}$, $i \in [i_t^z]$ are the (aggregated) interval lengths, listed left to right, then
\[  \frac{1}{i^z_{t}} \sum_{i \in [i^z_{t}]} \1 { t^{1/\alpha} e^z_{t,i} \leq x } 
=  \sum_{j \in [n_0+1]} \frac{i^z_{t,j}}{i^z_t} \frac{1}{i^z_{t,j}} \sum_{i \in [i^z_{j,t}]} \1 { t^{1/\alpha} e^z_{j,t,i} \leq x } . \]
Since $\sum_{j=1}^{n_0+1} i^z_{j,t} = i^z_t$ and $\sum_{j=1}^{n_0+1} u_j = 1$, 
we  conclude from~\eqref{eq:bd-limit-subgaps} that, for any $z \in \Delta_{n_0}$,
\begin{equation}
    \label{eq:bd-limit-2}
    \lim_{t \to \infty} t^{-1/\alpha} i^z_t = \rho , \as ,
\text{ and } \lim_{t \to \infty} \frac{1}{i^z_t} \sum_{i \in [i^z_t]} \1 { t^{1/\alpha} e^z_{t,i} \leq x } = G(x) , \as \end{equation}
If $\tau_0 =0$ and $\tau_n \in \RP$ denotes the time of the $n$th splitting event, 
then $\cS_{n_0}, \cS_{n_0+1}, \ldots$ is embedded at times $\tau_0, \tau_1, \ldots$
of the continuous-time process stated at $\cS_{n_0} = z$. 
Given $\cS_n$, $n \geq n_0$, let $\ell_{n,1},\ldots,\ell_{n,n+1}$ denote the
lengths of the gaps, so $i^z_{\tau_n} =n+1$ and $\ell_{n,i} = e^z_{\tau_n,i}$.
Translating~\eqref{eq:bd-limit-2} into discrete time thus gives
\begin{equation}
    \label{eq:bd-limit-3}
    \lim_{n \to \infty} \tau_n^{-1/\alpha} n = \rho , \as ,
\text{ and } \lim_{n \to \infty} \frac{1}{n+1} \sum_{i \in [n+1]} \1 { \tau_n^{1/\alpha} \ell_{n,i} \leq x } = G(x) , \as \end{equation}
Let $\cG_n$ denote the $\sigma$-algebra generated by $\cS_0, \ldots, \cS_n$
and $\tau_0,\ldots,\tau_n$.
Then, for $U_n$ a uniform random variable on $[n+1]$, independent of $\cG_n$,
set 
$\tl_n = (n+1) \ell_{n,U_n}$, so that
\[ \Pr ( \tl_n \leq x \mid \cG_n ) =  
  \frac{1}{n+1} \sum_{i \in [n+1]} \bbind \left\{  \tau_n^{1/\alpha} \ell_{n,i} \leq x \tau_n^{1/\alpha} (n+1)^{-1} \right\} .\]
Since, by~\eqref{eq:bd-limit-3}, $(n+1) \tau_n^{-1/\alpha} \to \rho$, a.s.,
for any $\eps >0$ and all $n$ sufficiently large,
\[  \Pr ( \tl_n \leq x \mid \cG_n ) \leq 
  \frac{1}{n+1} \sum_{i \in [n+1]} \bbind \left\{  \tau_n^{1/\alpha} \ell_{n,i} \leq x (\rho^{-1} +\eps ) \right\}, \]
  by monotonicity, so that, by~\eqref{eq:bd-limit-3}, 
$\limsup_{n \to \infty}  \Pr ( \tl_n \leq x \mid \cG_n ) \leq G( x(\rho^{-1} +\eps) )$, a.s. 
  By a similar argument in the other direction, and continuity of $G$ given at~\eqref{eq:G-def}, we get, a.s.,
  \begin{equation}
	\label{eq:g-lim}
	\lim_{n\to\infty} \Pr (\tl_n \leq x \mid \cG_n) = \frac{1}{\rho} \int_0^{x/\rho} \frac{q(y)}{y}  \ud y
  = \int_0^x g(z) \ud z , \text{ where } g(x) := \frac{q(x/\rho)}{\rho x} . \end{equation}
This establishes the a.s.~convergence result in~\eqref{eq:gap-distribution-limit}
  with $g(x)$ as displayed, and the $L^1$ convergence follows by the bounded
  convergence theorem. This proves~(ii). Moreover, Lemma~\ref{lem:fixed-point}
	shows that $g$ as defined in~\eqref{eq:g-lim}
	is continuous on $(0,\infty)$, and satisfies the asymptotics  for $g$ given in part~(iii) of the theorem,
	with $c_{g,0} = c_{q,0} \rho^{-2-\beta}$, $c_{g,\infty} = c_{q,\infty} \rho^{-2a}$, and $\theta = \rho^{-\alpha}$.
	Thus~(iii) is also proved.
	
	For part~(i), fix $x \in [0,1]$ and let $j_x = \min \{ j \in [n_0+1] : \sum_{i=1}^j u_i \geq x \}$,
so that $j_x$ is the index of the initial gap that contains~$x$.
Let $x' = \sum_{i=1}^{j_x-1} u_i$, so $0 \leq x' \leq x$.
 Then
\begin{equation}
    \label{eq:gap-aggregation}
 c^z_t(x) := \max \biggl\{ m \in \{ 0, 1, \ldots, i_t^z \} : \sum_{i=1}^m e^z_{t,i} \leq x \biggr\}
 = c^z_{j_x,t} (x) + \sum_{j < j_x} i^z_{j,t}  ,\end{equation}
where $c^z_{j_x,t} (x)$ means the number of intervals at time $t$ contained in initial gap $j_x$ (whose left endpoint is at $x'$) that fall wholly in $[0,x]$.
By scaling, $c^z_{j,t} (x)$, $i^z_{j,t}$ have the same distribution as $c_{ u_j^\alpha t}  ( \frac{x-x'}{u_{j}}  )$, $i_{u_{j}^\alpha t}$, and so we have from~\eqref{eq:bd-N-limit}
and~\eqref{eq:bd-limit-subgaps}
that 
\[   i^z_{j,t} \sim \rho t^{1/\alpha} u_{j}, \text{ for all $j$, and } c^z_{j_x,t} (x) \sim i^z_{j_x,t} \left( \frac{x-x'}{u_{j_x}} \right)  \sim \rho t^{1/\alpha} ( x - x') .\]
Together with~\eqref{eq:gap-aggregation}, this implies that
$c^z_t(x) \sim \rho t^{1/\alpha} x$. It follows from~\eqref{eq:bd-limit-2} that
$c^z_t(x)/i^z_t \to x$, a.s., and thus we get~(i) after translating the result into discrete time.
\end{proof}

\begin{remark}
\label{rem:explicit-f-Z-q}
In the special case where $\phi (x) \equiv 1$ (uniform splitting), 
 the explicit solutions to~\eqref{eq:Z-mgf}, \eqref{eq:f-Z-fixed-point},
and~\eqref{eq:q-in-terms-of-f-Z} are $m_Z (t) = (1-t)^{-\frac{\alpha+2}{\alpha}}$ for $|t| < 1$, and
\[ f_Z (r ) = \frac{r^{2/\alpha}}{\Gamma (1 + \frac{2}{\alpha} )} \re^{-r} , ~ \text{and} ~
q( r) = \frac{2r}{\Gamma (1 + \frac{2}{\alpha} )} \re^{-r^\alpha}, ~ r \in \RP , 
\]
so that $\rho = \Gamma (1/\alpha) / \Gamma (2/\alpha)$ (cf.~\cite[p.~113]{bd2}),
which with~\eqref{eq:g-lim} justifies Remark~\ref{rem:uniform-g}.
\end{remark}

\section*{Acknowledgements}

The authors
are grateful to Michael Grinfeld for
introducing them to deposition and nucleation
models, and to Michael Grinfeld and Paul Mulheran
for stimulating discussions on this topic over several years. The authors
also thank an anonymous referee for a careful reading and
 helpful remarks.





\end{document}